\documentclass[12pt,leqno,amscd,amssymb,verbatim, url]{amsart}
\input xy
\xyoption{all}

\newtheorem{thm}{Theorem}[section]
\usepackage{amsfonts,latexsym}
\usepackage{amssymb,amsmath,amsthm,amscd,latexsym,amsfonts,bezier}
\usepackage{bbold} %added by Len sometime late 2015
\usepackage{pgf,pgfarrows,pgfnodes,pgfautomata,pgfheaps,pgfshade}
\usepackage{amsmath,amssymb}
\usepackage[latin1]{inputenc}
\usepackage{colortbl}
\usepackage[english]{babel}
\usepackage[margin=15mm]{geometry}
\usepackage{tikz}
\usetikzlibrary{matrix,arrows}

\newtheorem{lem}[thm]{Lemma}
\newtheorem{cor}[thm]{Corollary}

\newtheorem{prop}[thm]{Proposition}
\theoremstyle{definition}

\newtheorem{defn}[thm]{Definition}

\newtheorem{rem}[thm]{Remark}
\newtheorem{rems}[thm]{Remarks}
\newtheorem{subsec}[thm]{\!}

\numberwithin{equation}{thm}

\newtheorem{theorem}{Theorem}
\usepackage{wrapfig}
\newcommand{\C}{\mathbb{C}}
\newcommand{\Z}{\mathbb{Z}}

\newcommand{\Ind}{\text{Ind}\,}
\newcommand{\Lg}{\mathfrak{g}}

           \newcommand{\Ext}{\text{\rm{Ext}}\,}

\newcommand{\Hom}{\mbox{Hom}\,}

\newcommand{\ch}{\mbox{char}}

\newcommand{\IndBU}{\mbox{Ind}_{\mathbb{B}}^{\mathbb{U}}\,}
\newcommand{\IndBG}{\text{Ind}_B^G\,}
\newcommand{\RIndBG}{\mbox{RInd}_B^G\,}
\newcommand{\Tlammu}{\mbox{T}_\lambda^\mu\,}
\newcommand{\Tmulam}{\mbox{T}_\mu^\lambda\,}
\newcommand{\RIndBU}{\mbox{RInd}_{\mathbb{B}}^{\mathbb{U}}\,}
\newcommand{\IndpB}{\mbox{Ind}_{\mathbb{p}}^{\mathbb{B}}\,}

\newcommand{\g}{\mathfrak{g}}
\newcommand{\lh}{\mathfrak{h}}

\newcommand{\RInd}{\mbox{RInd}}

\newcommand{\RHom}{\mbox{RHom}\,}

\newcommand{\ExtB}{\mbox{Ext}_{\mathbb{B}}}
\newcommand{\ExtU}{\mbox{Ext}_{\mathbb{U}}}

\newcommand{\vupperone}{\mbox{V}^{[1]}}
\newcommand{\lr}{\longrightarrow}
\newcommand{\bB}{\mathbb{B}}
\newcommand{\bU}{\mathbb{U}}
\newcommand{\Itmu}{I_{\mu}^{[1]}}

\newcommand{\kB}{\mbox{k}_{\mathbb{B}}\,}

\newcommand{\blist}{\begin{list}{\rom{(\roman{enumi})}}{\setlength{\leftmarg
in}{0em} \setlength{\itemindent}{7ex}
\setlength{\labelsep}{2ex}\setlength{\listparindent}{\parindent}
\usecounter{enumi}}}
\newcommand{\elist}{\end{list}}

\newcommand{\rarrow}{\rightarrow}

\newcommand{\calC}{{\mathcal{C}}}

\newcommand{\perv}{\mathcal{P}\!{\it erv}}
\begin{document}

\begin{abstract}
A key result in a 2004 paper by S. Arkhipov, R. Bezrukavnikov, and
V. Ginzburg \cite{ABG} compares the bounded derived category
$D^{b}block(\mathbb{U})$ of modules for the principal block of a
Lusztig quantum enveloping algebra $\mathbb{U}$ at an $\ell$th root
of unity with a subcategory $D_{triv}(\mathbb{B})$ of the
%bounded
derived category of integrable type 1 modules for a Borel part $\mathbb{B}%
\subset\mathbb{U}.$ Specifically, according to this ``Induction
Theorem'' \cite[Theorem 3.5.5]{ABG} the right derived functor of
induction $\mbox{Ind}_{\mathbb{B}}^{\mathbb{U}}\,$ yields an
equivalence of categories
$\mbox{RInd}_{\mathbb{B}}^{\mathbb{U}}\,:D_{triv}(\mathbb{B})\overset{\sim
}{\rightarrow} D^{b}block(\mathbb{U})$ (under appropriate hypotheses
on $\ell$).
%(see *** below for full details).
The authors of \cite{ABG} suggest a similar result holds for
algebraic groups in positive characteristic $p$, and  this paper
provides a statement with proof for such a modular induction
theorem. Our argument uses the philosophy of \cite{ABG} as well as
new ingredients. A secondary goal of this paper has been to put the
original characteristic zero quantum result on firmer ground, and we
provide arguments as needed to give a complete proof of that result
also. Finally, using the modular result, we have been able in
\cite{HKS} to introduce truncation functors, associated to finite
weight posets, which effectively commute with the modular induction
equivalence, assuming $p>2h-2$, with $h$ the Coxeter number.
%Don't forget to include relevant/analogous restrictions on $p.$
%Is the latter inequality correct? Why not $\geq?$ Check.
This enables interpreting the equivalence at the level of derived
categories of modules for suitable finite dimensional
quasi-hereditary algebras. We expect similar results to hold in the
original quantum setting, assuming $\ell>2h-2$.
\end{abstract}

%\begin{frontmatter}

\title{Remarks on the ABG Induction Theorem}
%\\DRAFT as of February 28, 2016}

\author{Terrell L. Hodge}\footnote{Corresponding author: Terrell L. Hodge}
\address{Department of Mathematics, Western Michigan University, Kalamazaoo, MI 49008}
\email{terrell.hodge@wmich.edu}

\author{Paramasamy Karuppuchamy}
\address{Department of Mathematics, University of Toledo, Toledo, OH  43606}
\email{paramasamy.karuppuchamy@utoledo.edu}

\author{Leonard L. Scott}
\address{Department of Mathematics, University of Virginia, Charlottesville, VA  22903}
\email{lls2l@virginia.edu}

%\date{June 29, 2010}
\date{\today}

\subjclass[2000]{Primary 17B55, 20G; Secondary 17B50}%Len suggested using these; are the same ones as in newest Graded QHA II paper with BP
\thanks{Research supported in part by NSF grant DMS-1001900 and Simons Foundation Collaborative
Research award $\#$359363.}
% added Len 6-1815 and experimented with placement in file. This thanks will probably need to be removed from the last section, now.

\maketitle
%moved to after abstract Len 6-18-15 on suggestion of pdfLatex compiler

%\begin{frontmatter}

%\footnote{Corresponding author}

%\date{July 12, 2010}

%\end{frontmatter}

%\bigskip
% commented out Len 18-06-15
\bigskip

% back in 02-01-16

\section{Introduction}

%\subsection{Notation}

\ \ \ \ If $G$ is a semisimple algebraic group and $B$ a Borel
subgroup, it is well known that the category of rational $G$-modules
fully embeds via the restriction functor into the category of
rational $B$-modules. Explicitly describing the objects in the image
of restriction is a difficult problem, unsolved in general. However,
as we will see here, it is possible to make progress at the derived
category level. Our starting point is a result \cite[Theorem
3.5.5]{ABG} by S. Arkhipov, R. Bezrukavnikov, and V. Ginzburg in the
world of quantum groups. The result establishes a natural
equivalence between the bounded derived category of modules for the
principal block of a Lusztig quantum enveloping  algebra at a root
of unity with an {\em explicit} subcategory of the bounded derived
category of integrable modules for a Borel part of this quantum
algebra. We will refer to this result as ``the induction theorem."
We begin this paper with its explicit statement.

Suppose $\mathbb{U}$ is a Lusztig quantum algebra, associated to a
root datum
$\mathfrak{R}=(\Pi,\mathbb{X},\Pi^{\vee},\mathbb{X}^{\vee}),$ and
specialized to a characteristic $0$ field $K$ with
$\ell^{\mbox{th}}$ root of unity $q\in K,$ as defined in section
\ref{sec:back}. In particular, we assume $q$ is defined by $q^{\ell}
= 1$ with $\ell$ odd, and not divisible by $3$ if the root system
corresponding to $\mathfrak{R}$ has a component of type $G_2.$ We
denote the root system in general by $R$. Moreover, we  assume $\ell
> h$, where $h$  is the Coxeter number of $R$, unless otherwise
noted.
%Move these assumptions/repeat these assumptions to/in Section \ref{sec:back}.
Suppose $\mathbb{B}=\mathbb{U}^{-}\otimes
_{K}\mathbb{U}^{0}\subset\mathbb{U}$ is a `Borel part' of
$\mathbb{U}$ arising from a triangular decomposition of $\mathbb{U}$
as in Section \ref{sec:back}. \pagebreak Denote by
$D^{b}block(\mathbb{U})$ the bounded derived category of the abelian
category of type
 1 integrable modules in the ``principal
block" of $\mathbb{U}$.\footnote{More precisely, the principal block
of $\mathbb{U}$ is the full subcategory of finite-dimensional
integrable type 1 $\mathbb{U}$-modules whose composition factors are
all ``linked" to the trivial module. Equivalently, the highest
weights of these composition factors all belong to the ``dot" orbit
of 0 under the affine Weyl group. A precise definition of the ``dot"
action is given in Section 2.} For $D^b(\mathbb{B})$ the usual
bounded derived category of the module
%be more precise
category for $\mathbb{B},$ let $D_{triv}(\mathbb{B})$ be the full
triangulated subcategory of
 $D^b(\mathbb{B})$  whose objects are
complexes representable by
\[
M=\cdots\rightarrow M_{i-1}\rightarrow M_{i}\rightarrow
M_{i+1}\rightarrow
\cdots\quad i\in\mathbb{Z}%
\]
so that for all $i\in\mathbb{Z},$

\begin{itemize}
\item[(i)] $M_{i}$ is an integrable $\mathbb{B}-$module;
%Add the defn somewhere

\item[(ii)] $M_{i}$ has a grading $M_{i} = \oplus_{\nu\in \mathbb{Y}}M_{i}(\nu)$ by the
root lattice $\mathbb{Y}$ of the root datum $\mathfrak{R};$

\item[(iii)] for any $m\in M_{i}(\nu)$ and $u\in\mathbb{U}^{0},$ $um =
\nu(u)\cdot m;$

\item[(iv)] the total cohomology module $H^{\bullet}(M) = \oplus
_{i\in\mathbb{Z}}H^{i}(M)$ has a finite composition series, all of
whose successive quotients are of the form
$K_{\mathbb{B}}(\ell\lambda),\lambda\in \mathbb{Y}.$
%Problem -- notationally, using $K$ vs. $k;$ also will want notation $K$ for
%elts of $\mathbb{U},$ ugh...decide how to switch this around --
\end{itemize}
Here $K_{\mathbb{B}}(\ell\lambda)$ denotes a  $1$-dimensional
$\mathbb{B}$-module associated to $\ell\lambda$. Further details on
notation may be found in Section \ref{sec:back}.
%Add comments or delete last line.

%Our main goal is to provide a proof of the following key result  in
%\cite{ABG}.
\begin{theorem}
\label{T:induct} [\emph{Induction Theorem}, Theorem 3.5.5 \cite{ABG}] For an
appropriately  defined induction functor $\mbox{Ind}_{\mathbb{B}}^{\mathbb{U}%
}\,$, its right derived functor $R\text{Ind}\,_{\mathbb{B}%
}^{\mathbb{U}}$ yields an \textbf{equivalence} of triangulated  categories
\[
D_{\mathrm{triv}}(\mathbb{B}) \overset{\sim}{\rightarrow} D^{b}\it{block}%
(\mathbb{U}).
\]

\end{theorem}

%Recall how to TEX arrow w/ sim sign above it
A precise definition for $\mbox{Ind}_{\mathbb{B}}^{\mathbb{U}}\,$
appears in Section \ref{sec:back}. It is an analog of induction
(right adjoint to restriction) in the theory of representations of
algebraic groups, and has similar properties. To ``induce" a module,
one applies induction in the sense of associative rings and
algebras, then passes to the largest type 1 integrable submodule.

\bigskip

The present paper contains,  as a secondary feature, a complete
proof of the above result, along the lines of \cite{ABG}, though
with some variations and a number of corrections. We are grateful to
Pramod Achar for alerting us to possible issues (first observed by
his collaborator, Simon Riche) in the proof of \cite[Lemma
4.1.1(ii)]{ABG}. The argument we eventually found (the proof of our
Lemma \ref{wallcrossingfunctors}(ii)) is quite substantial, spanning
two appendices and improving a theorem of Rickard \cite{R94}. Other
corrections we make are more minor, often rooted in inadequacies in
the quantum literature. The ``variations" mentioned often occur from
our desire to give a proof that ``carries over" to the
characteristic $p$ algebraic groups case. Indeed, the latter has
been the central aim of our work here.

    The existence of a modular analog of the induction theorem
 was suggested by the assertion \cite[p. 616]{ABG}:
\textquotedblleft An analogue of Theorem 3.5.5 holds also for the
principal block of complex representations of the algebraic group
$G(F)$ over an algebraically closed field of characteristic $p>0.$
Our proof of the theorem applies to the latter case as
well.\textquotedblright\ Replacing the term ``complex
representations" in the quote above with \textquotedblleft rational
representations\textquotedblright\ (likely intended) yields the
statement below, which this paper confirms is indeed a theorem.
However, while the philosophy and some ingredients of the proof we
present may be found in the \cite{ABG} treatment of the quantum
case, additional critical ingredients are also required. See, for
example, Corollary \ref{cor:Len2}(1) and Lemma \ref{l:YImuExts}.

 To
set the notation, $\it{block}(G)$ is the principal block of
finite-dimensional rational $G$-modules, and $D_{triv}(B)$ is
defined analogously to $D_{triv}(\mathbb{B})$ in the quantum case.
That is, rational $B$-modules replace (Type 1) integrable
$\mathbb{B}$ modules, the distribution algebra of a maximal split
torus $T \subseteq B$ is used for $\mathbb{U}^0$, and
$k_B(p\lambda)$, with $k$ as below, replaces
$K_{\mathbb{B}}(\ell\lambda)$ above (in the definition of
$D_{\mathrm{triv}}(\mathbb{B})$, which becomes
$D_{\mathrm{triv}}(B)$). As before $h$ denotes the Coxeter number of
the underlying root system, now regarded as associated to $G$.

\begin{theorem}
\label{T:Ginduct} Let $G$ be a semisimple algebraic group
over an algebraically closed field $k$ of positive
characteristic $p > h$. Let $B$ be a Borel subgroup of $G$. Then the
 functor $R\text{Ind}\,_{B}^{G}$ yields an \textbf{equivalence} of
triangulated categories
\[
D_{\mathrm{triv}}(B)\rightarrow D^{b}\it{block}(G).
\]
\end{theorem}

%Do we want Borel subgroup $B$, or Borel part of $\Dist(G)$?

Generally, we use the ``quantum case" to refer to the context of
Theorem 1, and the ``algebraic groups case" (or ``positive
characteristic case," or ``modular case) when referring  to the
context of Theorem 2. Of course, some discussions in a given "case"
do not require the full hypotheses of these theorems. (We sometimes
keep track of such situations.)

Theorem 2 is a starting point for yet another result, proved in
\cite{HKS}. It shows, for $p > 2h-2$, that $\mbox{RInd}_B^G$ in
Theorem 2 induces an equivalence between certain natural full
triangulated subcategories
$$
D_{triv}(\mbox{Dist}(B)_{\Lambda_{m}}) \rarrow
D^b(\it{block}(G)_{\Gamma_{m}}),
$$
depending on $p$ and indexed by an integer $m > 0$. $\Lambda_m$ is a
finite subposet in a variation of van der Kallen's ``excellent
order" on weights \cite{vdK1}, and $\Gamma_m$ is a finite subposet
of dominant weights in the usual dominance order. The arguments
$\mbox{Dist}(B)_{\Lambda_{m}}$ and $\it{block}(G)_{\Gamma_{m}}$ of
the constructions in the display refer to finite-dimensional
quasi-hereditary algebra quotients of the distribution algebra
$\mbox{Dist}(B)$ and $\mbox{Dist}(G)$, respectively, the latter
associated to $\mbox{G}$--a mild abuse of notation.

For further details, see \cite{HKS}. Collectively, these more ``finite" equivalences
can be used to reconstruct the full equivalence given by $\mbox{RInd}_B^G$ in Theorem 2, thereby deepening our understanding of it.

\medskip
This paper is organized as follows. Section \ref{sec:back} collects
notation and some needed background material. Section
\ref{sec:inducproof} proves Theorems \ref{T:induct} and
\ref{T:Ginduct}. The statements above these theorems contain the
start of a dictionary for going back and forth between the
characteristic 0 quantum root of unity case and the positive
characteristic algebraic group case. Indeed, there is nothing to
stop us from using the same names as in
Theorem \ref{T:induct} for parallel objects in
Theorem \ref{T:Ginduct}, putting $\mathbb{U}=\mbox{Dist}(G)$,
$\mathbb{B}= \mbox{Dist}(B)$, $\it{block}(\mathbb{U}) = \it{block}(G), D_{triv}(\mathbb{B})= D_{triv}(B)$,
and even writing $\mbox{RInd}_{\mathbb{B}}^{\mathbb{U}}$ for $\mbox{RInd}_B^G$. We can then restate\\

 %$\begin{theorem}
 {\sc Theorem 2.1} (Cosmetic variation on Theorem 2) {\it Let $\mathbb{U}$ be the distribution algebra of a
 semisimple algebraic group over an algebraically closed field of prime characteristic $\ell = p > h$. Let
 $\mathbb{B}$ be the distribution algebra of a Borel subgroup. Then the induction functor
 $\rm{RInd}_{\mathbb{B}}^{\mathbb{U}}$ induces an equivalence}
 $$D_{triv}(\mathbb{B}) \rightarrow D^b \mbox{ block}(\mathbb{U}).$$
 %\end{theorem}
\bigskip
 In Section 3 we give a simultaneous proof of both Theorem \ref{T:induct} and the above version of
 Theorem \ref{T:Ginduct}.

 Some of the rationale for the overall approach is discussed in subsection \ref{subsec: summary}.
 Sections 4, 5, and 6
 present
 three appendices, labeled
 A, B, C, respectively.  The first two are used to prove Lemma \ref{wallcrossingfunctors}(ii), which
 restates
 \cite[Lem. 4.1.1(ii)]{ABG}, asserting that it holds in both the quantum and modular algebraic groups cases.
 The modular case, at least, is of
 independent interest
 of categorifying a theorem \cite[Thm. 2.1]{R94} of Rickard in the regular weight case, and the quantum case of
 the lemma may be
 viewed as giving an analogous quantum result. Also, Appendix C, independent of the
 rest of this paper, corrects the statement and proof of
\cite[Lem. 9.10.5]{ABG} as a service to the reader. Appendix C was
previously labeled and quoted as ``Appendix," in previous versions
of this paper.  Finally, a few acknowledgements and thanks are
collected in the final section.

    Theorem \ref{T:Ginduct} was first announced in \cite{HKS},
though the proof underwent several corrections after that, the last
in August, 2015, when a proof of Lemma
\ref{wallcrossingfunctors}(ii) was written down. This was done in
the modular case, and completed our proof of the modular induction
theorem. The proof actually also works in the quantum case, thus
proving \cite[Lem. 4.1.1(ii)]{ABG}, though we found it necessary to
work through some foundational issues regarding quantum induction
(Remarks \ref{r:Kempf}(d),(e)). We also found it necessary to fill
in other details in the quantum literature to complete our
simultaneous treatment of the quantum and modular induction
theorems. Another proof of the modular induction theorem, as part of
a larger geometric program, has recently been posted by Achar and
Riche \cite{AR}.
\bigskip

\section{Background} \label{sec:back} Generally, we follow Lusztig
\cite{L5}, \cite{L3} for basic material on quantum enveloping
algebras, and Andersen's paper \cite{A} for many additional results
on their representation theory, especially results on induced
representations that parallel those found in Jantzen \cite{J} in the
case of semisimple algebraic groups. These results on induced
representations have their origin in an earlier paper of
Andersen-Polo-Wen \cite{APW}, as supplemented by \cite{AW}. For the
study of (characteristic zero) quantum groups at $\ell^{th}$ roots
of unity with $\ell$ a prime power, the \cite{APW} paper is
generally sufficient, while the context of \cite{AW} allows all
values $\ell$ (orders of roots of unity) used in (the main results
of) this paper. (It does restrict $\ell$ to be odd, and not
divisible by 3 in case the root system has a component of type G2.)
The context of \cite{A} is even more general, though it references
an argument from \cite{AW}, and there are a number of references of
convenience (which could be avoided) to arguments in \cite{APW}.

We are interested in the semisimple algebraic groups case as much or
more so than in the quantum case, but focus now on giving  notation
below as befits the quantum case, where there is much less
uniformity in the literature than in the algebraic groups case. All
the notation and results have analogs in \cite{J}, however.   In
later parts of this paper, excluding the appendices, we will try to
treat both the algebraic groups and quantum cases simultaneously and
with the same notation.  Some of our quantum group notation has been
chosen to maintain consistency with these later
discussions.\footnote{In some cases, our notation differs from
\cite{ABG}. In particular we use a ``simply connected" set-up, which
allows modules with weights in $\mathbb{X}$ (defined below), rather
than the ``adjoint" set-up of \cite{ABG}, which restricts attention
to modules with weights in $\mathbb{Y}$ below (the root lattice).
This does not affect the triangulated categories (up to natural
equivalences) entering into the statements of Theorems 1, 2 and 2.1.
Also, we always induce from Borel subgroups associated to negative
roots, and correspondingly use a different (but more standard)
affine Weyl group ``dot" action, defined later in subsection
\ref{subsubsec:dot}.} Some important background on induction and
cohomology is given in subsection \ref{s:induction}, modifying and
completing a number of references given by \cite{ABG} to the
literature on quantum group representations. Many of the results we
discuss in that subsection are well-known in the algebraic groups
case, and we generally do not track their analogs there in detail.
Starting with subsection \ref{s:derived} and continuing in the rest
of Section 2 and all of Section 3, we use the ``uniform" notation
for both the quantum and positive characteristic cases, though some
differentiation of the two cases is sometimes required for proofs.
Appendices A and B, used to prove Lemma \ref{wallcrossingfunctors},
are given in algebraic groups notation, with the quantum case
treated in remarks. A compact quantum group reference written in the
spirit of comparing general results in the quantum and algebraic
groups cases may be found \cite[Appendix H]{J} in summary form.

\subsection{Quantum Enveloping Algebras and Algebraic Groups}
%include Triangular Decompositions
%[MORE MATERIAL TO BE ADDED HERE; ALSO, NEED TO CLEAN UP/IMPROVE THE LANGUAGE,DROP THE PRIME NOTATION, ETC.]
\subsubsection{Root Datum} \label{subsub:root datum}
Assume $\mathfrak{g}_{\C}$ is a complex semisimple Lie algebra of
rank $n,$ with Cartan matrix $C = (c_{ij})_{1\leq i,j\leq n},$ and
Killing form $\kappa:\g_{\mathbb{C}}\times \g_{\mathbb{C}} \to
\mathbb{C}.$  Then from a choice of Cartan subalgebra $\lh_{\C}
\subset \g_{\C}$ one obtains a root-datum realization $\mathfrak{R}
= (\Pi, \mathbb{X}, \Pi^{\vee},{\mathbb{X}}^{\vee})$ of $C$ from the
following data.
\begin{itemize}
\item $R\subset \lh_{\C}^*$ denotes the set of roots arising from the
Cartan decomposition $\g_{\C} = \lh_{\C} \oplus \bigoplus_{\alpha\in
R} \g_{\C,\alpha}$ into $\lh_{\C}^*$-weight spaces under the
restriction of the adjoint action of $\g_{\C}$ to $\lh_{\C},$ with
corresponding elements $t_{\alpha}\in \lh_{\C},$ $t_{\alpha}
\leftrightarrow \alpha \in R$ arising from the identification of
$\lh_{\C}$ with $\lh_{\C}^*$ obtained from the nondegeneracy of the
Killing form by setting, for any $\phi\in \lh_{\C}^* ,$ $t_{\phi}\in
\lh_{\C}$ to be the unique element such that  $\phi (h) =
\kappa(t_{\phi},h)$ ~for all $h\in \lh_{\C}.$
\item Take as the set of coroots
$R^{\vee}
:=\{\alpha^{\vee}:=\frac{2\alpha}{(\alpha,\alpha)}\,|\,\alpha\in
R\}.$ For
$h_{\alpha}:=\frac{2t_{\alpha}}{\kappa(t_{\alpha},t_{\alpha})},$
there is a correspondence $h_{\alpha}\leftrightarrow \alpha^{\vee}$
under the identification of $\lh_{\C}$ with $\lh_{\C}^*.$
\item Take $E =
E_{\mathbb{Q}}\otimes_{\mathbb{Q}}\mathbb{R},$ for $E_{\mathbb{Q}}$
the $\mathbb{Q}-$span of the roots $R$ in $\lh_{\C}^*.$ The rational
space $E_{\mathbb{Q}}$ has a nondegenerate bilinear form obtained
from restriction of $(\lambda,\mu) = (t_{\lambda},t_{\mu})$ on
$\lh_{\C}^*;$ this extends uniquely to a positive definite form
$(-,-)$ on $E,$ making $E$ into an $n$-dimensional Euclidean space.
Both $R$ and $R^{\vee}$ are root systems in $E;$ in particular they
both span $E.$  Observe that, for
$<\zeta,\eta>:=\frac{2(\zeta,\eta)}{(\eta,\eta)}\,\forall
\zeta,\eta\in E,$ one has  $<\beta,\alpha> = (\beta,\alpha^{\vee})$
for $\alpha,\beta\in R.$
\item The Weyl group $W$ is the subgroup  of
$GL(E)$ generated by reflections $s_{\alpha}(v) = v -
(v,\alpha^{\vee})\alpha, \alpha\in R.$
\item $\Pi  = \{\alpha_1,\dots, \alpha_n\}$ to be a set of simple roots (i.e., basis for $E$ such
that any $\alpha \in R$ satisfies $\alpha = \sum m_i\alpha_i\in
\oplus_{1\leq i\leq n}\mathbb{Z}\alpha_i$ with all $m_i\geq 0$ or
all $m_i\leq 0$). Then $W$ is generated by the $s_{i}:=s_{\alpha_i},
1\leq i\leq n.$
\item $\Pi^{\vee}:=\{\alpha_i^{\vee}\,|\,1\leq i\leq n\}$ to be the corresponding set of
simple coroots %(i.e., a $\mathbb{C}$-vector space basis for $\lh$),
\item ${\mathbb{X}}^{\vee} = \oplus_{i = 1}^{n}\Z\alpha^{\vee}_i$
coweight lattice
\item ${\mathbb{X}} = \{\lambda\in E\,|\, <\lambda,\alpha> =
(\lambda,\alpha^{\vee})\in \Z\quad \forall \alpha \in R\}=
\{\lambda\in E\,|\, (\lambda,\alpha_i^{\vee})\in \Z\quad \forall
\alpha_i \in \Pi\} \cong \Hom_{\Z}({\mathbb{X}}^{\vee},\Z)$ weight
lattice
\item $\mathbb{Y}$ is the subgroup of ${\mathbb{X}}$ generated by $R;$ $\mathbb{Y} = \bigoplus_{i=1}^n\Z\alpha_i$ root lattice
\item ${\mathbb{X}}^+:=\{\lambda\in {\mathbb{X}}\,|\, (\lambda, \alpha_i^{\vee})\geq 0 \quad \forall 1\leq i \leq n\}$ dominant
weights; for $\varpi_1,\dots,\varpi_n$ the dual basis defined by
$(\varpi_i,\alpha_j^{\vee}) = \delta_{i,j}$, we get ${\mathbb{X}} =
\bigoplus_{i = 1}^n\Z\varpi_i.$
\item In our notation the Cartan matrix $C = (c_{i,j})$ is given by $c_{i,j} =
(\alpha_j,\alpha_i^{\vee}), 1\leq i,j\leq n.$
\end{itemize}
Furthermore, for any $\alpha \in R,$ set $d_{\alpha} =
\frac{(\alpha,\alpha)}{2}, $ so then $d_{\alpha} \in \{1, 2, 3\}$
and $(\lambda,\alpha) = d_\alpha(\lambda,\alpha^{\vee}) =
d_{\alpha}<\lambda,\alpha>\in \mathbb{Z}$. Writing $d_{i}$ for
$d_{\alpha_i},$ and $D = diag(d_1,\dots, d_n),$ one has that $DC =
(d_{i}c_{i,j})$ is symmetric.

\subsubsection{Quantum Enveloping Algebra}
\label{subsub:QEAs}
  Our description of quantum enveloping algebras here follows
  Lusztig \cite{L3}, with similar notation, especially for
  generators. There are some differences in the notational names of algebras
  and subalgebras, and the labeling of relations. We make no
  distinction in the terms ``quantum enveloping algebra," ``quantum
  algebra," and ``quantum group."

 Take $v$ to be an indeterminate, and consider
the following expressions in the ring $\mathbb{Q}(v):$
\begin{eqnarray}
[n]_d &:= &\frac{v^{nd} - v^{-nd}}{v^d-v^{-d}}, \mbox{ for } d,n\in
\mathbb{N};
 \mbox{ when  }
d=1,
\mbox{ have }[n]:= [n]_1=\frac{v^n - v^{-n}}{v-v^{-1}}\notag\\
\end{eqnarray}
\begin{eqnarray}
 [n]_d!&:=&\Pi_{s=1}^{n}\frac{v^{d\cdot s} - v^{-d\cdot s}}{v^d -
v^{-d}} = \Pi_{s=1}^n [s]_d, \mbox{ for } n,d\in \mathbb{N}\notag\\
\left [\begin{smallmatrix}n\\t\end{smallmatrix}\right]_d&:=
&\Pi_{s=1}^{t}\frac{v^{d(n-s+1)} - v^{-d(n-s+1)}}{v^{ds} - v^{-ds}}
\mbox{ for }n\in \mathbb{Z}, d,t\in \mathbb{N};
 \mbox{ when  } d=1,
\mbox{ we set } [\begin{smallmatrix}n\\t\end{smallmatrix}]:=[\begin{smallmatrix}n\\t\end{smallmatrix}]_1.\notag \\
\end{eqnarray}
The simply connected quantum enveloping algebra\footnote{Over
$\mathbb{Q}(v)$ Lustzig \cite[p.90]{L3} and in \cite{L5} credits
this form to Drinfeld and Jimbo. Lusztig himself considers in these
papers more general rings as coefficients, especially
$\mathbb{Z}[v,v^{-1}]$. The $\prime$ notation, convenient for us
here (freeing the unprimed $\mathbb{U}$ for other uses) is not used
in \cite{L3}. Our usage of it is similar to that of \cite{L5},
suggesting the use of a quotient field, such as $\mathbb{Q}(v)$, in
the coefficient system.} $\mathbb{U}_v^{'} =
\mathbb{U}_v^{'}(\mathfrak{R})$ is the $\mathbb{Q}(v)$-algebra
generated by the symbols $E_i, F_i, K_i^{\pm 1}, 1\leq i \leq n,$
subject to the five sets of relations below, as found in
\cite[p.90]{L3}. We take this opportunity to warn the reader that we
will sometimes also need to refer to Lusztig's book \cite{L}, where
the symbols $K_i$ here (and in \cite{L3}) correspond to symbols
$\tilde{K}_i$ there.\footnote{In \cite{L} larger quantum algebras
are built, which we do not need. There are (new) elements $K_i$ in
these larger algebras, which serve as $d_i^{th}$ roots for the
elements $\tilde{K}_i$.}
\begin{itemize}
\item[(a1)] $K_iK_j = K_jK_i,~ K_iK_{i}^{-1} = 1 = K_i^{-1}K_i,~$
for $1\leq i, j\leq n, i\not=j;$
\item[(a2)] $K_iE_j=v^{d_ic_{i,j}}E_jK_i,$ and
 $K_iF_j=v^{-d_{i}c_{i,j}}F_jK_i,~$ for $1\leq i,j \leq n;$
%these are just for the $K_i^{d_i},$ not their inverses, eh
\item[(a3)] $E_iF_j - F_jE^{}_i=\delta_{i,j}\frac{K_i -
K_i^{-1}}{v^{d_i} - v^{-d_i}},~$ for $1\leq i \leq n.$
\item[(a4)] $\sum_{s+t = 1-c_{i,j}}(-1)^s \left [\begin{smallmatrix}{1-c_{i,j}}\\s\end{smallmatrix}\right
] E_i^sE_jE_i^t = 0,~$ for $1\leq i,j \leq n, \quad i\not= j; $

\item[(a5)] $\sum_{s+t = 1-c_{i,j}}(-1)^s \left [\begin{smallmatrix}{1-c_{i,j}}\\s\end{smallmatrix}\right
] F_i^sF_jF_i^t = 0,~$ for $1\leq i,j \leq n, \quad i\not= j; $
\end{itemize}
Note that it is also common to let $v_i = v^{d_i},$ so e.g., the
first part of (a2) can be rewritten as
%This  may not be right, given our choice of notation for the Cartan matrix -- check
$K_iE_jK_i^{-1}=v^{(\alpha_i,\alpha_j)}E_j$ for $1\leq i,j \leq n,$
and similarly for the second part of (a2).

The algebra $\mathbb{U}_v^{'}$ is also a Hopf algebra, with
comultiplication $\Delta$, antipode $S$, and counit $\epsilon$ given
as below [$ibid$]. These formulas hold for all indices $i$ with
$1\leq i \leq n.$
\begin{itemize}
\item[(b)] $\Delta E_i = E_i\otimes 1 \oplus K_i\otimes E_i,~$
  $\Delta F_i = F_i\otimes K_i^{-1} \oplus 1\otimes F_i,~$
\item[]      $\Delta K_i = K_i \otimes K_i,$
\item[(c1)] $S E_i = -K_i^{-1}E_i,~$  $SF_i = -F_iK_i,~$ $SK_i =
K_i^{-1},$
\item[(c2)] $\epsilon E_i = \epsilon F_i = 0,~$ $\epsilon K_i = 1.$
\end{itemize}

\medskip
We next give a brief discussion of the Lusztig integral form of this
Hopf algebra \cite{L3}.

Starting with $E_i, F_i, K_i \in \mathbb{U}_v,$ $1\leq i \leq n,$
and $s, t \in \mathbb{N}$, $c \in \mathbb{Z}$, define the divided
powers $E_i^{(s)}, F_i^{(s)}, \left [\begin{smallmatrix}K_i;
c\\t\end{smallmatrix}\right ]$ by
\begin{eqnarray}\label{e:divpow}
E_i^{(s)}&:= &\frac{E_i^s}{[s]_{d_i}!}, \notag\\
F_i^{(s)}&:= &\frac{F_i^s}{[s]_{d_i}!},\notag\\
\left [\begin{smallmatrix}K_{i}; c\\t\end{smallmatrix}\right ]&:=
&\Pi_{j=1}^{t} \frac{K_iv^{d_i(c-j+1)} -
K_i^{-1}v^{d_i(-c+j-1)}}{v^{d_ij} - v^{-d_i j}}\label{e:Ks}.
\end{eqnarray}

Each term on the left above with $s=0$ or $t=0$ is defined to be
$1$.

Set $\mathcal{Z} = \mathbb{Z}[v,v^{-1}].$ Keeping the hypotheses
that $\mathfrak{R}$ is the root datum for a semisimple complex Lie
algebra, with $\Pi^{\vee} = \{\alpha_1^{\vee},\dots,
\alpha_n^{\vee}\},$ then  the Lusztig integral form\footnote{This is
the form employed in \cite{A} (and also \cite[Appendix H]{J}). In
\cite[2.4]{ABG}, a version of Lusztig's integral form is given very
loosely, but apparently intended to be defined by an ``adjoint type"
version of the relations we use here. The latter relations, however,
appear to be consistent with the alternate ``simply connected"
development suggested \cite[Remark 2.6]{ABG}, a point of view we
have used throughout this paper.} $\mathbb{U}_{\mathcal{Z}} =
\mathbb{U}_{\mathcal{Z}}(\mathfrak{R})$ of $\mathbb{U}_v^{'}$ is the
$\mathcal{Z}$-subalgebra of $\mathbb{U}_{v}^{'}$ generated by
$E_i^{(s)}, F_i^{(s)}, K_i^{\pm 1} \text{and} \left
[\begin{smallmatrix}K_{i}; c\\t\end{smallmatrix}\right ], \text{for
all} i \text{ with} \quad 1\leq i \leq n, s,t\in \mathbb{N}, c\in
\mathbb{Z}$; equivalently (it turns out) $\mathbb{U}_{\mathcal{Z}}$
is generated by all $E_i^{(s)}, F_i^{(s)}, {K_i}^{\pm 1}$ and
\footnote{These additional expressions in the $K_i$ are
redundant--see the brief discussion \cite [p.3, bottom]{A}--but are
needed for the integral triangular decomposition.} $\left
[\begin{smallmatrix}K_{i}; 0\\t\end{smallmatrix}\right ],  \quad
1\leq i \leq n,\, s,t \in \mathbb{N}$. Corresponding to either set
of generators, both $\mathbb{U}'_{v}$ and $\mathbb{U}_{\mathcal{Z}}$
have (compatibly generated) triangular decompositions
$\mathbb{U}'_{v} = {\mathbb{U}'_{v}}^{-} \otimes
{\mathbb{U}'_{v}}^{0} \otimes{ \mathbb{U}'_{v}}^{+},$ resp.,
$\mathbb{U}_{\mathcal{Z}}= \mathbb{U}^{-}_{\mathcal{Z}} \otimes
\mathbb{U}_{\mathcal{Z}}^{0} \otimes \mathbb{U}_{\mathcal{Z}}^{+}.$
Either set of generators, with the relations (a1),$\ldots$, (a5),
define $\mathbb{U}_{v}^{'}$ over $\mathbb{Q}(v)$. These relations
are often sufficient to work with $\mathbb{U}_{\mathcal{Z}}$.
However, there are many additional useful relations \cite[\S 6]{L3}
on the elements $\left [\begin{smallmatrix}K_{i};
0\\t\end{smallmatrix}\right ]$ and their interactions with the
``divided powers" $E_i^{(s)}, F_i^{(s)}$. Also, there are analogs of
the latter elements for all positive roots. All of these elements
belong to $\mathbb{U}_{\mathcal{Z}}$, and may be used to define the
latter by generators and relations in its own right, and to
construct for it a monomial basis \cite{L3}.

    Finally, the $\mathcal{Z}$-algebra $\mathbb{U}_{\mathcal{Z}}$ is a Hopf
algebra, inheriting its Hopf algebra structure from
${\mathbb{U}'_{v}}$ \cite[8.11]{L3}. (All the Hopf algebras
mentioned in this paragraph have bijective antipodes, with clearly
invertible squares.)  Similar statements apply for
$\mathbb{U}_{\mathcal{Z}}^{0}$, ${{\mathbb{U}^{-}}_{\mathcal{Z}}}
\otimes \mathbb{U}_{\mathcal{Z}}^{0}$, and
$\mathbb{U}_{\mathcal{Z}}^{0} \otimes \mathbb{U}_{\mathcal{Z}}^{+}$
[$ibid$]. Also, the root of unity specializations discussed in the
next section inherit Hopf algebra structures from
$\mathbb{U}_{\mathcal{Z}}$, as do the ``small" quantum groups
$\mathbb{u}$,
$\mathbb{u}^{-},\mathbb{u}^{0},\mathbb{u}^{+}$[$ibid$].The
(``Frobenius") homomorphism discussed later in section
\ref{subsec:twisted} is a homomorphism of Hopf algebras.

\smallskip

\subsection{Quantum specializations at roots of unity--notation}
\label{subsec:specializations} For any commutative ring $K$ and
invertible element $q\in K,$ with unique accompanying `evaluation
morphism' $\epsilon_q:\mathcal{Z} \to K$ satisfying $v \mapsto q,$
define the \textbf{specialization} of $\mathbb{U}_{\mathcal{Z}} =
\mathbb{U}_{\mathcal{Z}}(\mathfrak{R})$ by
\begin{equation}\label{e:specialization}
\mathbb{U}_{q,K}=\mathbb{U}_{\mathcal{Z}} \otimes_{\mathcal{Z}} K,
\end{equation} where the tensor product is formed by using the
$\mathcal{Z}$-module structure on $K$ given by $\epsilon_{q}.$ For
this paper we will be interested in specializations where
\begin{itemize}
\item $K$ is a field of characteristic zero
\item $q$ is a primitive $\ell^{th}$-root of unity in $K$ with
$\ell$ odd, and $\ell \not=3$ if the root system of $\Lg$ has a
component of type $G_2.$
\end{itemize}
$\textit{We henceforth fix this meaning for}~ K, \ell, q$, unless
otherwise noted. $\textit{We also take}$ $\ell>h$, the Coxeter
number\footnote{All these restrictions agree in substance with those
in \cite{ABG}, though $h$ there denotes the dual Coxeter number.
Also, the literature differs as to whether $q$ is chosen to be the
image of $v$ or the image of $v^2$, the latter fitting somewhat
better with Hecke algebra notation. This makes little difference
when the order of the image of $v$ is odd, as is the case here.},
from Corollary \ref{c:nBWcor} forward.  We also now introduce
further notation that will be used in this quantum root of unity
setting, and also used in a parallel setting from algebraic groups,
discussed below. Relatively abbreviated notations are chosen to
facilitate later parallel discussions.  In the present
$\textit{quantum context}$, we let $\mathbb{U}$ denote the
specialization $\mathbb{U}_{q,K}$ as in (\ref{e:specialization}).
Similar conventions are adapted for
$\mathbb{U}^{+},\mathbb{U}^{0},\mathbb{U}^{-}$ and
$\mathbb{B}=\mathbb{U}^{0}\cdot\mathbb{U}^{-}$. Lusztig's finite
dimensional Hopf algebra \cite[\S 8.2]{L3} (the ``small" quantum
group) is denoted $\mathbb{u}$, with components of its triangular
decomposition denoted
$\mathbb{u}^{-},\mathbb{u}^{0},\mathbb{u}^{+}$. For example,
$\mathbb{u}^{0}$ is generated by all $K_i^{\pm}$, and
{$\mathbb{u}^{-}$ is generated by all the $F_i$
\cite[pp.107-108]{L3}. Imitating the notation in \cite{ABG} we set
$\mathbb{b}:=\mathbb{u}^{-}\cdot\mathbb{u}^{0}$ and
$\mathbb{p}:=\mathbb{b}\cdot\mathbb{U}^{0}$.

Overall, our notation here is quite similar to that used for quantum
groups at a root of unity in \cite{ABG}, with the exception that our
characteristic $0$ field $K$ (which may be compared with
$\mathbb{k}$ in \cite{ABG}) is not assumed to be algebraically
closed. Also, our $\mathbb{B}$ is
$\mathbb{U}^{0}\cdot\mathbb{U}^{-}$, whereas in \cite{ABG} the same
symbol $\mathbb{B}$ is used to denote
$\mathbb{U}^{0}\cdot\mathbb{U}^{+}$.

\subsection{Some parallel algebraic groups notation}\label{subsec:parallel}
Let $G$ be a simply connected semisimple algebraic group, with root
datum $\mathfrak{R}$, over an algebraically closed field $k$ of
characteristic $p>h$. We assume $G$ is defined and split over the
prime field $\mathbb{F}_p$.  In particular there is a Borel subgroup
$B=TU$, with $U$ the unipotent radical of $B$, and $T$ a maximal
torus, all defined over $\mathbb{F}_p$, with $T$ isomorphic (over
the same field) to a direct product of copies of $k^{\times}$. The
root groups in $B$ are viewed as negative. We refer to this set-up
as the $\textit{algebraic groups context}$ or, even more loosely, as
the $\textit{algebraic groups case}$. The distribution algebras
Dist$(G)$, Dist$(B)$, Dist$(T)$, Dist$(U)$, and Dist$(G_1)$ (the
restricted enveloping algebra) parallel $\mathbb{U}$, $\mathbb{B}$,
$\mathbb{U^{0}}$, $\mathbb{U^{-}}$, and $\mathbb{u}$, respectively.
We will use the latter symbol set in place of the former, when the
context is clear, or if both the algebraic groups and quantum
contexts have been explicitly allowed. In either of these
circumstances, additional notational substitutions in the same
spirit may also be made, such as $\mathbb{p}$ for Dist$(B_1T)$.

\subsection{Affine Weyl Groups} Affine Weyl groups $W_\ell$ are used to index modules in both
the quantum and algebraic groups context, with $p$ used for $\ell$
in the latter. Our main references for affine Weyl groups are
\cite{J} and \cite{A}. To clarify discussions and differences in
these references, we temporarily allow $\ell$ to be any positive
integer.
%In this subsection, we:
%\begin{itemize}

%\item Define affine Weyl group in terms of reflections w.r.t. $R$,
%e.g., as in \cite{J}
%\item Define affine Weyl group $W_a(R^{\vee})$ as in Bourbaki
%\item Define affine Weyl group as in Andersen \cite{A}
%\item Consider affine Weyl group as in \cite{DDPW} or {Carter}
%\item Compare these formulations %\use DDPW, \cite{A}, too?
%\item Define dot action.
%\item Compare conventions of \cite{ABG} and what we've introduced
%above.
%\end{itemize}

\subsubsection{Affine Weyl Groups, as in
\cite{J}}\label{subsubsec:affJ} Following e.g., the conventions and
notation in \cite[\S 6.1]{J}, for $\beta\in R$ and $m\in
\mathbb{Z},$ define the affine reflection on ${\mathbb{X}}$ by
$$s_{\beta,m}(\lambda) = \lambda - (\lambda,\beta^{\vee}) -
m)\beta, \quad \forall \lambda\in {\mathbb{X}};$$ one could take
${\mathbb{X}}\otimes_{\mathbb{Z}}\mathbb{R}$ in place of
${\mathbb{X}}.$ Thus, for the reflections $s_{\beta}$ given by
$s_{\beta}(\lambda) = \lambda -(\lambda,\beta^{\vee})\beta$
 one has
$$s_{\beta,m}(\lambda) = s_{\beta}(\lambda) + m\beta \quad \forall
\lambda.$$ For any positive integer $\ell,$ the affine Weyl group
$W_{\ell}$ is the group
$$W_{\ell}:=<s_{\beta,n\ell}\,|\, \beta\in R, n\in \mathbb{Z}>.$$
In the notation used in \cite[\S 6.1]{J}, there is a (largely
formal) isomorphism $W_{\ell}\cong W_{a}(R^{\vee}),$ for
$W_{a}(R^{\vee}):=<s_{\beta,m}\,|\, \beta\in R,m\in \mathbb{Z}>,$ as
defined by Bourbaki \cite[ch. VI, \S 2]{B}. When $\ell=1$ this
isomorphism is an equality. The Bourbaki reference makes a good case
for the labelling with $R^{\vee}$, though it is common in algebraic
group theory to associate both $W_\ell$ and $W_{a}(R^{\vee})$ with
the root system $R$. A familiar semidirect product description is
obtained by regarding $\ell\mathbb{Z}R$ as a group of translations
on ${\mathbb{X}}\otimes_{\mathbb{Z}}\mathbb{R}$, namely,
$W_{\ell}\cong \ell \mathbb{Z}R \rtimes W =\ell\mathbb{Y} \rtimes W$
(\cite{J} references \cite[ch. VI\S 2 prop. 1]{B} for a proof). Here
$W$ is the usual Weyl group associated to $R$.
%But Jantzen also says he's defining reflections
% by form $x \mapsto x - (<x + \rho,\alpha^{\vee}> - np)\alpha$
%but this isn't consistent with the defn of $W_p$ he gives!

\subsubsection{Dot Action} \label{subsubsec:dot}
Set $\rho = \frac{1}{2}\sum_{\alpha\in R^+}\alpha.$ For $w\in W,$
$\lambda \in {\mathbb{X}},$ one sets,
$$w\cdot \lambda = w(\lambda + \rho) - \rho.$$
%Define $\rho!$
More generally, the affine Weyl group $W_{\ell}\cong \ell\mathbb{Y}
\rtimes W$ acts, for $w_a:= \ell\tau \rtimes w  \in \ell
\mathbb{Y}\rtimes W$ by
$$w_a \cdot \lambda:=w(\lambda + \rho) - \rho + \ell \tau, \quad \text{for all }\lambda \in \mathbb{X}.$$
Our ``dot" action $\cdot$ , which is often used, is not given by the
same formula as the action $\bullet$ defined in \cite{ABG}, possibly
intending some variation on \cite[Lem. 3.5.1]{ABG} (which is
incorrect as stated for ``positive" Borel subalgebras). A simpler
approach, it seems to us, is to use ``negative" Borel subalgebras
and the usual ``dot" action.
%This is ABG definition -- not consistent -- rectify

\subsubsection{Affine Weyl Groups as in \cite{A}} \label{subsubsec:affA}
 Take $q\in K$ to be a root of unity. (In \cite{A} the field $K$ can have any
characteristic, though that is not relevant to our discussions here,
and we may keep our assumption that $K$ has characteristic 0.)  Set
$\ell$ to be the order of $q^2,$ so that $q$ is a primitive
$\ell^{th}$ or $2\ell^{th}$ root of unity. For the Cartan matrix $C$
of $\mathfrak{R}$ and symmetrization $DC$ (as at the end of our
Section \ref{subsub:root datum}), set $\ell_{i} =
\frac{\ell}{gcd(\ell,d_i)}$.
%Is this normalization correct -- double-check w/Paramasamy and Len...
For each $\beta\in R,$ there is some $i, 1\leq i \leq n$ so that
$\beta$ is conjugate under the Weyl group $W$ of $R$ to $\alpha_i$.
Set $\ell_{\beta} = \ell_{i}$~ (well-defined). For each $\beta\in R$
and $m\in \mathbb{Z},$ as in Section \ref{subsubsec:affJ}, we have
the affine reflection $s_{\beta,m\ell_{\beta}}$ with
$$s_{\beta,m\ell_{\beta}}\cdot\lambda = s_{\beta}\cdot\lambda + m\ell_{\beta}\beta \quad
\forall \lambda\in {\mathbb{X}}.$$ Here, $s_{\beta}$ and
$s_{\beta}\cdot\lambda$ are defined just as in Sections
\ref{subsubsec:affJ} and \ref{subsubsec:dot}. Now define a new group
of affine reflections
$$W_{D,\ell}:=<s_{\beta,m\ell_{\beta}}\,|\, \beta\in R, m\in \mathbb{Z}>.$$

Take $W^{\vee}_{\ell}$ to be the group
$$W^{\vee}_{\ell}:=<s_{\beta^{\vee},n\ell}\,|\,\beta\in R^+, n\in
\mathbb{Z}>,$$ generated by reflections as in Section
\ref{subsubsec:affJ}, but utilizing coroots in place of roots. The
following proposition relates the three groups $W_{\ell},
W_{D,\ell},$ and $W_{\ell}^{\vee}.$

\begin{prop} \label{p:aff} Assume $R$ is indecomposable. There are
identifications giving inclusions
$$W_{\ell} \subseteq  W_{D,\ell} \subseteq W_{\ell}^{\vee}$$
so that \begin{itemize}
\item[(i)] If $gcd(d_i,\ell) = 1$ for all $1\leq i\leq n,$ then $W_{\ell} = W_{D,\ell}.$

\item[(ii)] On the other hand, if $gcd(d_i,\ell) \neq 1$ for some $1\leq i\leq n,$
 then $W_{D,\ell} =
W_{\ell}^{\vee}.$
\end{itemize}
\end{prop}

\begin{proof} Without loss, some $d_i\neq1$. Since $R$ is assumed to
be indecomposable, all $d_i\neq 1$ take the same value $d\in
\{1,2,3\}$. If $d$ does not divide $\ell$, $\ell_i=$ for all indices
$i$, and it follows that $\ell=\ell_\beta$ for all $\beta \in R$.
Consequently, $W_{D,\ell}=W_\ell$. On the other hand, if $d$ does
divide $\ell$, then $d\beta^{\vee}=\beta$ and $d\ell_\beta =\ell$
for all long $\beta\in R$, and $\beta^{\vee}=\beta$ for all short
roots $\beta$. It follows that $W_{D,\ell} = W_{\ell}^{\vee}$ in
this case. This proves the proposition.
\end{proof}

\begin{rem} We have included Proposition
\ref{p:aff} in part to address possible confusion that a reader
casually comparing \cite{J} and \cite{A} may encounter.  \cite[p.6]{A} %check page #
says ``Note that if $\ell$ is prime to all entries of the Cartan
matrix, then the group $W_{D,\ell}$ (denoted $W_{\ell}$ in
\cite{A}!) is the `usual' affine Weyl group of $R.$ However, in
general $W_{D,\ell}$ is the affine Weyl group of the dual root
system''. As we have pointed out above,  the `` `usual' affine Weyl
group'' in algebraic groups discussions is $W_{\ell}$ as defined in
\cite{J} and Section \ref{subsubsec:affJ} above, and that``the
affine Weyl group on the dual root system'' referred to by \cite{A}
is $W_{\ell}^{\vee}\cong W_a(R)$, rather than $W_a(R^{\vee}).$ The
proposition and our previous discussion perhaps make precise what
Andersen intended. In any case, in this paper, under the assumption
below (\ref{e:specialization}) that $\ell$ be odd, and not divisible
by $3$ in case the root system has a component of type $G2$, it is
clear from the proposition that $W_{\ell} = W_{D,\ell}$.

\end{rem}
\subsection{Induction and Cohomology}\label{s:induction}
%What assumption, if any, on the root system? Andersen speaks of using the
%root system of a semisimple algebraic group in the intro to his paper, but
%later just seems to say "Let $R$ be a root system...."
We continue the notation of Section \ref{subsubsec:affA} above,
appropriate for Andersen's paper \cite{A}. This is somewhat more
general than our standard assumptions stated below
(\ref{e:specialization}). Those more special assumptions are all
that we need for this paper, and are explicitly used as hypotheses
in \cite{AW}. The latter paper, along with some arguments of
\cite{APW} could possibly be used as an alternate source for some of
the results of this subsection, with the standard assumptions below
(\ref{e:specialization}) as hypotheses. Unfortunately, it is not
possible to quote \cite{APW} directly, since its standing
assumptions effectively require $\ell$ to be a prime power. (See
\cite[Lem. 6.6]{APW}, \cite[p.35]{AW}.) On the other hand, \cite{A}
contains explicit statements (with weaker hypotheses) of most of the
results we need, with the exceptions tractable with modest effort.

Accordingly, we follow \cite{A} using the notation for $K, q, \ell$
in the previous subsection. In addition, we use the notation
$\mathbb{U}_{q,K}$  as in (\ref{e:specialization}), though with more
general assumptions than those below (\ref{e:specialization}). We
will define induction functors
$\Ind_{\mathbb{B}_{q,K}}^{\mathbb{U}_{q,K}}$ from the category of
integrable $\mathbb{B}_{q,K}$-modules of Type 1 to the category of
integrable $\mathbb{U}_{q,K}$-modules of Type 1 (both categories
defined below). For the moment, we will not use our preferred
$\mathbb{U}, \mathbb{B}, \dots$ notation, to help remind the reader
of our slightly different context, with weaker hypotheses. Of
course, we will obtain from this construction the induction functors
$\Ind_{\mathbb{B}}^{\mathbb{U}}$ whose right derived functors are
the focus of this paper.

We begin as in \cite[\S 1]{A}. First, we coordinate the notation
$\mathbb{X}$ in our section \ref{subsub:root datum} with the
``weights" $\mathbb{Z}^n$, given in \cite[p.3]{A}. The
correspondence is simply to let $\lambda\in \mathbb{X}$ correspond
to the $n$-tuple with $i^{th}$ coordinate
$\lambda_i=<\lambda,\alpha_i^{\vee}>$. Then, as in $\textit{loc.
cit.}$, $\lambda$ defines a 1-dimensional representation of (what we
call here) $\mathbb{U}_{q,K}^0$ via the homomorphism
$\chi_\lambda:\mathbb{U}_{q,K}^0\to K$ sending $K_i^{\pm}$ to
$q^{\pm d_i \lambda_i}$ and $\left [\begin{smallmatrix}K_{i};
c\\t\end{smallmatrix}\right ]$ to $\left
[\begin{smallmatrix}\lambda_i+c\\t\end{smallmatrix}\right]_{d_i}.$
 Here $1\leq i\leq n,~~ c\in \mathbb{Z}$, and $t\in \mathbb{N}$.  For
any $\mathbb{U}_{q,K}^0$-module $M,$ let $M_\lambda$ denote the sum
of all 1-dimensional submodules on which $\mathbb{U}_{q,K}^0$ acts
via the homomorphism $\chi_\lambda$. We will call $M_\lambda$ the
``weight space" for $M$ associated to $\lambda$. If $M$ is the sum
(necessarily direct) of its weight spaces $M_\lambda$, $\lambda \in
\mathbb{X}$, we say that $M$ is integrable of Type 1 as a
$\mathbb{U}_{q,K}^0$-module. If we start with $M$ a
$\mathbb{B}$-module (resp., a $\mathbb{U}$-module), we say that $M$
is integrable of Type 1 as a $\mathbb{B}$-module (resp., as a
$\mathbb{U}$-module) if it is integrable of Type 1 as a
$\mathbb{U}_{q,K}^0$-module, and each vector $v\in M$ is, for each
index $i$, killed by all $F_i^{(s)}$ for $s$ sufficiently large
(resp., killed by $F_i^{(s)}$ and $E_i^{(s)}$ for $s$ sufficiently
large).\

Next, suppose that $V$ is any $\mathbb{U}_{q,K}$-module. Define
\begin{equation} \mathcal{F}(V):=\{v\in \oplus_{\lambda\in
{\mathbb{X}}}M_{\lambda}\,|\,E_{i}^{(r)}v = F_{i}^{(r)}v = 0~
\forall i = 1,\dots, n \mbox{ and } \forall r >> 0\}.
\end{equation}
According to \cite[p.5]{A}, the submodule $\mathcal{F}(V)$ is a Type
1 integrable $\mathbb{U}_{q,K}$-module. \footnote{No argument is
given in \cite{A}, noting the property is ``not hard to check."
Perhaps this is true, once one knows how to do it. An argument for
the case of (positive or negative Borel) subalgebras may be obtained
with the method of \cite[proof of Lem. 3.5.3]{L}, but using the
generalized quantum Serre relations (through their corollary
\cite[Cor. 7.1.7]{L}) in place of the quantum Serre relations. The
case of the full quantum enveloping algebra then reduces to the rank
1 case, which can be handled with the formulas
\cite[3.14(b),(c)]{L}.} We can now define
\begin{equation}\label{e:induction}
H_{q}^{0}(M):=\mathcal{F}(\Hom_{\mathbb{B}_{q,K}}(\mathbb{U}_{q,K},M)),
\end{equation}
for any Type 1 integrable $\mathbb{B}_{q,K}$-module $M.$ This yields
a Type 1 integrable $\mathbb{U}_{q,K}$-module which we call the
induced module $\Ind_{\mathbb{B}_{q,K}}^{\mathbb{U}_{q,K}}(M)$,
later to be written in this paper as
$\Ind_{\mathbb{B}}^{\mathbb{U}}(M)$. (Andersen uses the word
``induction," but does not use our notation for the induced module,
preferring instead $H_{q}^{0}(M)$.)
%By "coincides" Do I mean "equals" or just "iso to"
In the definition of $H_{q}^{0}(M)$ above, left multiplication of
$\mathbb{B}_{q,K}$ on $\mathbb{U}_{q,K}$ provides the
$\mathbb{B}_{q,K}$-module structure on $\mathbb{U}_{q,K},$ and a
$\mathbb{U}_{q,K}$-module structure on
$\Hom_{\mathbb{B}_{q,K}}(\mathbb{U}_{q,K},M)$ is given by $uf(x) =
f(xu)$ for all $u,x\in \mathbb{U}_{q,K}$ and
$f\in\Hom_{\mathbb{B}_{q,K}}(\mathbb{U}_{q,K},M).$ The categories of
Type I integrable $\mathbb{B}_{q,K}$-modules and
$\mathbb{U}_{q,K}$-modules have enough injectives (as may be seen
from the ring cases, applying the ``largest Type 1 integrable
submodule functors," such as $\mathcal{F}$ above)) and, hence, the
left exact functor $\Ind_{\mathbb{B}_{q,K}}^{\mathbb{U}_{q,K}} =
H_{q}^0$ has right derived functors
$R^n\Ind_{\mathbb{B}_{q,K}}^{\mathbb{U}_{q,K}} = H_{q}^n.$

\begin{defn}\label{d:linked}
For $\leq$ the usual order on ${\mathbb{X}}$ determined by the
positive roots $R^+,$ set $\mu,\lambda\in {\mathbb{X}}$ to be
\textbf{linked} if $\mu = w\cdot \lambda$ for some $w\in
W_{D,\ell}.$  If there is a chain $\lambda = \lambda_1,\dots,
\lambda_s = \mu$ and a sequence $s_{\beta_1,m_{1}\ell_{\beta_1}},
\dots, s_{\beta_{s-1},m_{s-1}\ell_{\beta_{s-1}}}$ for which
$\lambda_{i}
 \geq \lambda_{i+1}:=s_{\beta_i,m_{i}\ell_{\beta_i}}\cdot \lambda_{i},$
 $i = 1,\dots, s-1,$
 then $\mu$ is \textbf{strongly linked} to $\lambda,$ denoted
 $\mu\uparrow_{D,\ell} \lambda.$
 %Note: I'm using slightly different notation from \cite{A}, to conform to what
 %see in \cite{J}.
\end{defn}

\begin{rems} (1) The relationship of strong linkage for weights ${\mathbb{X}}$ refines
that of the usual ordering $\leq$. That is, $\mu
\uparrow_{D,\ell}\lambda$ implies $\mu \leq \lambda.$

(2) In the analogous $\ch(k) = p >0$ representation theory of
algebraic groups, one defines $\mu \uparrow \lambda$ for weights
$\mu,\lambda\in {\mathbb{X}}$ by using the affine Weyl group
$W_{\ell}, \ell = p,$ in place of $W_{D,\ell}.$ In this
circumstance, under mild restrictions on the prime $p$ relative to
the root system $R,$ one has $W_{p} = W_{D,\ell},$ by Proposition
\ref{p:aff}(1).
\end{rems}

 Let  $\mathcal{C}_q$ denote the category of Type 1 integrable
$\mathbb{U}_{q,K}$-modules.  The following fundamental result
ultimately yields a splitting of $\mathcal{C}_q$
 into a direct sum of blocks associated to orbits of an appropriate affine Weyl group.
 For application to the Induction Theorem
 \ref{T:induct}, we will just need the version $\mathbb{U}$ of
 $\mathbb{U}_{q,K}$ described below (\ref{e:specialization})
 in which case $W_{D,\ell} = W_{\ell}$.
 We will then focus on  $block(\mathbb{U}),$
 the $\textbf{principal block}$ of $\mathcal{C}_q,$ corresponding to the orbit
 $W_{D,\ell}\cdot 0.$ (Composition factors $L_q(\mu)$ of modules in the block are indexed by dominant weights $\mu$ in the orbit.)
 However, the results below hold more
 generally. They are claimed in \cite{A} under the standing hypotheses of this subsection on
 $q,K,\ell$, even with $K$ allowed to have positive characteristiic.
 However, it should be pointed out that the only reference given in support of one key auxilliary result
 \cite[Thm. 2.1]{A}, a Grothendieck vanishing theorem needed in the proofs, is to the paper \cite{AW}. The latter has as one of its
 standing assumptions that $\ell$ be odd, and not divisible by 3 in case the root system has a component of type G2.
 This assumption on $\ell$ is, of course, satisfied by our
 $\mathbb{U}$, so we have not pursued the issue further. Possibly, it was the intent of Andersen to
 claim that the argument in \cite{AW} worked in the more general
 set-up of \cite{A}, though there is no explicit comment to that
 effect.

 \begin{theorem}\label{t:strong}
 (1) (Strong Linkage Principle  \cite[Theorem 3.1, Theorem 3.13]{A})
 Let $\lambda\in \mathbb{X}^{+} - \rho$. (Thus, $\lambda + \rho \in
\mathbb{X}^+$.) Let $\mu\in \mathbb{X}^+.$ If $L_q(\mu)$ is a
composition factor of some $H_q^i(w\cdot \lambda)$ with $w\in W$ and
$i\in \mathbb{N},$ then $\mu \uparrow_{D,\ell} \lambda.$

(2) (Linkage Principle \cite[Thm. 4.3, Cor. 4.4]{A}). Let
$\lambda,\mu\in {\mathbb{X}}^+.$ If
$\Ext^1_{\mathcal{C}_q}(L_q(\lambda),L_q(\mu))\not=0,$ then
$\lambda$ is linked, but not equal to, $\mu.$ Consequently, if $M\in
\mathcal{C}_{\mathbb{U}}$ is indecomposable, then the highest
weights of all composition factors of $M$ are linked, and the
category $\mathcal{C}_q$ splits into blocks corresponding to the
orbits for the dot action of $W_{D,\ell}$ on ${\mathbb{X}}^+.$
\end{theorem}

\begin{proof} We refer the reader ro \cite{A} for the proofs, on which we
we make several remarks which may be helpful. First, note that there
appears to be a serious misprint, an expression apparently carried
over unintentionally  to one result from a previous one, in the
statement of the auxiliary result \cite[Prop. 3.6]{A}: In the
expression ``$<\lambda,\alpha_i^{\vee}> = -1$", the subexpression
``$= -1$" should be replaced with ``$\geq 0$".

   Next, note that the exact sequences labeled (3) and (4) of \cite[p.8]{A}
exist (and are later used) in the case $s=1$ of the discussion
there, with all terms in both exact sequences sequence equal to
zero. (The reader might have been led by the wording  to think these
sequences were defined only for $s>1$.)

   Next, there is an organizational issue on \cite[p.10]{A}.
The first three lines of the proof of \cite[Cor. 3.8]{A} do not use
the ``minimality" hypothesis of that corollary, and are implicitly
quoted later on the same page, in the proof of \cite[Thm. 3.8]{A},
where it is claimed ``we have already checked the result for $w=1$."

   There are further minor points which occur on the same page
\cite[p.10]{A}. In one repeated case, the vanishing of $H^0_q$ on
1-dimensional nondominant $\mathbb{B}$-modules is given without
proof, or hint. One approach that works is to use the version proved
in the rank 1 case, then use induction from a corresponding
parabolic subalgebra (and a Grothendieck spectral sequence).

    At another place on the same page, \cite{APW} is quoted to help
determine, using a Weyl group action, the highest weight of a module
$H^0_q(k_\lambda)$. However, an alternate argument may be given
directly from the induced module definition. Quoting \cite{APW} in
this context is undesirable, because of the (implicit) restrictive
set-up of that paper regarding $\ell$. A similar issue, which we
already noted above, before the statement of theorem, regards the
reference to \cite{AW} for a proof of \cite[Thm. 2.1]{A}. As noted
above, the generality of the \cite{AW} set-up is sufficient for
applications in this paper.

   Finally, the ``splitting into blocks" is justified in \cite{A} by
corollary \cite[Cor. 4.4]{A}. Both the corollary and the splitting
are made a straightforward consequence of \cite[Thm. 4.3]{A} by the
local finiteness of Type 1 integrable modules, for which we refer
ahead to Proposition \ref{p:locfinite} below.
\end{proof}

The paper \cite{A} gives some history of the Theorem \ref{t:strong},
most of which had been proved piecemeal previously by Andersen and
his students and collaborators. There is, of course, a completely
corresponding theorem--first proved in full generality by Andersen--
for semsimple algebraic groups, as discussed in Jantzen's book
\cite{J}. With a certain amount of hindsight, some conceptual
similarities can be imposed on the proofs and statements of
supporting results.
%See also Andersen reference suggested by Jantzen?
In particular, the presentation in \cite{J} of Strong Linkage for
the algebraic groups case (\cite[II 6.13]{J}), working with
$\mathbb{X}^+$ rather than $\mathbb{X}^+ - \rho$, breaks the proof
down into a lemma and two propositions \cite[II 6.15, 6.16]{J}. We
have combined these propositions into an $\mathbb{X}^+ -\rho$
quantum analogue stated below. We need this extra detail (for
$\mathbb{X}^+$) in order to provide more precise information about
the appearance of irreducible modules $L_q(\mu)$ as composition
factors in appropriate cohomology modules. $H_q^i(\nu).$ The Theorem
\cite[Thm. 2.1]{A}, discussed above, is needed in the proof (beyond
the use of Strong Linkage).

\begin{prop} \label{p:qcohom}%quantum analogue of \cite[II. 6.15, 6.16, 6.13, 6.17]{J}
\begin{enumerate}
\item Let $i\in \mathbb{N}$ and $w\in W$. If
$L_q(\mu),$ $\mu\in {\mathbb{X}}^+$ is a composition factor of
$H_q^i(w\cdot \lambda)$ with $\lambda \in \mathbb{X}^+ - \rho$, then
$\mu \uparrow_{D,\ell} \lambda$. If $\ell(w)\not= i$, then $\mu <
\lambda.$
\item Suppose $\lambda\in {\mathbb{X}}^+.$ Then $L_q(\lambda)$ is a
composition factor with multiplicity one of each
$H_q^{\ell(w)}(w\cdot \lambda)$ with $w\in W.$
%(Recall, in particular, $\mu \uparrow \lambda \Rightarrow \mu \leq
%\lambda$ and $\mu \in W_p\cdot \lambda.$)%Need to fix this W
%\item (Strong Linkage Principle) Let $\lambda\in {\mathbb{X}}(T)$ with
%$<\lambda + \rho, \alpha^{\vee}> \, \geq 0$ for all $\alpha\in R^+.$
%Let $\mu\in {\mathbb{X}}(T)^+.$ If $L_q(\mu)$ is a composition factor of some
%$H_q^i(w\cdot \lambda)$ with $w\in W$ and $i\in \mathbb{N},$ then
%$\mu \uparrow \lambda.$
%\item (The Linkage Principle) Let $\mu,\lambda\in {\mathbb{X}}(T)^+.$ Then
%$\Ext^1_G(L(\lambda),L(\mu)) \not=0 \Rightarrow \lambda\in
%W_p\cdot \mu.$ %This is the statement of splitting into blocks, eh!
\end{enumerate}
\end{prop}

\begin{proof} The first part of item (1) just repeats Strong Linkage.
The second part of item (1), and  item (2), can be deduced from the
approach in the proof of \cite[Thm. 3.9]{A}. Note that any weight
$\mu$ in that proof which arises from the application of \cite[Lem.
3.7]{A} is strictly less than $\lambda$. As a consequence, the
argument shows, in the presence of Strong Linkage, that the lemma
holds for $i$ and $w$ if and only if it holds for $i+1$ and $sw$
(assuming $sw>$). This property can be applied repeatedly, moving
$i$ up or down. Using it, as in the first three lines of the proof
of \cite[Cor. 3.8]{A}, we obtain item (1) of the proposition. (This
uses the discussed \cite[Thm. 2.9]{A}.) Similarly, item (2) is
reduced to the case $i=0$ and $w=1$. Here it follows by showing,
directly from the definition, that $\lambda$ has a 1-dimensional
weight space in $H_q^0(\lambda)$. This completes the proof of the
proposition.
%proof.[COMPLETE
%WRITE-UP USING OLDER HKS NOTES. Compare with proofs in \cite{J}?
%RECALL ALSO THAT \cite{ABG} CLAIMED (A VERSION OF THESE) RESULTS
%WERE ``BOREL-WEIL THEOREM'' (their Lemma 3.5.1) AND FOUND IN$
%\cite{APW}, but we didn't find them there!!!]
\end{proof}
We remark that the discussions above of results in \cite{A} corrects
the proof of \cite[Lem. 3.5.1]{ABG}. This is with our choice of
``negative" Borel subalgebras and our (standard) ``dot" action. Both
\cite{A} and \cite{APW} use "negative" Borel subalgebras as we do.
The statement of the lemma given by \cite{ABG} is apparently an
attempt to use \cite{APW} in a ``positive" Borel subalgebra context,
but the lemma is still incorrectly stated for that context. Also,
they quote \cite{APW} for the proof, though the latter paper does
not contain as strong a result \cite[Lem. 3.5.1]{ABG} Instead, the
main result \cite[Thm. 6.7]{APW}  of its Borel-Weil-Bott section is
a ``lowest $\ell$ alcove"  version, and explicitly requires that
$\ell$ be a prime power.

\medskip The next corollary restates the main conclusions (those dealing
with $\mathbb{X}^+$) of the proposition above in a form handy for
later use. As already indicated, the (completely analogous)
algebraic groups version is the combination of the two propositions
\cite[II 6.15, 6.16]{J}.

\begin{cor}\label{c:qcohom}
If $\mu\in {\mathbb{X}}$ and $\lambda = w\cdot \mu\in
{\mathbb{X}}(T)^+$ (i.e., $\lambda$ is the dominant weight in the
$W-$orbit of $\mu$), then $L_q(\lambda)$ occurs just once as a
composition factor of any of the modules $H_q^i(\mu), $ $i$ running
over all nonnegative integers.  Precisely, one has
$[H_q^i(\mu):L_q(\lambda)]\not=0$ only for $i = \ell(w),$ and then
$[H_q^{\ell(w)}(\mu):L_q(\lambda)]=1.$ If $\eta\in {\mathbb{X}}^+$
with $\eta\not=\lambda,$ then for all $i\in \mathbb{N},$
$[H_q^i(\mu):L_q(\eta)]\not=0$ implies $\eta < \lambda$ and that
$\eta$ is strongly linked to $\lambda$
\end{cor}

%[NEED TO ADD COMMENTS/DETAILS TO PROOF of Cor \ref{c:nBWcor}, HERE, see e.g., April 30th 2011 discussion notes.]

%Need to spell out the proof here of \ref{c:BWcor} and contrast with [ABG] approach;
%see ABG Long Notes_4-4-08

%Let $\zeta\in \mathbb{C}$ be a primitive $\ell^{th}$-root of unity
%for $\ell$ an odd prime, $\ell \not=3$ if $\Lg$ has a factor of type
%$G_2.$ Let $\mathcal{O}$ be a complete discrete valuation ring with
%fraction field $K$ and residue field $\mathbb{F}_p;$ set $k =
%\overline{\mathbb{F}_p}.$ Set $G$ to be a connected, semisimple,
%simply connected algebraic group over $k$ that is defined and split
%over $\mathbb{F}_p \subset k.$

The proposition below is quite important for applications,
especially in the next subsection. There is a completely analogous
result for induction from Borel subgroups in reductive algebraic
groups, a special case of \cite[II,4.2]{J}.

\begin{prop}\label{p:finiteness} (\cite[Thm. 3.9, Thm. 2.1]{A}) Let $\mu \in
\mathbb{X}$. Then $H_q^i(\mu)$ has finite dimension over $K$, and
vanishes for $i>N$, the number of positive roots.
\end{prop}

\begin{proof} We ask the reader again to read \cite{A} for proofs,
after first reviewing our comments on the proof of Theorem
\ref{t:strong} above. \end{proof}.

The final proposition in this section, also useful in the next
section, is an analogue in the generality of \cite{A} of \cite[Cor.
1.28]{APW} and of \cite[Prop. 32.1.2]{L}. It does not appear to be
implied by either of these latter results, however.

\begin{prop}\label{p:locfinite}
Let $M$ be any Type 1 integrable $\mathbb{U}_{q,K}$-module. Then $M$
is locally finite, in the sense that each vector $v\in M$ generates
a finite dimensional $\mathbb{U}_{q,K}$-module. Similarly, Type 1
integrable $\mathbb{B}_{q,K}$-modules are locally finite.
\end{prop}

\begin{proof}.
For this proof, let $\mathbb{U}_j^{+}$ denote, for each index $j$,
the $\mathcal{Z}$ span (a subalgebra)  in $\mathbb{U}_{\mathcal{Z}}$
of all the elements $E^{(s)}_j,~ s\in \mathbb{N}$ , with a similar
notation for $\mathbb{U}_j^{-}$. Lusztig constructs his PBW-type
basis \cite[Thm. 6.7]{L3} for the quantum enveloping algebra
$\mathbb{U}_{\mathcal{Z}}$ using (finitely many) compositions of his
explicit braid group automorphisms $T_i$, applied to the various
$\mathbb{U}_j^{\pm}$. This process yields, for each positive root
$\alpha$, $\mathcal{Z}$-subalgebras $\mathbb{U}_{\alpha}^{\pm}$, and
the whole quantum algebra $\mathbb{U}_{\mathcal{Z}}$ is a
(ring-theoretic) product of finitely many of these, together with
$\mathbb{U}^0_{\mathcal{Z}}$.

In several formulas listed in \cite[37.1.3]{L} Lusztig gives
explicit formulas for several similar automorphisms, including their
action on basis elements of each $\mathbb{U}_j^{\pm}$. The setting
for the action of these automorphisms is a $\mathbb{Q}(v)$-algebra
$\bold{U}$ containing the algebra we have called $\mathbb{U}_v'$;
moreover, the action of these automorphisms on the various elements
$K_i$ (in the notation of \cite{L}) shows that all these
automorphisms act bijectively on $\mathbb{U}_v'$. It is easy to pick
out the braid group automorphism $T_i$ defined in \cite{L3} in this
context, as (the restriction to $\mathbb{U}'_v$ of) $T''_{i,1}$ in
\cite[37.1.2]{L}). Accordingly, we learn, for each index $j$, that
$T_i(\mathbb{U}_j^{\pm})$ is contained in a product of
$\mathbb{U}^0_{\mathcal{Z}}$ and at most three $\mathcal{Z}$
subalgebras, each of the latter having the form
$\mathbb{U}_{j''}^{\pm}$, for some index $j''$ (in $1,\ldots,n$). To
this information we add the fact that $T_i$ stabilizes
$\mathbb{U}_0$, which may be deduced from
\cite[Thm.3.3,Thm.6.6(ii),Thm.6.7(c)]{L3}.

It follows now that $\mathbb{U}_{\mathcal{Z}}$ is a product of
finitely many of the various subalgebras $\mathbb{U}_j^{\pm}$ ,
together with $\mathbb{U}^0_{\mathcal{Z}}$. However, it is obvious
that, if $V$ is any finite-dimension\texttt{}al subspace of $M$,
then any (ring-theoretic) product $\mathbb{U}_j^{\pm} V$ is
finite-dimensional. Repeated application of this fact completes the
proof of the proposition. \end{proof}

\subsection{Some derived category considerations}\label{s:derived}

We finally begin to use the assumptions and notation first given
below (\ref{e:specialization}, which the reader should review at
this point. The notations include a common notation $\mathbb{U}$ for
a quantum enveloping algebra,  specialized at an $\ell^{th}$ root of
unity, and the distribution algebra of a simply connected semisimple
algebraic group $G$.  There are similar common notations associated
to various subalgebras of $\mathbb{U}$, and distribution algebras
associated to subgroups of $G$, such as the (negative) Borel
subgroup $B$. Both $p>h$ and $\ell>h$ are required, and there are
further conditions  on $\ell$. (It must be odd, and not divisible by
3 when the root system of $\mathbb{U}$ has a component of type G2).

\smallskip
 In addition, we introduce here the notations
$\mathcal{C}_{\mathbb{U}}$, $\mathcal{C}_{\mathbb{B}},\ldots$ for
the categories of Type 1 integrable $\mathbb{U},
\mathbb{B},\ldots$-modules, respectively, in the quantum case. In
the algebraic groups case, the same notations reference the
categories of rational $G,B,\ldots$-modules, respectively. These
latter categories may be rewritten, according to our conventions for
naming distribution algebras, as the categories of rational
$\mathbb{U}, \mathbb{B},\ldots$-modules. Here, ``locally finite"
would be a more accurate term than ``rational," but we will use
either term in unambiguous contexts.

\medskip
This section provides a starting point for the proof of the
Induction Theorems \ref{T:induct} and \ref{T:Ginduct}, the latter as
reformulated in Theorem 2.1. The result below, a corollary of the
those in the previous subsection, is the starting point. The
statement and proof work in both the quantum and algebraic groups
context, in the notation discussed above.

By $block(\mathbb{U})$ we mean the category of finite dimensional
modules in the principal block of $\mathcal{C}_{\mathbb{U}}$;
equivalently, it is the full subcategory of all finite-dimensional
modules whose composition factors have highest weights in $W_\ell
\cdot 0$ (taking $\ell = p$ in the algebraic groups case).

By $D^b block(\mathbb{U})$ we mean the bounded derived category of
the abelian category  $block(\mathbb{U})$, as defined by
Verdier--see, for example, \cite[Chapter I]{Ha}. Let
$D^b_{block(\mathbb{U})}(\mathcal{C}_{\mathbb{U}})$ denote the full
subcategory of $D^b(\mathcal{C}_{\mathbb{U}})$ consisting of objects
which have each of their (finitely many) cohomology groups in
$block(\mathbb{U})$. Using the local finiteness of rational modules
(see Proposition \ref{p:locfinite} in the quantum case), we observe
the following lemma.

\begin{lem}\label{l:fullembedding} The natural map $D^b
block(\mathbb{U})\to D^b(\mathcal{C}_{\mathbb{U}})$, arising from
the inclusion functor at the abelian category level, induces an
equivalence
$$D^b block(\mathbb{U})\cong
D^b_{block(\mathbb{U})}(\mathcal{C}_{\mathbb{U}})$$.
\end{lem}
\begin{proof} Let $K^{\bullet}$ be a bounded complex of objects in
$\mathcal{C}_\mathbb{U}$ with each cohomology group belonging to
$block(\mathbb{U})$. We claim there is a bounded subcomplex
$F^{\bullet}$ of $K^{\bullet}$, with finite dimensional objects in
$\mathcal{C}_\mathbb{U}$ in each degree, such that the inclusion map
$F^{\bullet}\subseteq K^{\bullet}$ is  a quasi-isomorphism. To
construct the subcomplex $F^{\bullet}$, we may assume, inductively,
its terms in all degrees $\geq i$ are constructed, so that they form
a subcomplex $F^{\geq i}$. In addition, we require, inductively,
that inclusion of this complex into $K^{\bullet}$ induces an
isomorphism on cohomoology in grades $>i$ and an epimorphism on the
$i^{th}$ cohomology groups. Then, we wish to construct $F^{i-1}
\subseteq K^{i-1}$ so that the resulting complex $F^{\geq i-1}$ has
the analogous properties for $i-1$ in place of $i$. Let $\delta$
denote the differential $K^{i-1}\to K^{i}$. Then
$\delta(K^{i-1})\cap F^i$ is, of course, both finite and contained
in the image of $\delta$. Choose a finite dimensional subspace $E$
of $K^{i-1}$ such that $\delta(E)=\delta(K^{i-1})\cap F^i$, Also
choose a finite dimensional subspace $E'$ of the Kernel of $\delta$
such that the image of $E'$ in the natural surjection  $Ker ~\delta
\to H^{i-1}(K^{\bullet})$ is all of $H^{i-1}(K^{\bullet})$. Take
$F^{i-1}$ to be the $\mathbb{U}$-module generated by $E+E'$. Then
$\delta(F^{i-1})=\delta(K^{i-1})\cap F^i$. Consequently, inclusion
induces a monomorphism $H^i(F^{\geq (i-1)})\to H^i(K^{\bullet})$.
The (downward) induction hypothesis  implies the same map is  a
surjection, so it must be an isomorphism. Our construction of
$F^{i-1}$ gives a surjection of $H^{i-1}$ for the same inclusion of
complexes. The inductive step may be repeated, eventually reaching
cohomological degrees $j$ where $K^j$ and all lower degree terms are
$0$. At that point we may take $F^j$ also zero, and zero in lower
degrees. This gives that $F^{\bullet}\subseteq K^{\bullet}$ induces
an isomorphism on all cohomology groups. That is, it induces a
quasi-isomorphism, as required in the claim.

We remark further, that, by taking block projections, the complex
$F^{\bullet}$ may be assumed to consist in each degree of objects in
$block(\mathbb{U}$.

The lemma proposes that the natural map induces an equivalence. It
follows from the claim and remark that every object on the
right-hand side of the proposed equivalence is, indeed, in the
strict image of the left hand side. It remains to show the natural
map induces a full embedding at the derived category morphism level.
For this, observe the claim above can be strengthened so that the
constructed complex $F^{\bullet}$ contains any given finite
dimensional subcomplex $N^{\bullet}$ of $K^{\bullet}$. (Strengthen
the induction hypothesis in the proof by adding the assertion
$N^{\geq i}\subseteq F^{\geq i}$. Then, at the inductive step,
replace $E$ by $E+N^{i-1}$; this does not effect $\delta(E)$, since
$\delta(N^{i-1})\subseteq \delta(K^{i-1})\cap F_i$.) As before, if
$N^{\bullet}$ is a complex of objects in the principal block, we may
assume the complex $F^{\bullet}$ constructed is also a complex of
objects in the principal block.

Taking the same idea yet another step further, we can even assume
$F^{\bullet}$ contains any given finite number of subcomplexes like
$N^{\bullet}$, since the sum of any number of finite subcomplexes of
$K^{\bullet}$ is again a finite subcomplex.

  Now use the standard direct limit constructions (in the second variable)
  of derived category morphisms. Here we mean the Verdier
  localization construction of derived categories (and bounded derived categories), which proceeds (first) by
  localization of homotopy categories of complexes.  (See \cite[pp.32,37]{Ha}.) In particular any morphism on
  the right hand side from an object $M_1^{\bullet}$ to $M_2^{\bullet}$ is represented by an object $K^{\bullet}$
  and two morphisms $M_1^{\bullet}\to K^{\bullet}$ and $M_2^{\bullet}\to K^{\bullet}$, the latter a quasi-isomorphism.
  Now assume $M_1^{\bullet}$ and $M_2^{\bullet}$ are finite dimensional complexes of objects in $block(\mathbb{U}$,
  and let $N_1^{\bullet}$ and $N_2^{\bullet}$ be their respective
  images in $K^{\bullet}$.  Applying the strengthened versions of the
  claim, above, we may construct a finite dimensional complex
  $F^{\bullet}$ containing both $N_1^{\bullet}$, $N_2^{\bullet}$,
  and contained in contained in $K^{\bullet}$. Moreover, the latter
  inclusion is constructed to be a quasi-isomorphism. It follows
  that the pair of morphisms $M_1^{\bullet}\to F^{\bullet}$ and $M_2^{\bullet}\to
  F^{\bullet}$ represent a derived category morphism on the left
  hand side of the display in the lemma. This proves the
  surjectivity required in the full embedding property at the
  morphism level.

      It remains to prove injectivity. Suppose we are given a morphism on the left hand side of the display which becomes zero on the right hand side.
  The morphism on the left may be represented by the following configuration: We are given $M_1^{\bullet}$, $M_2^{\bullet}$ and $J^{\bullet}$, all
  finite dimensional complexes of objects in
  $block(\mathbb{U})$, and a pair of morphisms $M_1^{\bullet}\to J^{\bullet}$ and $M_2^{\bullet}\to
  J^{\bullet}$, the latter a quasi-isomorphism. To say that the derived category morphism represented by this configuration becomes zero,
   when considered on the right hand side,
  means the following: There is a complex $K^{\bullet}$ of objects in $\mathcal{C}_{\mathbb{U}}$ and a quasi-isomorphism $J^{\bullet}\to K^{\bullet}$
  such that the composite map of complexes
  $M_1^{\bullet}\to J^{\bullet}\to K^{\bullet}$ is homotopy equivalent to zero.
  Let $h=\{h_i\}_{i\in\mathbb{Z}}$ be a family of maps defining the
  homotopy in question. That is, each $h_i:M_1^i\to K^{i-1}$ is a morphism in
  $\mathcal{C}_{\mathbb{U}}$, and $\delta_{K^{\bullet}}\circ h + h\circ \delta_{M_1^{\bullet}}$ is
  the given map $M_1^{\bullet}\to J^{\bullet}\to K^{\bullet}$. Here
  the subscripted symbols $\delta$ denote the evident families of
  differentials. Observe that the sum $L^{\bullet}$ over $i$ of all $\delta_{K^{i-1}}\circ h_i(M_1^i) +
  h_i(M_1^i)$ is a finite dimensional subcomplex of $K^{\bullet}$,
  with all of its objects and differentials in $block(\mathbb{U})$
  Using the extended claims above, we can construct a finite
  dimensional $block(\mathbb{U})$-complex $F^{\bullet}$, contained in $K^{\bullet}$ as a $\mathcal{C}_{\mathbb{U}}$-subcomplex,
  and itself containing each of
  $L^{\bullet}$, the image of $J^{\bullet}$ in $K^{\bullet}$, and the image of $M_1^{\bullet}$ in $K^{\bullet}$ (already in $L^{\bullet}$, actually).
  In addition, the above constructions allow us to assume that that
  the inclusion $F^{\bullet}\subseteq K^{\bullet}$ is a
  quasi-isomorphism. It follows that $J^{\bullet}\to F^{\bullet}$ is
  a quasi-isomorphism. Consequently, the derived category morphism
  (viewed as a direct limit)
  represented by the original configuration is also represented by the
  pair of maps $M_1^{\bullet}\to J^{\bullet}\to F^{\bullet}$ and $M_2^{\bullet}\to
  J^{\bullet}\to F^{\bullet}$. However, the map $M_1^{\bullet}\to
  J^{\bullet}\to F^{\bullet}$ is visibly homotopic to zero, using
  the same function $h$ to define the required homotopy. (By
  construction $h_i(M_1^i)\subseteq F^{i-1}$ for each $i$.) ~Thus, the morphism
  represented by the original configuration is zero in its
  associated direct limit. This proves the required injectivity and
  completes the proof of the lemma.

\end{proof}

 It is suggested below \cite[Defn. 3.5.6]{A} that, in the quantum
 case, $block(\mathbb{U})$ is ``known" to have enough injectives.
 There is such a result about injectives in \cite{APW}. But the contest, while
 possibly too restrictive,
 ostensibly
 applies only to the cases where $\ell$ is a prime power, as do the
 discussions in \cite{APW2}. Our argument above does not depend on
 such a property, and, indeed, applies to the algebraic groups case,
 where there are no finite-dimensional injectives.

\medskip
Though it is somewhat informal, we henceforth identify $D^b
block(\mathbb{U})$ and
$D^b_{block(\mathbb{U})}(\mathcal{C}_{\mathbb{U}})$ through the
isomorphism above. This language is used in the result below.

 %Add more on proof and/or further comments.
\begin{cor} \label{c:nBWcor}
\begin{itemize}
\item[(i)] For any $\lambda \in \mathbb{Y},$
$\RInd_{\mathbb{B}}^{\mathbb{U}}(l \lambda)\in D^b
block(\mathbb{U}).$
\item[(ii)] The category $D^b block(\mathbb{U})$ is generated, as a
triangulated category, by the family of objects\\
$\{\RInd_{\mathbb{B}}^{\mathbb{U}}(l \lambda)\}_{\lambda\in
\mathbb{Y}}.$
\end{itemize}
\end{cor}

\begin{proof} The linkage principle, Corollary \ref{c:qcohom} and Proposition
\ref{p:finiteness} imply item (i) above. Corollary \ref{c:qcohom},
used with induction on weights and standard cohomological degree
truncation operators \cite[p.29]{BBD}, implies item (ii).
\end{proof}

\begin{rems}\label{r:Kempf}
(a)The result above is stated as Cor. 3.5.2 in \cite{ABG} in their
quantum enveloping algebra set-up. Their proof, overall, relies on
similar considerations, though some of the references supplied to
\cite{APW} for their preparatory lemma \cite[Lem. 3.5.1]{ABG} are
inaccurate. It is not clear if their proof applies to the case where
$\ell$ is not a prime power.

(b) This is perhaps a good point to mention that Kempf's vanishing
theorem, well-known in the algebraic groups case \cite[II,\S 4]{J},
also holds \cite[Thm. 5.3]{AW} in the quantum case under the
hypotheses of this section. Thus, the higher derived functors
$R^n\Ind_{\mathbb{B}}^{\mathbb{U}}(\mu)$ are zero for $\mu$ dominant
and $n>0$. The generalized tensor identities \cite[I,Prop.4.8]{J}
also work here \cite[Prop.4.7]{AW}. These results are stated using
individual higher derived functors $R^n$ in each degree $n\geq 0$,
but their proofs show that there are isomomorphisms
$\RInd_{\mathbb{B}}^{\mathbb{U}}(M\otimes N|_{\mathbb{B}})\cong
\RInd_{\mathbb{B}}^{\mathbb{U}}(M)\otimes N$ whenever $M\in
\mathcal{C}_{\mathbb{B}}$ and $N\in \mathcal{C}_{\mathbb{U}}$. These
isomorphisms may also be deduced from the natural maps in (d) below,
applied with $M$ replaced by an injective resolution.

(c) The roles of left and right in the tensor identity may be
reversed. (See (d) below. The first argument for such a reversal is
probably that for \cite[Prop. 2.7]{APW}.) Also, although the quantum
algebras we deal with are not generally cocommutative
($\mathbb{U}^0$ being an exception), the orders of tensor products
of integrable modules we deal with can often be interchanged (up to
isomorphism). This holds in particular for tensor products of finite
dimensional modules in $\mathcal{C}_{\mathbb{U}}$. See
\cite[32.16]{L}. We have, however, not investigated the naturality
properties this reversal may or may not have. The reversal is
natural in the tensor identity case, as can be seen from (d) below.
If $M$ there is then replaced by a complex of injective modules, a
natural reversal is obtained in the generalized tensor identity
case.

(d) It is sometimes useful to have explicit natural isomorphisms

$$\alpha_{_{M,N}}: \Ind_\bB^\bU (M\otimes N|_\bB) \lr \Ind_\bB^\bU (M)\otimes N; ~{\rm and}    $$
$$\gamma_{_{N,M}}: \Ind_\bB^\bU (N|_\bB\otimes M) \lr N\otimes \Ind_\bB^\bU (M)    $$
where $M\in \calC _\bB, N\in \calC_\bU.$ We give such natural isomorphisms for the convenience of the reader:\\
Drop the subscripts $M,N$ and regard both modules in the top row as
a contained in $\Hom_k (\bU, M\otimes N). $

We have, for $f\in \Ind_\bB^\bU(M\otimes N|_\bB), ~ x \in \bU$,
$$\alpha(f)(x)=(1\otimes S(x_2))f(x_1)$$
in (implicit sum) Sweedler notation. (Thus $\Delta(x)=x_1\otimes
x_2$, a sum over an invisible implicit index shared by $x_1,~ x_2$.)
To check $\bB$-equivariance of $\alpha(f)$, let $h\in \bB.$ Then
$$ \begin{tabular}{lll}
$\alpha(f)(hx)$ & = & $(1\otimes S((hx)_2))f((hx)_1)$ \\

& = & $(1\otimes S(x_2)S(h_2))f(h_1x_1) $,

\end{tabular}$$
where the last line is a sum over two implicit and independent
indices, one for the $x's$ and one for the $h's$. Continuing, we
obtain further similar expressions

$$ \begin{tabular}{lll}
\hspace*{1.1in} &  = & $(1\otimes S(x_2)S(h_3))(h_1\otimes h_2)f(x_1) $\\
           &  = & $(1\otimes S(x_2))(h_1\otimes S(h_3h_2)f(x_1) $\\
                    &  = & $(1\otimes S(x_2))(h_1\otimes S\epsilon(h_2))f(x_1) $\\
                    &  = & $(1\otimes S(x_2))(h_1\otimes 1)f(x_1) $\\
                        &  = & $(h\otimes 1)(1\otimes S(x_2))f(x_1) $ \\
                    &  = & $(h\otimes 1)\alpha(x) $
\end{tabular}$$
which is the desired equivariance.  Notice the top line is, for any
fixed $x_1, ~ x_2$, the image of $h\in \bB$ under a linear map $\bB
\stackrel{\bigtriangleup}{\lr} \bB\otimes \bB \otimes \bB \lr M.$
Here $\bigtriangleup$ denotes, with some abuse of notation, the map
usually denoted $(1\otimes \bigtriangleup)\circ \bigtriangleup$ or
$(\bigtriangleup \otimes 1)\circ \bigtriangleup$, with
$\bigtriangleup :\bB\lr \bB\otimes\bB$ the comultiplication. We
write both $\Delta(h)=h_1\otimes h_2$ and $\Delta(h)=h_1\otimes
h_2\otimes h_3$, depending on context. The inverse $\beta$ of
$\alpha$ is given, for $g\in \Ind_\bB^\bU(M)\otimes N \subseteq
\Hom_{\mathbb{k}}(\bU, M\otimes N), ~ x\in \bU$, by
$$\beta(g)(x)=(1\otimes x_2)g(x_1).$$

We leave it to the reader to check that $\beta(g)$ satisfies the
appropriate $\bB$-equivariance (by an argument similar in spirit to
that for $\alpha$), and that $\beta$ is inverse to $\alpha$.  The
formula for $\gamma$, is for $f\in \Ind_\bB^\bU(N|_\bB \otimes M)
\subseteq \Hom_{mathbb{k}}(\bU, N\otimes M), ~ x\in \bU$,

$$\gamma(f)(x) =(S(x_1)\otimes1)f(x_2)$$

For the inverse $\delta$ of $\gamma$ it is, for $g \in N\otimes \Ind
_\bB^\bU(M) \subseteq \Hom_{mathbb{k}}(\bU, N\otimes M), ~x \in
\bU,$

$$\delta(g) =(x_1 \otimes 1)g(x_2)$$

Again, the reader may check, with arguments similar in spirit to
those illustrated, that $\gamma$ amd $\delta$ satisfy the
appropriate equivariance properties and are inverse to each other.

(e) We remark that $Ind_\bB^\bU(M)$ is equipped with a natural
``counit" $\epsilon_M:Ind_\bB^\bU(M)|_{\mathbb{B}}\lr M$ which, as
is well-known (see Wikipedia) may be used its property of being
right adjoint to restriction. In the full module categories for
$\bU$ and $\bB$ $\epsilon_M$ may be given as evaluation at 1 on the
right adjoint $\Hom_\bB(\bU,M)$, and it follows that $\epsilon_M$
may be similarly interepreted for $\Ind_\bB^\bU(M)$ when dealing
with integrable modules. We only want to observe here that there is
a similar ``evaluation at 1" counit for each of the modules
$Ind_\bB^\bU (M\otimes N|_\bB), \Ind_\bB^\bU (M)\otimes
N),\Ind_\bB^\bU (N|_\bB\otimes M), N\otimes \Ind_\bB^\bU (M)$ above,
providing each of these constructions with the structure of a right
adjoint to restriction. The proof is easy, noting the isomorphisms
in (d) commute with evaluation at 1 on the ambient $\Hom_\bB(\bU,-)$
module. Rewriting this fact in the $\epsilon$ notation, we have that
$\epsilon_M\otimes N|_\bB$ and $N|_\bB\otimes \epsilon_M$ are
counits (that is, provide a right adjoint structure) for
$\Ind_\bB^\bU (M)\otimes N),~ N\otimes \Ind_\bB^\bU (M)$,
respectively. We will use this fact in Appendix B.

(f)  Finally, we explain briefly how the induction functors we have
used above, based on the formalism in \cite{A} and compatible with
\cite{APW}, fit with the algebraic groups formalism in
\cite[I,3.3]{J}. Actually, the original definition \cite[\S
1]{CPS77} of induction $\Ind_{B}^{G}(M)$ in the algebraic groups
case, for a finite-dimensional rational $G$-module $M$, was the set
$\text{Morph}_B(G,M)$ of $B$-equivariant morphisms from $G$ to $M$,
with an evident direct union used for a general rational module $M$.
This definition is formally quite close to the \cite{A} definition
in the quantum case. Using Sullivan's theorem \cite[Thm.
6.8]{CPS80}, that all locally finite $\text{Dist}(G)$ modules are
rational, it is easy to see that this definition coincides with the
definition of \cite{A} used above, with $\text{Dist(G}$ in the role
of $\bU$ and $\text{Dist(G)}$ in the role of $\bU$. Finally, to
connect the \cite{CPS77} definition with that of \cite{J}, simply
replace the $B\times G^{op}$ action $(b,g)x=bxg$ on $G$ with the
isomorphic action $(b,g)x=g^{-1}xb^{-1}$, where $b\in B$ and $x,g
\in G$.

\end{rems}

\bigskip

\subsection{Some Special Twisted Induced
Modules}\label{subsec:twisted}  In this subsection and the next, we
will adapt the ``uniform" notation for the quantum and positive
characteristic cases introduced in the discussion of Theorem 2.1 and
elaborated in subsection \ref{subsec:specializations}. We will
presume and utilize the definitions for the quantum Frobenius
morphism as in, e.g., \cite[Thm. 8.10]{L3}, \cite{L}. We use a
similar notation in positive characteristic, where the Frobenius
morphism originated and is well-known.

\begin{rems}\label{r:coinductio}
   In the quantum case, the Frobenius morphism is a homomorphism
$\varphi : \mathbb{U}\to Dist(G')$, where $Dist(G')$ is the
distribution algebra (over $K$) of an algebraic group $G'$
(semisimple, simply connected, and defined and split over $K$, with
the same root datum as $\mathbb{U}$). If $M$ is a rational
$G'$-module over $K$, we may twist it through $\varphi$ and obtain
an integral $\mathbb{U}$ module $^\varphi M$, trivial on
$\mathbb{u}$. We will use the notation $M^{[1]}:=~ ^{\varphi}M$, and
the same notation for twisting a module through the Frobenius in the
corresponding characteristic $p$ algebraic groups situation (where
$\mathbb{U}=Dist(G)$ is both the domain and the target of the
Frobenius homomomrphism).

Returning to the quantum case, we remind the reader of our notation
$\mathbb{p}=\mathbb{b}\cdot \mathbb{U}^{0}$ (and that this notation
is simlar to that in \cite{ABG}, except that our $\mathbb{b}$ is
associated to negative roots). The Frobenius homomoprhism  is
compatible with triangular decompositions of its domain and target;
see \cite[Thm. 8.10]{L3}. So, the above $M^{[1]}$ notation also
makes sense, if $M$ is (in the obvious analogous notation) a
rational $B'$ or $T'$-module.  This results, respectively, in an
integral $\mathbb{B}$ or $\mathbb{p}$ module $M^{[1]}$,  trivial as
a $\mathbb{b}$-module. (There is some ambiguity of notation here: if
$M$ is not obviously a $G'$-module, we deliberately do not include
all of $\mathbb{u}$ as part of the domain of definition of $M^{[1]}$
without explicit mention otherwise.) Conversely, we claim any
integrable module $N$ for $\mathbb{B}$ or $\mathbb{p}$, which is
trivial for $\mathbb{b}$, has, respectively, this form, and in a
unique way. The corresponding assertion for  $\mathbb{U}^{-}$ for
modules trivial for $\mathbb{u}^{-}$ follows from the (negative root
analogs of) \cite[Lem.s 8.8, 8.9]{L3}. (Note also from these results
that the Kernel ideal of the Frobenius homomorphism on
$\mathbb{U}^{-}$ is the left $\mathbb{U}^{-}$ ideal generated by the
augmentation ideal of $\mathbb{u}^{-}$. Similarly, the Kernel must
be the right ideal generated by this augmentation ideal.) At the
level of $\mathbb{U}^0$ for modules trivial on $\mathbb{u}^0$ it
follows from the explicit form of the Frobenius homomorphism on
$\mathbb{U}^0$ in [{\it op cit}, p.110, bottom] and the monomial
bases [{\it op cit}, Thm. 6.7(c), Thm.8.3(ii)] for $\mathbb{U}^0$
and $\mathbb{u}^0$, respectively. The claim follows.

We remark that the parenthetic note above shows an interesting
additional property: If $N$ is any integrable $\mathbb{B}$ module,
with unspecified action of $\mathbb{b}$, the largest
$\mathbb{b}$-module quotient of $N$ with trivial $\mathbb{b}$-module
action is naturally a (twisted) $\mathbb{B}$-module, call it
$M^{[1]}$. Consequently, if $E$ is any $T'$-module, then, using
rational induction for algebraic groups,
$\Hom_{\mathbb{B}}(N,\Ind_{T'}^{B'}(E)^{[1]})\cong
\Hom_{\mathbb{B}}(M^{[1]},\Ind_{T'}^{B'}(E)^{[1]})\cong
\Hom_{B'}(M,\Ind_{T'}^{B'}(E))\cong \Hom_{T'}(M|_{p'},E) \cong
\Hom_{\mathbb{p}}(N|_{\mathbb{p}},E^{[1]})$. This shows the functor
sending $E^{[1]}$  to $\Ind_{T'}^{B'}(E)^{[1]})$ serves, on
appropriate categories of twisted integrable modules, as a right
adjoint to restriction on corresponding integrable (but not
necessarily twisted) categories. In fact, a right adjoint
$\Ind_{\mathbb{p}}^{\mathbb{B}}$ on the full integrable categories
is constructed in \cite[\S 2.7]{ABG}. (There is a similar
construction in \cite{APW2}, but the set-up there ostensibly
requires $\ell$ to be a prime power.) The adjointness properties
observed above in this paragraph imply
$$\Ind_{\mathbb{p}}^{\mathbb{B}}(E^{[1]})\cong
\Ind_{T'}^{B'}(E)^{[1]},$$ for $E$ any rational $T'$-module.

A completely adequate version of the isomorphism is noted without
proof in \cite[(2.8.2)]{ABG}, presumably based on \cite[Lem.
2.6.5]{ABG}, which discusses also details of the Frobenius
homomorphism at the $\mathbb{B}$ module level. The main application
in both \cite{ABG} and this paper occurs with $E^{[1]}$ a
1-dimensional $\mathbb{p}$ module $k(\ell
\lambda):=k_{\mathbb{p}}(\ell \lambda)$, with $\lambda \in
\mathbb{Y}$.

 We note further that, once an induction
$\Ind_{\mathbb{p}}^{\mathbb{B}}$ is available (as a right adjoint to
restriction from $\mathcal{C}_{\mathbb{B}}$ to
$\mathcal{C}_{\mathbb{p}}$), an induction functor
$\Ind_{\mathbb{p}}^{\mathbb{U}}$ can be constructed as the
commposition $\Ind_{\mathbb{B}}^{\mathbb{U}}\circ
\Ind_{\mathbb{p}}^{\mathbb{B}}$.

We conclude these remarks by noting that all features of the above
paragraphs have obvious parallels that hold in the (characteristic
$p$) algebraic groups case, with $G=G'$, etc. We continue this dual
use of the notations $G',\ldots$ below.
\end{rems}
\medskip

%(analogous to working with truncated
%$TB_1$-modules in place of $B$-modules,
%for $B$ a Borel of algebraic group $G$).
Adapting \cite[\S 4.3]{ABG} to our negative Borel framework, set
\begin{equation}\label{e:I} I_{\mu}:= \Ind_{T'}^{B'}(\mu)\cong
\lim_{\overset{\longrightarrow} {\\\nu\in
\mathbb{Y}^{++}}}(V_{\nu}|_{B'}\otimes \mathbb{k}_{B'}(\mu + \nu))
\end{equation}
where  $V_{\nu}$ denotes the ``costandard" $G'$-module
$\Ind_{B'}^{G'}(-w_0\nu) = H^0(-w_0\nu)$ with highest weight
$-w_0\nu$. We take $\mu$ to be any weight in $\mathbb{X}$, though we
will only use the case $\mu\in \mathbb{Y}$ In the quantum case, $G'$
is a semsimple algebraic group in characteristic 0, so $V_{\nu}$ is
irreducible, though we will not need that to explain the isomorphism
in \ref{e:I}, which we do now:

In general, the lowest weight of $V_{\nu}$ is $-\nu$, appearing with
multiplicity 1, and $\mathbb{k}_{B'}(-\nu)$ is the $B'$-socle of
$V_{\nu}$. Let $\omega$ be in $\mathbb{Y}^{++}$, so that
$\nu+\omega$ is also in $\mathbb{Y}^{++}$, and is ``larger" than
$\nu$ if $\omega \neq 0$. This defines an evident directed system of
weights. There is a natural homomorphism of $G'$-modules
$V_{\nu}\otimes V_{\omega}\to V_{\nu + \omega}$ which is an
isomorphism on highest (and, applying $w_0$, on lowest) weight
spaces. In particular, the induced homomorphism of $B'$-modules
$V_{\nu}\otimes \mathbb{k}_{B'}(-\omega)\to V_{\nu + \omega}$ is an
isomorphism on $B'$-socles, hence injective. Tensoring on the right
with 1-dimensional modules $\mathbb{k}_{B'}(\nu + \omega +\mu)$
gives injections
$$V_{\nu}\otimes \mathbb{k}_{B'}(\nu +\mu)\to V_{\nu + \omega}\otimes
\mathbb{k}_{B'}(\nu +\omega +\mu).$$ For fixed $\mu$ these describe
the directed system underlying the direct limit in \ref{e:I}, and
shows it is a directed system of injections, all with a common socle
$\mathbb{k}_{B'}(\mu)$ and with the weight $\mu$ appearing with
multiplicity 1. In particular, the direct limit exists as a rational
$B'$ module $I$  and has the same socle $\mathbb{k}_{B'}(\mu)$.
Consequently, there is a map $I\to \Ind_{T'}^{B'}(k(\mu)=I_{\mu}$
which is an isomorphism on socles. (The induced module definition of
$I_{\mu}$ shows its only 1-dimensional submodule is
$\mathbb{k}_{B'}(\mu)$.) Thus, $I\subseteq I_{\mu}$. To get
equality, it is enough to show that, for any weight $\tau$ of
$I_{\mu}$, there is a $\nu \in \mathbb{Y}^{++}$ such that the $\tau$
weight space of $V_{\nu}\otimes \mathbb{k}_{B'}(\nu + \mu)$ has
dimension equal to that of the $\tau$ weight space of $I_{\mu}$. The
(weight space by weight space) linear dual $I_{\mu}^*$ of $I_{\mu}$
is (after conversion to a left $B'$-module) generated by its
(1-dimensional) $-\mu$ weight space, call it $\mathbb{k}v$. Thus,
$I_{\mu}^*=Dist({U'}^-)v$ may be viewed as the homomorphic image of
$M(\nu)|_{B'}\otimes \mathbb{k}_{B'}(-\nu -\mu)$, where $M(\nu)$
denotes the Verma module for $Dist(G')$ with highest weight $\nu$.
The linear dual of the injection $V_{\nu}\otimes \mathbb{k}_{B'}(\nu
+ \mu) \to I_{\mu}$ gives a surjection $I_{\mu}^*\to
\Delta(\nu)|_{B'}\otimes \mathbb{k}_{B'}(-\nu -\mu)$, where
$\Delta(\nu)$ is the Weyl module of highest weight $\nu$. Thus, we
have a composition of surjections
$$M(\nu)|_{B'}\otimes\mathbb{k}_{B'}(-\nu -\mu)\to I_{\mu}^*\to \Delta(\nu)|_{B'}\otimes
\mathbb{k}_{B'}(-\nu -\mu).$$

We want to show that the $-\tau$ weight space dimensions on the left
and right (and, thus, also in the middle) are the same for some
choice of $\nu$. This is equivalent to showing that the $\nu +\mu
-\tau$ weight spaces of the Verma and Weyl modules with highest
weight $\nu$ are the same for some $\nu$, given $\mu$ and $\tau$.
This dimension is obviously independent of the base field
$\mathbb{k}$ in the Verma module case, and the same independence is
true in the Weyl module case by \cite[II,8.3(3)]{J}. Over the
complex numbers, the Kernel of the map $M(\nu)\to \Delta(\nu)$ is
generated as a $B'$-module, by the elements $F_i^{N_i+1}v^+$, where
$v^+$ is a highest weight vector, $N_i$ is the coefficient of $\mu$
at the $i^{th}$ fundamental weight, and $F_i$ is a Chevalley basis
root vector associated with the negative of the $i^{th}$ fundamental
root $\alpha_i$. (All observed in \cite[\S 4.3]{ABG}.) It is easy to
choose all $\nu$ so that each coefficient of $\tau - \mu$ at
$\alpha_i$ is smaller than $N_i+1$. In this casee the Kernel has a
zero weight space for weight $\nu +\mu -\tau$. Thus, the dimensions
of the weight spaces for this weight are the same in both the Verma
and Weyl module, as desired. This proves $I=I_{\mu}$.

%Len wrote $\mu$ instead of $\nu$ after the equality; typo
By applying  Frobenius twists, one has
\begin{equation}\label{e:Itwist}
\IndpB(\ell\mu)\cong I_{\mu}^{[1]} = \Ind_{T'}^{B'}(\mu)^{[1]}\cong
\lim_{\overset{\longrightarrow} {\\\nu\in
\mathbb{Y}^{++}}}(V_{\nu}^{[1]}|_{\mathbb{B}}\otimes
\mathbb{k}_{\mathbb{B}}(\ell\mu + \ell\nu)).
\end{equation}

\begin{defn} For any dominant weight $\sigma$ define $J_{\sigma,\mu}=V_\sigma^{[1]}|_{\mathbb{B}} \otimes
\mathbb{k}_{\mathbb{B}}(\ell \mu + \ell \sigma)$.
\end{defn}

In the lemma below, and elsewhere in this paper, we freely use ``$\Ext^n$" for a derived category ``$\mbox{Hom}^n$."

%Lemma below suggested to be added from the March 2014 meeting.
\begin{lem} \label{lem:Len1} Let $Y$ be any finite dimensional $\mathbb{B}$-module, and $\mu$ any weight in
$\mathbb{Y}$ (or $\mathbb{X}$). Then, for sufficiently large $\sigma,$ we have
 \begin{itemize}
\item[(a)] $\Ext^n_{\mathbb{B}}(Y,J_{\sigma,\mu}) \cong \Ext^n_{\mathbb{B}}(Y,I_{\mu}^{[1]}),$ and
\item[(b)]  $\Ext^n_{\mathbb{B}}(\RInd_{\mathbb{B}}^{\mathbb{U}}(Y),J_{\sigma,\mu})\cong \Ext_{\mathbb{B}}^n(\RInd_{\mathbb{B}}^{\mathbb{U}}(Y),I_{\mu}^{[1]})$
\end{itemize} for all nonnegative integers $n.$
    In addition we have (independently of $\sigma$)
%But what do we mean here, since there's no $\sigma$ in sight below??
\begin{itemize}
\item[(c)] $\Ext^n_{\mathbb{B}}(Y,I_{\mu}^{[1]})\cong \Ext^n_{\mathbb{p}}(Y,\mathbb{k}_{\mathbb{B}}(\ell \mu)),$ %Len had typo: k_b, not %k_{\mathbb{B}}
and
\item[(d)] $\Ext^n_{\mathbb{B}}(\RInd_{\mathbb{B}}^{\mathbb{U}}(Y),I_{\mu}^{[1]})\cong \Ext^n_{\mathbb{U}}(RInd_{\mathbb{B}}^{\mathbb{U}}(Y),RInd_{\mathbb{B}}^{\mathbb{U}}(I_{\mu}^{[1]}))$
for all nonnegative integers $n.$
\end{itemize}
\end{lem}
\begin{proof} For part a) observe that $\Ext_{\mathbb{B}}^n(Y,\mathbb{k}_{\mathbb{B}}(\omega))=0$ unless $\omega$ is
dominated by some weight $\nu$ with the weight space $Y_{\nu} \not=
0.$  There is no such $\nu$,
%TLH: Language is confusing/vague; rewrite so the ``this" is clear. OK
if the height of $\omega$ is sufficiently large, depending only on
the finite dimensional module $Y.$ For $\sigma$ sufficiently large,
all the nonzero weight spaces in $J'_{\sigma,\mu}:=
I_{\mu}^{[1]}/J_{\sigma,\mu}$ occur for weights with a large height,
determined by $\mu$ and the choice of $\sigma.$ For such a $\sigma$
and $\mu,$ we have $\Ext_{\mathbb{B}}^n(Y,J'_{\sigma,\mu})=0$ for
all nonnegative integers $n.$ (This can be easily seen with direct
limit arguments.) Part a) follows, and part b) may be obtained by
using (modules in a finite complex representing)
$\RInd_{\mathbb{B}}^{\mathbb{U}}(Y)$ for $Y$ in part  a) to give
part b). Parts c) and d) are standard reciprocity results. (For part
(c), note that $I_{\mu}^{[1]} \cong
\Ind_{\mathbb{p}}^{\mathbb{B}}(\mathbb{k}_{\mathbb{B}}(\ell \mu))
\cong \RInd_{\mathbb{p}}^{\mathbb{B}}(\mathbb{k}_{\mathbb{B}}(\ell
\mu)).$) This completes the proof of the lemma.
\end{proof}
%TLH: I think I'd like to increase the specificity of the write-up above, e.g., statements re: weight spaces and ``standard reciprocity"

%Next result also suggested by Len from March meeting:
\begin{cor} \label{cor:Len2} Let $Y$ be a finite-dimensional $\mathbb{B}-$module.  Assume all composition factors of
$Y$ have weight $\ell \omega$ for some weight $\omega\in \mathbb{Y},$ and fix $\mu\in \mathbb{Y}.$
\begin{enumerate}
\item  If $n$ is an odd nonnegative integer, then both $\Ext_{\mathbb{B}}^n(Y,I_{\mu}^{[1]})$ and
$\Ext^n_{\mathbb{U}}(RInd_{\mathbb{B}}^{\mathbb{U}}(Y)$, $RInd_{\mathbb{B}}^{\mathbb{U}}(I_{\mu}^{[1]}))$ are zero.
\item For any nonnegative integer $n,$ if the Y-composition factor weights $\ell \omega$ all have $\omega$ of sufficiently
large height, depending only on $n$ and $\mu,$ then both
$\Ext^n_{\mathbb{B}}(Y,I_{\mu}^{[1]})$ and
$\Ext^n_{\mathbb{U}}(RInd_{\mathbb{B}}^{\mathbb{U}}(Y),RInd_{\mathbb{B}}^{\mathbb{U}}(I_{\mu}^{[1]}))$
are zero. [We refer ahead to (\ref{e:ndimtensor1}) in the proof of
this part.]
%TLH: Not good form. Include appropriate TEX-style ref for (3.4.2), and reorder results.
%What is the relevance of the ref to (3.4.2), anyway? REF STYLE FIXED Len 020216
\end{enumerate}

\end{cor}

\begin{proof}
The $\mathbb{b}-$cohomology of the trivial module is, as a
$\mathbb{B}$-module, the symmetric algebra
$S^*(\mathfrak{n}^{*[1]}).$ Here $\mathfrak{n}$ is the Lie algebra
of the ``unipotent radical'' of $\mathbb{B}^{\prime},$ with
$\mathfrak{n}^{*[1]}$ the linear dual of that Lie algebra twisted by
the Frobenius, and given cohomological degree 2. See \cite[Prop.
2.3]{AJ} in characteristic $p > h$ and
\cite[Thm. 3]{GK}  \footnote{While several references are made in the
proof of this theorem to \cite{APW} and \cite{APW2}, they are of a general
formal nature, similar to those of \cite{A} we discuss above of \ref{d:linked},
and do not requiring that $\ell$ be a prime power.  It is, however, necessary in \cite{GK}
to quote a case of Kempf's theorem, but there \cite{GK} gives a (correct) reference to
\cite[Thm. 5.3]{AW} .} for the analogous
characteristic $0$ quantum group result. This gives part (1), using
Lemma \ref{lem:Len1}(c).

Part (2) also follows (in both the quantum and algebraic group
cases), using Lemma \ref{lem:Len1}, and (\ref{e:ndimtensor1}) below.
In more detail, observe first that it is sufficient to take $Y$ of
dimension 1. Then (2.14(c) clearly gives the required vanishing of
$\Ext^n_{\mathbb{B}}(Y,I^{[1]}_\mu)$, provided the weight of $Y$,
call it $\ell \lambda$, is of sufficient large height. Note this
also gives vanishing of $\Ext^n_{\mathbb{B}}(Y,J_{\sigma,\mu})$ for
all sufficiently large $\sigma$. Next, regarding $I^{[1]}_\mu$ as
the direct union of the various modules $J_{\sigma,\mu}$, we obtain
$\Ext^n_{\mathbb{U}}(\mbox{RInd}_{\mathbb{B}}^{\mathbb{U}}(Y)$,
$\mbox{RInd}_{\mathbb{B}}^{\mathbb{U}}(I^{[1]}_\mu))$ as a direct
limit of the various
$\Ext^n_{\mathbb{U}}(\mbox{RInd}_{\mathbb{B}}^{\mathbb{U}}(Y)$,
$\mbox{RInd}_{\mathbb{B}}^{\mathbb{U}}(J_{\sigma,\mu})) =
\Ext^n_{\mathbb{U}}(\mbox{RInd}_{\mathbb{B}}^{\mathbb{U}}(Y),\mbox{
RInd}_{\mathbb{B}}^{\mathbb{U}}(V_\sigma^{[1]} \otimes
\mathbb{k}_{\mathbb{B}}(\ell \mu + \ell \sigma))) \cong
\Ext^n_{\mathbb{U}}(\mbox{RInd}_{\mathbb{B}}^{\mathbb{U}}(Y),
V_\sigma^{[1]} \otimes
\mbox{RInd}_{\mathbb{B}}^{\mathbb{U}}(\mathbb{k}_{\mathbb{B}}(\ell\mu
+ \ell \sigma))$. Applying \ref{e:ndimtensor1}  with
$\mathbb{k}_{\mathbb{B}}(\ell \lambda) = Y$ and $N$ in
\ref{e:ndimtensor1} equal to $V_\sigma$, we obtain that the
dimension of
$\Ext^n_{\mathbb{U}}(\mbox{RInd}_{\mathbb{B}}^{\mathbb{U}}(Y),
V_\sigma^{[1]} \otimes
\mbox{RInd}_{\mathbb{B}}^{\mathbb{U}}(\mathbb{k}_{\mathbb{B}}(\ell\mu
+ \ell \sigma))$ is dim $\Ext^n_{\mathbb{B}}(Y,V_\sigma^{[1]}
\otimes \mathbb{k}_{\mathbb{B}}(\ell \mu + \ell \sigma)) = \mbox{dim
} \Ext^n_{\mathbb{B}}(Y,J_{\sigma,\mu}) =0$ for large $\sigma$. This
proves part (2), and the proof of the corollary is complete.
\end{proof}

\section{A Proof of the Induction Theorems} \label{sec:inducproof}
%This section is where I'm housing the new proof; old material has been copied in after the bibliography for storage for now.
After providing an initial framework and brief outline of the steps
to be used, this section proceeds step-by-step to complete the proof
of Theorem 1 and Theorem 2, the latter treated in its equivalent
formulation, Theorem 2.1. Indeed, at this stage the proofs can be
given simultaneously, largely thanks to results in the previous
section, such as Corollary \ref{cor:Len2}. In those results, the
statements make sense and are correct in both the quantum and
positive characteristic cases, though some attention to
differences--at least to different sources--are sometimes required
in their proofs. Such differentiation is no longer necessary in the
wording of proofs in this section. Once the proof has been
completed, subsection \ref{subsec: summary} summarizes some of the
similarities and differences between our approach and that taken in
\cite{ABG}.

\subsection{Sketch of the Proof of the Induction Theorems}
The proof of the induction theorems begins, as is implicit in
\cite{ABG}, with the idea  of utilizing the following `general
nonsense' result on categorical equivalences. The proof is left to
the reader.
 %Neeman's book, perhaps? Len will check.
 %Not there--Len

\begin{lem}\label{l:ntrianequiv} %(Appears as Lem. 3.9.3 in \cite{ABG})
Let $\mathcal{A}$ and $\mathcal{B}$ be two triangulated categories,
and let $\mathcal{F}: \mathcal{A} \to \mathcal{B}$ be a morphism of
triangulated categories, that is, $\mathcal{F}$ sends triangles in
$\mathcal{A}$ to triangles in $\mathcal{B},$ and commutes with the
respective translation functors on $\mathcal{A}$ and $\mathcal{B}.$
 Then $\mathcal{F}$ is an equivalence of
triangulated categories if there is a set of objects $S$ in
$\mathcal{A},$ such that the following two conditions hold:
\begin{itemize}
\item[(A)] the minimal full triangulated subcategory of $\mathcal{A}$
containing $S$ is, up to isomorphic objects,  the whole category
$\mathcal{A},$ and likewise for $\mathcal{B}$ in place of
$\mathcal{A}$ and $\mathcal{F}(S):=\{\mathcal{F}(a)\,|\,a\in S\}$ in
place of $S$;
\item[(B)] for any objects $a,a'\in S,$ the functor $\mathcal{F}$ induces
isomorphisms
$$\Hom_{\mathcal{A}}(a,a'[k]) \mapsto
\Hom_{\mathcal{B}}(\mathcal{F}(a),\mathcal{F}(a')[k])\quad \quad
\mbox{for all } k\in \mathbb{Z}.$$
\end{itemize}
\end{lem}

This lemma will be applied with $\mathcal{A} =
D_{triv}(\mathbb{B}),$ $\mathcal{B} = D^bblock(\mathbb{U})$
\footnote{Throughout this paper we use quantum and algebraic groups
in a ``simply connected" setting. In particular this means that all
of our $\mathbb{B}$-modules, always assumed to be a direct sum of
their weight spaces, can have associated weights which are in
$\mathbb{X}$, not just $\mathbb{Y}$, as in the \cite{ABG} ``adjoint
group" setting. However, the more general $\mathbb{B}$-modules are
obviously the natural direct sum of submodules whose associated
weights belong to a fixed coset (one for each summand) of
$\mathbb{Y}$ in $\mathbb{X}$. This has the consequence that the two
versions of $D_{triv}(\mathbb{B})$ in these respective
$\mathbb{B}$-module contexts are naturally equivalent. Similar
considerations apply to $\mathbb{U}$-modules and
$block(\mathbb{U})$.}  $S = k_{\mathbb{B}}(\ell \lambda),\lambda\in
\mathbb{Y}.$ Standard arguments with (distinguished triangles
arising from) homological degree truncations (see \cite[Exemples
1.3.2]{BBD}) show that $S$ does, indeed, generate $\mathcal{A}$ as a
triangulated category.

The functor $\IndBU$ has been discussed in Section \ref{sec:back}.
It is an additive left exact functor, so, from general principles,
its right derived functor $\RIndBU$, which exists, is a morphism of
triangulated categories.
% (\cite{:}). Corollary \ref{c:nBWcor}, a
%consequence of results on linkage for cohomology (reviewed earlier
%%\footnote{CHANGE to ``earlier" if this
%material gets reordered!}
%in this paper), ensures that $\RInd_{\mathbb{B}}^{\mathbb{U}}$ is
%indeed defined on $D_{triv}(\mathbb{B})$   and has image living in
%$D^bblock(\mathbb{U}).$ Thus, the functor induced by $\RIndBU$ is a
%candidate for the morphism $\mathcal{F}$ in Lemma
%\ref{l:ntrianequiv}. Moreover,
%%\footnote{TLH: Removed
%%old statement that $\RIndBU$ is left exact so $R^{0}\IndBU \cong
%%\IndBU.$
Corollary \ref{c:nBWcor} shows that $\RIndBU(S)$ generates
$D^bblock(\mathbb{U}).$ Consequently, the starting hypotheses of
Lemma \ref{l:ntrianequiv} and its condition (A)  are met, with
$\mathcal{F}$ the restriction of $\RIndBU$ to
$D_{triv}(\mathbb{B})$.

So, to establish the induction theorems, it suffices to prove
condition (B) of Lemma \ref{l:ntrianequiv} holds, that is, $\RIndBU$
gives isomorphisms
$$\Hom_{D_{triv}(\mathbb{B})}(k_{\mathbb{B}}(\ell
\lambda),k_{\mathbb{B}}(\ell \mu)[n]) \cong
 \Hom_{D^bblock(\mathbb{U})}(\RIndBU k_{\mathbb{B}}(\ell \lambda),\RIndBU k_{\mathbb{B}}(\ell \mu)[n])
$$ for any $\lambda,\mu\in \mathbb{Y}$ and $n\in \mathbb{Z}.$
That is, it suffices prove  for all $\lambda, \mu\in \mathbb{Y}$ and
$n \geq 0$
%Why can reduce to $n \geq 0, eh?$
that applying $\RIndBU$ produces isomorphisms
\begin{equation}\label{e:nisos}
\Ext^n_{\mathbb{B}}(k_{\mathbb{B}}(\ell \lambda),k_{\mathbb{B}}(\ell
\mu)) \cong
 \Ext^n_{block(\mathbb{U})}(\RIndBU k_{\mathbb{B}}(\ell \lambda),\RIndBU k_{\mathbb{B}}(\ell
 \mu)),
 \end{equation}
where the right hand side is an `$\Ext$' computed in the sense of
hypercohomology.  To establish (\ref{e:nisos}) we will proceed as
follows, with the first step the same as in \cite{ABG}:
\begin{itemize}
\item[STEP 1:] Show
\begin{equation}\label{e:nstep1}
\dim(\Ext^n_{\mathbb{B}}(k_{\mathbb{B}}(\ell
\lambda),k_{\mathbb{B}}(\ell \mu))) =
\dim(\Ext^n_{block(\mathbb{U})}(\RIndBU k_{\mathbb{B}}(\ell
\lambda),\RIndBU k_{\mathbb{B}}(\ell
 \mu)))
 \end{equation} for any $\lambda,\mu\in \mathbb{Y}$ and $n\geq 0.$

\item[STEP 2]: Employing the equalities (\ref{e:nstep1}), show, for the twisted $\mathbb{B}-$modules
$I_{\mu}^{[1]}$ defined in (\ref{e:Itwist}), that there are
isomorphisms
\begin{equation}\label{e:nktoIiso}
\Ext_{\mathbb{B}}^n(k_{\mathbb{B}}(\ell\lambda), I_{\mu}^{[1]})
\cong
\Ext_{block(\mathbb{U})}^n(\RIndBU(k_{\mathbb{B}}(\ell\lambda)),
\IndBU(I_{\mu}^{[1]})),
\end{equation}
arising from the functoriality of $\RIndBU.$
%These isomorphisms hold
%under the assumption that the weights $\ell\lambda,\ell\mu$ lie in a
%fixed ideal of weights $\tGamma(m), m\in \mathbb{Z}^{\geq 0}$ to be
%defined, but with the property that $\cup_{m\geq 0}\tGamma(m) = {\mathbb{X}}.$

\item[STEP 3]:
%By allowing the ideals $\tGamma(m)$ to expand (that is,
%letting $m\to \infty$), and making use of the structure of the
%modules $I_{\eta}^{[1]},$
Making use of the structure of the modules $I_{\eta}^{[1]},$ recover
the desired isomorphisms (\ref{e:nisos}) from the isomorphisms
(\ref{e:nktoIiso}), and (once again) the dimension equalities
(\ref{e:nstep1}).
%Expand here, based on final form of proof of Step 3.

\end{itemize}
\subsection{Step 1 of the Proof of the Induction
Theorems}\label{subsec:step1}

We largely follow \cite{ABG} for this step, with the exception of
the proof of part (ii) of Lemma \ref{wallcrossingfunctors} below,
which is \cite[Lem.4.1.1(ii)]{ABG} in the quantum case. We give a
proof in the algebraic groups case in Appendices A and B of this
paper. However, the proof \cite[Lem.4.1.1(ii]{ABG} given in
\cite{ABG} is not nearly adequate, in our view, even in the quantum
case, and we point out in appendices A and B how our proofs there
apply to complete it. As mentioned in the introduction, we are
grateful to P. Achar and S. Riche for a suggestion to the effect
that we look more closely at the \cite{ABG} proof of this result.
(In a very preliminary version of this paper, we had assumed the
proof in \cite{ABG} was sufficient in the quantum case, and even
that it applied to the modular case.) We also thank S. Riche for
pointing out an error in our first naive attempt at a correction.

 The result Lemma \ref{wallcrossingfunctors}(ii) below is actually
quite strong, in either the algebraic groups or quantum case, and
gives, in the regular weight case, a categorification of Rickard's
theorem \cite[Thm. 2,1]{R94} on derived equivalences arising from
translations. (Essentially, the latter theorem gives Lemma
\ref{wallcrossingfunctors}(i), when the derived equivalences
involved in the statement of the theorem are identified in its
proof. But the theorem only claims a version of Lemma
\ref{wallcrossingfunctors}(ii) at a character-theoretic level.)

We introduce the lemma by observing, as in the discussion above
\cite[Lem. 4.1.1]{ABG} that translation functors may constructed as
in the algebraic groups case. This is carried out in \cite[\S
9]{APW}, though with the explicit assumption that $\ell$ be a prime
power. This may be removed by appealing to Theorem 3 above. It is
easy to check that the resulting constructions have the familiar
adjointness and exactness properties of the algebraic groups case
\cite[II, Lem. 7.6]{J}. Continuing the discussion in \cite{ABG}, let
$\Xi_\alpha:block(\mathbb U)\longrightarrow block(\mathbb U)$ denote
a composition of  translation functors first `to the wall' labelled
assoiated to a simple reflection $s_\alpha$ and, then back `out of
the wall'. There are canonical adjunction morphisms
$f:id\longrightarrow\Xi_\alpha$,and $g:\Xi_\alpha\longrightarrow
id$. It is noted in \cite{ABG} that the mapping cone, $C(f)$, of $f$
gives rise to a triangulated functor from $D^block(\mathbb{U})$ to
itself, denoted $\theta^{+}_\alpha$. A similar construction (using
$C(g)[-1]$) gives a triangulted functor $\theta^{-}_\alpha$. These
constructions carry over easily to the algebraic groups case. The
statement below is \cite[Lem. 4.1.1]{ABG} in both the quantum and
algebraic groups cases.
\begin{lem}\label{wallcrossingfunctors} With the notation discussed above, we have in
both the quantum and algebric groups cases: \\

(i) In $D^b block(\mathbb{U})$ we have the following canonical
isopmorphisms $ \theta_\alpha^+\circ \theta_\alpha^- \cong id ~~
{\rm and }~~\theta_\alpha^-\circ \theta_\alpha^+ \cong id$. In
particular $\theta _\alpha ^+$ and $\theta_\alpha^-$ are
autoequivalences.

(ii) If $\lambda \in W_{aff} \cdot 0$ and $\lambda
^{s_\alpha}>\lambda$ then $\theta ^+_\alpha(\RIndBU\lambda)\cong
\RIndBU (\lambda^{s_\alpha})$.
\end{lem}

\begin{rem} As discussed above, the proof of part (ii) is given in
appendices A and B. The proof of part (i) may be obtained from
the argument for \cite[Thm. 2.1]{R94}, or, alternately, from the
argument for \cite[Lem. 4.1.1(i)]{ABG}. (Both arguments involve
similar ingredients.) We mention that \cite{ABG} defines both
$\theta^+$ and $\theta^-$ as mapping cones. The reader should be
aware that the natural definition of $\theta^-$ is as a shifted
mapping cone, as in the description given here, to obtain property
$(i)$.
%[CAN JUSTIFY THIS REMARK WITH AN EXAMPLE IF NECESSARY]
\end{rem}
\begin{lem} \label{l:ndimlem} For any $\lambda,\mu\in R$ and $n\geq 0,$
\begin{equation}
\dim(\Ext^n_{\mathbb{B}}(k_{\mathbb{B}}(\ell
\lambda),k_{\mathbb{B}}(\ell \mu))) =
\dim(\Ext^n_{block(\mathbb{U})}(\RIndBU k_{\mathbb{B}}(\ell
\lambda),\RIndBU k_{\mathbb{B}}(\ell\mu)))
\end{equation}
\end{lem}

\begin{proof} The proof follows \cite[proof of Lem. 4.2.2]{ABG}. We
include some details for completeness.
 First, the identity
$$\Hom_{\mathbb{B}}(k_{\mathbb{B}}(0), M)\stackrel{adjunction}{=}\Hom_{\mathbb {U}}(k_{\mathbb{U}}(0),\RIndBU M) \stackrel{BWB}{=}\Hom_{\mathbb{U}}(\RIndBU k_{\mathbb{B}}(0),\RIndBU M)$$
is established, using Borel Weil Bott (BWB) type theory. In fact,
Corollary \ref{c:qcohom} is sufficient in the quantum case, and the
better known algebraic groups case of that corollary  is discussed
just above it. To summarize,  the identity above holds in both
cases, for (at least) any object $M$ in $D_{triv}(\mathbb{B})$. This
proves the lemma in the special case $\lambda =0$ and arbitrary $\mu
\in \mathbb{Y}.$

The general case will be reduced to the special case by means of
translation functors. For any $\lambda, \mu \in \mathbb{Y}$ and $\nu
\in \mathbb{Y}^+$,
we claim (*)\\
$\RHom_{block(\mathbb{U})}(\RIndBU k_{\mathbb{B}}(\ell\lambda),
\RIndBU k_{\mathbb{B}}(\ell\mu)) \cong
\RHom_{block(\mathbb{U})}(\RIndBU
k_{\mathbb{B}}(\ell\lambda+\ell\nu), \RIndBU
k_{\mathbb{B}}(\ell\mu+\ell\nu))
$.\\
  Let $\nu=s_{\alpha_1}s_{\alpha_2}\cdots s_{\alpha_r} \in {\mathbb Y}\subset W_{aff}$ be a reduced expression.  Then $\ell \nu = 0^{s_{\alpha_1}s_{\alpha_2}\cdots s_{\alpha_r}}$ and hence $\ell \lambda+\ell\nu= (\ell \lambda )^{s_{\alpha_1}s_{\alpha_2}\cdots s_{\alpha_r}}>
(\ell \lambda )^{s_{\alpha_1}s_{\alpha_2}\cdots
s_{\alpha_{r-1}}}>\cdots > \ell\lambda$.  Now the repeated use of
the second part of Lemma \ref{wallcrossingfunctors} give us $\RIndBU
k_{\mathbb{B}}
(\ell\lambda+\ell\nu)\cong\theta_{\alpha_r}^+\circ\theta_{\alpha_{r-1}}^+
\circ \cdots \circ \theta_{\alpha_1}^+(\RIndBU
k_{\mathbb{B}}(\ell\lambda)).$ This together with the first property
of Lemma \ref{wallcrossingfunctors} gives (*).

For any $\lambda$, $\mu$ $\in \mathbb{Y}$, choose a large $\nu\in
\mathbb{Y}^+$ such that $\nu-\lambda \in \mathbb{Y}^+$. Using
$\nu-\lambda$ in place of $\nu$ in (*) (i.e. a shift by
$\nu-\lambda$), we get

$\RHom_{block(\mathbb{U})}(\RIndBU k_{\mathbb{B}}(\ell\lambda),
\RIndBU k_{\mathbb{B}}(\ell\mu)) \cong
\RHom_{block(\mathbb{U})}(\RIndBU k_{\mathbb{B}}(\ell\nu), \RIndBU
k_{\mathbb{B}}(\ell\mu+\ell\nu-\ell\lambda)) $ (**).

Again, a shift by $\nu$ gives

$\RHom_{block{\mathbb(U)}}(\RIndBU k_{\mathbb{B}}(0), \RIndBU
k_{\mathbb{B}}(\ell\mu-\ell\lambda)) \cong
\RHom_{block(\mathbb{U})}(\RIndBU k_{\mathbb{B}}(\ell\nu),\RIndBU
k_{\mathbb{B}}(\ell\mu+\ell\nu-\ell\lambda)) $ (**).

Notice the right hand terms of the two isomorphisms labeled (**) are
the same, so that we can view (**) as providing isomorphisms of
three expressions, all obtained by applying
$R\Hom_{block(\mathbb{U})}$. Passing, for any fixed $i$, to
$R^i\Hom_{block(\mathbb{U})}$ expressions, we obtain three vector
spaces of the same dimension. One of these vectors spaces
$R^i\Hom_{block{\mathbb(U)}}(\RIndBU k_{\mathbb{B}}(0), \RIndBU
k_{\mathbb{B}}(\ell\mu-\ell\lambda))$ is isomorphic to
$R^i\Hom_{\mathbb B}(k_\mathbb {B}(0), k_{\mathbb {B}}
(\ell\mu-\ell\lambda))$, by the special case above with
$(\ell\mu-\ell\lambda))$ used in the role of $\ell\mu$. Finally,
using the isomorphism $\RHom _{\mathbb B}(k_{\mathbb {B}}(0),
k_{\mathbb {B}}(\ell\mu-\ell\lambda)) \cong \RHom _{\mathbb
B}(k_{\mathbb {B}}(\ell\lambda), k_{\mathbb {B}}(\ell\mu))$ we
complete the proof of the lemma.
\end{proof}

\iffalse Recall the functors $\theta^+_{\alpha}:D^b
block(\mathbb{U}) \to D^b block(\mathbb{U})$, obtained from the
wall-crossing functors\footnote{i.e., composites of translation
functors ``in to" and then ``out from" a wall, see \cite[II,
7.21]{J}, and analogously defined for the quantum groups case using
translation functors given in \cite[Section 7]{APW}.
$\Theta_\alpha^+$ is constructed as the mapping cone (functor) of
the natural adjunction map from the identity functor to the
wall-crossing functor.} from the proof of Step I. \fi

Observe that for any finite-dimensional $G$-module $V,$ with weights
in $\mathbb{Y},$
$$
\theta^+_{\alpha}(M\otimes V^{[1]}) = \theta^+_{\alpha}(M)\otimes
V^{[1]}.
$$
Thus, using the proof of Step 1 and recalling the generallized
tensor identity (see Remark \ref{r:Kempf}(ii), tensoring on the
right the first (resp., second) component of each term appearing in
the equality given by Lemma \ref{l:ndimlem}, by any
finite-dimensional twisted $G-$module $M^{[1]}$ (resp., $N^{[1]}$)
with weights in $\mathbb{Y}$, preserves the equality of dimensions:
\begin{eqnarray}\label{e:ndimtensor1}
\lefteqn{\hspace*{-20pt}
\dim(\Ext^n_{\mathbb{B}}(M^{[1]}|_{\mathbb{B}}\otimes
k_{\mathbb{B}}(\ell \lambda), k_{\mathbb{B}}(\ell \mu)))}  \\[2mm]
 &=&
\dim(\Ext^n_{block(\mathbb{U})}(M^{[1]}\otimes\RIndBU
k_{\mathbb{B}}(\ell \lambda),\RIndBU k_{\mathbb{B}}(\ell
 \mu))), \nonumber
 \end{eqnarray}
and
\begin{eqnarray}
\lefteqn{\hspace*{-20pt}
\dim(\Ext^n_{\mathbb{B}}(k_{\mathbb{B}}(\ell \lambda),
N^{[1]}|_{\mathbb{B}}\otimes k_{\mathbb{B}}(\ell \mu)))
} \nonumber \\[2mm]
 &=& \dim(\Ext^n_{block(\mathbb{U})}(\RIndBU k_{\mathbb{B}}(\ell
\lambda),N^{[1]}\otimes \RIndBU k_{\mathbb{B}}(\ell
 \mu))).\notag
 \end{eqnarray}

\subsection{Step 2 of the Proof of the Induction
Theorems}\label{subsec:step2}
\begin{lem}\label{l:M}

Let $\lambda$ be any element of $\mathbb{Y}$ (or $\mathbb{X}$).
Choose $\nu = N\rho,$  with $\rho$ as in section
\ref{subsubsec:dot}, and with $N\in\mathbb{N}$ large enough so that
$\ell\tau:=\ell(\nu + \lambda)$ is dominant. Let $V_{\nu}$ be as in
(\ref{e:I}). Then the $\mathbb{B}$-module $M= V_{\nu}^{[1]}\otimes
\kB(\ell\tau)$ satisfies the following three properties:
\begin{enumerate}
\item $\kB(\ell\lambda)\subset M.$
\item All composition factors of
$M/\kB(\ell\lambda)$ have the form $\kB(\ell\eta))$ with $\eta >
\lambda$ in the dominance order.
\item The map $\RIndBU (M) \to M$ is $\mathbb{p}$-split.
\end{enumerate}
\end{lem}

\begin{proof} Once again, the argument uses ideas from \cite{ABG}, especially in the analysis of the map in part (3).

 Note $M \cong M_0^{[1]},$ where $M_0$ is the $B'-$module $V_\nu \otimes k_{B'}(\tau).$ The $B'-$socle of $V_{\nu}$ is $k_{B'}(-\nu),$ so
 $k_{B'}(\lambda) \cong k_{B'}(-\nu)\otimes k_{B'}(\tau)$ is the $B'-$socle of $M_0.$ Parts 1) and 2) of the lemma for $M$ follow immediately
 from corresponding properties of $M_0.$

    Next, put $F_0 = \IndBU(M_0)\cong V_\nu \otimes \IndBU(k_{\mathbb{B}}(\tau))$ and $F=F_0^{[1]}.$ The natural $\mathbb{B}-$map $\varphi_0: F_0\to M_0$ is surjective, as follows from the surjectivity of $\IndBU(k_{\mathbb{B}}(\tau) \to  k_{\mathbb{B}}(\tau).$

Consequently, the Frobenius twisted map $\varphi: F \to M$ is
surjective. It is also $\mathbb{p}-$split, since both domain and
range are completely reducible as $\mathbb{p}-$modules. The map
$\varphi$ gives rise (by adjointness) to a map $F\to \RIndBU(M)$
which, when composed with the natural map $\RIndBU(M)\to M$, is the
$\mathbb{p}-$split surjection $\varphi.$  Consequently,
$\RIndBU(M)\to M$ is also $\mathbb{p}-$split. This proves Property
(3) and completes the proof of Lemma \ref{l:M}. \end{proof}

Our arguments now begin to diverge from \cite{ABG}.
\smallskip
\begin{lem}\label{l:YImuExts}
Suppose $\mu \in \mathbb{Y}$ and $Y$ is a finite-dimensional
$\bB$-module all of whose composition factors are of the form
factors $\kB(\ell\lambda)$ with $\lambda \in \mathbb{Y}$).
%with weights $\ell\eta.$
%Be more specific on the weights? Should be same as usual
Then for all
nonnegative integers $n,$ and any $\mu,$ %add
\begin{equation}\label{e:YImu}
\Ext^n_{\bB}(Y,\Itmu)\cong \Ext^n_{\bU}(\RIndBU(Y),\RIndBU(\Itmu)).
\end{equation}
\end{lem}

\begin{proof}
First, observe that (\ref{e:YImu}) is true for $n$ odd, with both
sides zero, by Corollary \ref{cor:Len2}(1). This greatly simplifies
long exact sequence arguments in the remaining $n$ even cases. Now
fix $n$ even. Then (\ref{e:YImu}) is equivalent to the case where
$Y$ is one-dimensional. (In fact, for any given $Y$, (\ref{e:YImu})
is implied by the corresponding results for each of its composition
factors.)

Next, observe in the one-dimensional case, that it is sufficient to
check injectivity of the left-to-right map implicit in
(\ref{e:YImu}). This is a consequence of Corollary \ref{cor:Len2}(1)
and the dimensional equalities  (\ref{e:ndimtensor1}). In fact,
(\ref{e:YImu}) will be an isomorphism for any one-dimensional $Y$
and $\mu$ for which it is an injection or for any $Y$ and $\mu$
(satisfying the hypotheses of the lemma) for which (\ref{e:YImu}) is
an injection on each composition factor of $Y$.

We are now in a position to treat the one-dimensional case $Y =
k_{\mathbb{B}}(\ell \lambda)$, for our fixed even $n$, by downward
induction on the height of $\lambda$. Note that (\ref{e:YImu}) is
true (with both sides zero) for $\lambda$ sufficiently large, by
Corollary \ref{cor:Len2}(2). We may, hence, assume inductively that
(\ref{e:YImu}) holds for $Y = k_{\mathbb{B}}(\ell \eta)$ with $\eta$
of larger height than $\lambda$, or, more generally for all
finite-dimensional $Y$ with composition factors satisfying this
height condition.

Let $M$ be the $\mathbb{B}$-module guaranteed by Lemma \ref{l:M}.
Let $N$ be the cokernel of the $\bB$-module inclusion
$\kB(\ell\lambda) \hookrightarrow M.$ By our height induction, there
is an isomorphism (\ref{e:YImu}) for $Y = N$ and our fixed even $n$.

Also, the $\mathbb{p}$-split map $\RIndBU(M) \to M$ from Lemma
(\ref{l:M}) gives an injection
$$\Ext_{\mathbb{p}}^n(M,k_{\mathbb{p}}(\ell\mu)) \hookrightarrow
\Ext_{\mathbb{p}}^n(\RIndBU(M),k_{\mathbb{p}}(\ell\mu)),$$ or,
equivalently,
\begin{eqnarray}\label{e:Minj}
\ExtB^n(M,\Itmu)&\hookrightarrow &\ExtB^n(\RIndBU(M),\Itmu)\notag\\
{}& \cong & \ExtU^n(\RIndBU(M),\RIndBU(\Itmu)).
\end{eqnarray}
Likewise, the adjunction map $\RIndBU(\kB(\ell\lambda)) \to
\kB(\ell\lambda)$ gives a morphism
\begin{equation}\label{e:gamma}
\gamma:\ExtB^n(\kB(\ell\lambda), \Itmu)\to
\ExtU^n(\RIndBU(\kB(\ell\lambda)),\RIndBU(\Itmu).
\end{equation}

Let
$$\begin{matrix} A = \ExtB^n(N,\Itmu),& A' =
\ExtB^{n+1}(N,\Itmu) =0,\\[2mm]
 \widetilde{A} =
\ExtU^n(\RIndBU(N),\RIndBU(\Itmu)), & \widetilde{A}' =
\ExtU^{n+1}(\RIndBU(N),\RIndBU(\Itmu)) =0,\\[2mm]
B = \ExtB^n(M,\Itmu), & \widetilde{B} =
\ExtU^n(\RIndBU(M),\RIndBU(\Itmu)),\\[2mm]
C = \ExtB^n(\kB(\ell\lambda),\Itmu),& \widetilde{C} =
\ExtU^n(\IndBU(\kB(\ell\lambda)),\RIndBU(\Itmu))\end{matrix}.
$$
%Why is it \Ind vs \RInd for the $\kB\ell\lambda$ here? We aren't assuming dominant, are we?
%Also, in Len's (old) notes, he had $\Ind$ instead of \RInd for the $N$ terms, but why ok?
Then the $\mathbb{B}-$module exact sequence $0 \to \kB(\ell\lambda)
\to M \to N \to 0$ gives rise to a commutative diagram
$$
\xymatrix{\ar[d]^{\alpha}\ar[r]^{u} A & \ar[d]^{\beta}\ar[r]^{v}B &
\ar[d]^{\gamma}\ar[r]^{w}C &
A' =0 \\
\ar[r]^{\widetilde{u}}\widetilde{A} &
\ar[r]^{\widetilde{v}}\widetilde{B} &
\ar[r]^{\widetilde{w}}\widetilde{C} & \widetilde{A}'=0}
$$
with $\gamma$ the morphism in (\ref{e:gamma}), $\alpha$ the
isomorphism given by the induction argument so far, and $\beta$ an
injection given by (\ref{e:Minj}). Both rows are exact. By a
standard diagram chase, these conditions force $\gamma$ to be an
injection. As discussed above, this implies $\gamma$ is an
isomorphism, and completes the induction for our fixed $n$. Since
$n$ was an arbitrary even nonnegative integer, and since the odd
case has already been handled, the proof of the lemma is complete.
\end{proof}

To be clear: As a consequence of Lemma \ref{l:YImuExts}, we
immediately obtain the isomorphisms (\ref{e:nktoIiso}). using
$Y=k_{\mathbb{B}}(\ell\lambda)$. This completes Step 2.

\subsection{Step 3 of the Proof of the Induction Theorems} %new version
\label{subsec:step3} Recall that by Lemma \ref{l:ntrianequiv}, it
suffices to establish that there are isomorphisms ({\ref{e:nisos}):
\begin{equation}%\label{e:nisos}
\Ext^n_{\mathbb{B}}(k_{\mathbb{B}}(\ell \lambda),k_{\mathbb{B}}(\ell
\mu)) \cong
 \Ext^n_{block(\mathbb{U})}(\RIndBU k_{\mathbb{B}}(\ell \lambda),\RIndBU k_{\mathbb{B}}(\ell
 \mu))\notag
 \end{equation}
 arising from the application of $\RIndBU.$

It suffices to fix throughout an otherwise arbitrary weight
$\lambda\in \mathbb{Y}$.
%Weight, or more conditions?
For a given $\mu\in \mathbb{Y},$ all nonnegative integers $n$ will
be treated simultaneously. As a notational convenience, all
$\Ext^m_{\mathbb{B}}$-groups with a negative index $m$ are equal to
zero by definition.

Starting from $\S$3.3, we have isomorphisms (3.1.3) (arising, as
noted, from the functoriality of $\mbox{RInd}^U_{\mathbb{B}}$), and
wish to pass to analogous isomorphisms with
$k_{\mathbb{B}}(\ell\mu)$ in place of the terms $I_\mu^{[1]}$
appearing in (3.1.3).

To begin, note that ``twisting" the canonical $\mathbb{B}$-module
injection $\mu \hookrightarrow I_{\mu}$ of $k_\mathbb{B}(\mu)$ into
its injective hull $I_{\mu}$ leads to a s.e.s.
\begin{equation} \label{e:kBses1}
0\longrightarrow k_{\mathbb{B}}(\ell \mu) \longrightarrow I_{
\mu}^{[1]}\longrightarrow \frac{I_{\mu}^{[1]}}{k_{\mathbb{B}}(\ell
\mu)}=:\Sigma_{\mu} \longrightarrow 0,\end{equation} wherein the
module $I_{\mu}^{[1]}$  has  $\kB(\ell\mu)$ as its socle, and the
$\mathbb{B}$-composition factors of $\Sigma_{\mu}$ have the form
$\kB(\ell\tau),$
with $ht(\tau) > ht(\mu),$ using the usual height function.  %The $I_{\mu}$ structure results appear in Jantzen's treatment of Kempf's vanishing theorem, e.g, p. 207.
%Be sure defn of usual height function appears in background material, since will be ``messing" with height functions in the discussion of various orderings.
Set $ht_{\lambda}(\Sigma_{\mu})$ to be sum of all submodules %\mathbb{B}$-submodules?
$M$ of $\Sigma_{\mu}$ for which
%$\ell \eta$ is a high weight of a composition factor of $M$ implies
$ht(\eta) \leq  ht(\lambda)$ whenever $\ell \eta$ is a composition
factor of $M$. Then $ht_{\lambda}(\Sigma_{\mu})$ is the largest
submodule of $\Sigma_{\mu}$ with this property. In the corresponding
s.e.s.
\begin{equation}\label{e:sigmases}
0\longrightarrow
ht_{\lambda}(\Sigma_{\mu})\longrightarrow\Sigma_{\mu}\longrightarrow
\frac{\Sigma_{\mu}}{ht_{\lambda}(\Sigma_{\mu})}\longrightarrow 0,
\end{equation}
one has
$ht_{\lambda}(\frac{\Sigma_{\mu}}{ht_{\lambda}(\Sigma_{\mu})}) = 0.$

Since $\Ext^n_{\mathbb{B}}(k_{\mathbb{B}}(\ell
\zeta),k_{\mathbb{B}}(\ell \eta))\neq 0$ implies $ \zeta \geq \eta$,
one has, for any weights $\zeta, \eta\in \mathbb{X},$
\begin{equation}\label{e:extup}
\Ext^n_{\mathbb{B}}(k_{\mathbb{B}}(\ell \zeta),k_{\mathbb{B}}(\ell
\eta))=0 \mbox{ if } ht(\eta)>ht(\zeta).
\end{equation}
Thus, from (\ref{e:extup}), it follows that for any
finite-dimensional submodule $N \subset
\frac{\Sigma_{\mu}}{ht_{\lambda}(\Sigma_{\mu})},$
\begin{equation}\label{e:Nvan1}
\Ext^{n}_{\mathbb{B}}(k_{\mathbb{B}}(\ell \lambda),N) = 0 \quad
\forall n \geq 0
\end{equation}
Moreover, since $k_{\mathbb{B}}(\ell\lambda)$ is finite dimensional,
$\Ext^n_{\mathbb{B}}(\kB(\ell\lambda),-)$ commutes with direct
limits. Since $I_{\mu}$ is a direct limit of its finite dimensional
submodules, so is $\Sigma_{\mu},$  and also
$\frac{\Sigma_{\mu}}{ht_{\lambda}(\Sigma_{\mu})};$ consequently, the
vanishing property (\ref{e:Nvan1}) yields the vanishing results
\begin{equation}\label{e:van1}
 \Ext^{n}_{\mathbb{B}}(k_{\mathbb{B}}(\ell \lambda),\frac{\Sigma_{\mu}}{ht_{\lambda}(\Sigma_{\mu})}) = 0 \quad \forall n\geq 0.
 \end{equation}

The s.e.s.~(\ref{e:sigmases})  and the vanishing results
(\ref{e:van1}) yield  the l.e.s.

{%\tiny
$$\begin{array}{c}
\cdots \longrightarrow \Ext^{n-1}_{\mathbb{B}}(k_{\mathbb{B}}(\ell
\lambda),ht_{\lambda}(\Sigma_{\mu})) \longrightarrow
\Ext^{n-1}_{\mathbb{B}}(k_{\mathbb{B}}(\ell \lambda), \Sigma_{\mu})
\longrightarrow
 \Ext^{n-1}_{\mathbb{B}}(k_{\mathbb{B}}(\ell \lambda),\frac{\Sigma_{\mu}}{ht_{\lambda}(\Sigma_{\mu})})  = 0\\
 \longrightarrow
\Ext^n_{\mathbb{B}}(k_{\mathbb{B}}(\ell
\lambda),ht_{\lambda}(\Sigma_{\mu})) \longrightarrow
\Ext^n_{\mathbb{B}}(k_{\mathbb{B}}(\ell \lambda),\Sigma_{\mu})
\longrightarrow
 \Ext^{n}_{\mathbb{B}}(k_{\mathbb{B}}(\ell \lambda),\frac{\Sigma_{\mu}}{ht_{\lambda}(\Sigma_{\mu})}) = 0 \longrightarrow \cdots \\
 \end{array}$$}
whence isomorphisms
\begin{equation}\label{e:sigmamu}
\Ext^n_{\mathbb{B}}(k_{\mathbb{B}}(\ell
\lambda),ht_{\lambda}(\Sigma_{\mu}))
\overset{\cong}{\longrightarrow}
\Ext^n_{\mathbb{B}}(k_{\mathbb{B}}(\ell \lambda),\Sigma_{\mu}),
\quad \forall n \geq 0.
\end{equation}

Consider now  the distinguished triangle obtained from
(\ref{e:sigmases}) under $\RIndBU:$
\begin{equation}\label{e:sigmasesRInd}
\cdots\longrightarrow
\RIndBU(ht_{\lambda}(\Sigma_{\mu}))\longrightarrow\RIndBU(\Sigma_{\mu})\longrightarrow
\RIndBU(\frac{\Sigma_{\mu}}{ht_{\lambda}(\Sigma_{\mu})})\longrightarrow\cdots.
\end{equation}

Since the functor $\RIndBU$ commutes with taking direct limits,
$\RIndBU(\Sigma_{\mu}/ht_{\lambda}(\Sigma_{\mu}) =
\underset{\to}{\lim}\RIndBU(N)$ over finite dimensional submodules
$N\subset \Sigma_{\mu}/ht_{\lambda}(\Sigma_{\mu}).$ Thus, if we can
replicate the following vanishing result, analogous to
(\ref{e:Nvan1}):

\begin{equation} \label{e:Nvan2}
\Ext^{n}_{block(\mathbb{U})}(\RIndBU(k_{\mathbb{B}}(\ell
\lambda)),\RIndBU(N)) = 0 \quad \forall n \geq 0,
\end{equation}
then, from the preceding arguments,  mutatis mutandis, we will
obtain the following isomorphisms:
\begin{equation}\label{e:sigmamuRInd}
\Ext^n_{\mathbb{B}}(\RIndBU(k_{\mathbb{B}}(\ell
\lambda)),\RIndBU(ht_{\lambda}(\Sigma_{\mu})))
\overset{\cong}{\longrightarrow}
\Ext^n_{\mathbb{B}}(\RIndBU(k_{\mathbb{B}}(\ell
\lambda)),\RIndBU(\Sigma_{\mu})), \quad \forall n \geq 0.
\end{equation}
Recall that the key step in producing the isomorphisms (\ref{e:Nvan1}) was the earlier vanishing result (\ref{e:extup}), but now from (\ref{e:extup}) and the dimension equality (\ref{e:nstep1}), %currently (3.2.1)
\begin{equation} \label{e:extup2}
\Ext^n_{block(\mathbb{U})}(\RIndBU(k_{\mathbb{B}}(\ell
\zeta)),\RIndBU(k_{\mathbb{B}}(\ell \eta)))=0 \mbox{ if }
ht(\eta)>ht(\zeta).
\end{equation}
By applying this vanishing result of (\ref{e:extup2}) to composition
factors of $N,$ (\ref{e:Nvan2}), and hence (\ref{e:sigmamuRInd}), do
indeed follow as claimed.

%**********
%NEW  COMPLETION OF ARGUMENT (STREAMLINED; BASED ON XIAN CONVERSATION WITH LEN)? Seems like just expanding on
%old argument, actually. What am I missing?

%Consider now  the sequence obtained from (\ref{e:sigmases}) under $\RIndBU:$
%\begin{equation}\label{e:sigmasesRInd}
%0\longrightarrow \RIndBU(ht_{\lambda}(\Sigma_{\mu}))\longrightarrow\RIndBU(\Sigma_{\mu})\longrightarrow \RIndBU(\frac{\Sigma_{\mu}}{ht_{\lambda}(\Sigma_{\mu})})\longrightarrow 0.
%\end{equation}
%Since the functor $\RIndBU$ commutes with taking direct limits, we can repeat the argument leading to the isomorphisms (\ref{e:sigmamu}) %(with $\RIndBU(k_{\mathbb{B}}(\ell\lambda))$ in place of $k_{\mathbb{B}}(\ell\lambda),$ etc.) to obtain isomorphisms
%\begin{equation}\label{e:sigmamuRInd}
%\Ext^n_{\mathbb{B}}(\RIndBU(k_{\mathbb{B}}(\ell \lambda)),\RIndBU(ht_{\lambda}(\Sigma_{\mu})))
%\overset{\cong}{\longrightarrow}
%\Ext^n_{\mathbb{B}}(\RIndBU(k_{\mathbb{B}}(\ell \lambda)),\RIndBU(\Sigma_{\mu})), \quad \forall n \geq 1.
%\end{equation}

\smallskip
We now carry out a descending induction on  $ht(\mu).$
%If $ht(\mu) > ht(\lambda),$ then by (\ref{e:extup}) and the dimension equality (\ref{l:ndimlem}) applied just as in the first step of the
%$n = 0$ case,
%$\Ext^n_{\mathbb{B}}(k_{\mathbb{B}}(\ell \lambda),k_{\mathbb{B}}(\ell
%\mu)) \cong
% \Ext^n_{block(\mathbb{U})}(\RIndBU k_{\mathbb{B}}(\ell \lambda),\RIndBU k_{\mathbb{B}}(\ell
% \mu)).$
Assume for  all weights $\eta$ for which $ht(\eta) > ht(\mu),$
$\RIndBU$ induces, for all $n$, isomorphisms
\begin{equation} \label{e:nisoseta}
\Ext^n_{\mathbb{B}}(k_{\mathbb{B}}(\ell \lambda),k_{\mathbb{B}}(\ell
\eta)) \cong
 \Ext^n_{block(\mathbb{U})}(\RIndBU k_{\mathbb{B}}(\ell \lambda),\RIndBU k_{\mathbb{B}}(\ell
 \eta)).
 \end{equation}
 Then by the definition of  $ht_{\lambda}(\Sigma_{\mu})$ and its finite dimensionality, it follows from (\ref{e:nisoseta})
 that
\begin{equation}\label{e:nisossigma}
 \Ext^n_{\mathbb{B}}(k_{\mathbb{B}}(\ell \lambda),ht_{\lambda}(\Sigma_{\mu})) \cong
 \Ext^n_{block(\mathbb{U})}(\RIndBU k_{\mathbb{B}}(\ell \lambda),\RIndBU (ht_{\lambda}(\Sigma_{\mu})).
 \end{equation}

\smallskip

From the s.e.s. (\ref{e:kBses1}), upon letting $\widehat{V}$ denote
$\RIndBU(V), $ and $b(\mathbb{U})$ denote $block(\mathbb{U}),$ we
obtain two l.e.s.s tied together:

%{\tiny
%$$\begin{array}{c}
%\cdots\to \Ext^{n-1}_{\mathbb{B}}(k_{\mathbb{B}}(\ell \lambda),I_\mu^{[1]})
%\longrightarrow
% \Ext^{n-1}_{\mathbb{B}}(k_{\mathbb{B}}(\ell \lambda),\Sigma_{\mu})
% \longrightarrow
%\Ext^n_{\mathbb{B}}(k_{\mathbb{B}}(\ell \lambda),k_{\mathbb{B}}(\ell\mu))
%\longrightarrow
%\Ext^n_{\mathbb{B}}(k_{\mathbb{B}}(\ell \lambda),I_\mu^{[1]})
%\longrightarrow
% \Ext^{n}_{\mathbb{B}}(k_{\mathbb{B}}(\ell \lambda),\Sigma_{\mu})\to\cdots \\
% \end{array}$$}

{\tiny
$$\begin{array}{c}
\cdots \to \Ext^{n-1}_{\mathbb{B}}(k_{\mathbb{B}}(\ell
\lambda),I_\mu^{[1]}) \longrightarrow
 \Ext^{n-1}_{\mathbb{B}}(k_{\mathbb{B}}(\ell \lambda),\Sigma_{\mu})
 \longrightarrow
\Ext^n_{\mathbb{B}}(k_{\mathbb{B}}(\ell
\lambda),k_{\mathbb{B}}(\ell\mu)) \longrightarrow
\Ext^n_{\mathbb{B}}(k_{\mathbb{B}}(\ell \lambda),I_\mu^{[1]})
\longrightarrow
 \Ext^{n}_{\mathbb{B}}(k_{\mathbb{B}}(\ell \lambda),\Sigma_{\mu})\to \cdots \\
 ~~~\downarrow \hspace{1.25in} \downarrow  \hspace{1.2in} \downarrow \hspace{1.18in} \downarrow \hspace{1.1in} \downarrow\\
\cdots \to \Ext^{n-1}_{b(\mathbb{U})}\widehat{(k_{\mathbb{B}}(\ell
\lambda)},\widehat{I_\mu^{[1]}}) \longrightarrow
\Ext^{n-1}_{b(\mathbb{U})}\widehat{(k_{\mathbb{B}}(\ell
\lambda)},\widehat {\Sigma_{\mu}}) \longrightarrow
\Ext^n_{b(\mathbb{U})}(\widehat{k_{\mathbb{B}}(\ell
\lambda}),\widehat{k_{\mathbb{B}}(\ell\mu)})         \longrightarrow
\Ext^n_{b(\mathbb{U})}(\widehat{k_{\mathbb{B}}(\ell
\lambda}),\widehat{I_\mu^{[1]}}) \longrightarrow
\Ext^{n}_{b(\mathbb{U})}(\widehat {k_{\mathbb{B}}(\ell
\lambda}),\widehat{\Sigma_{\mu}}) \to \cdots

\end{array}$$
}

In the above diagram, all vertical morphisms arise from the
functoriality of $\RIndBU.$ The first and fourth vertical maps shown
are isomorphisms, as  given by Lemma \ref{l:YImuExts}.  By
(\ref{e:sigmamu}) and (\ref{e:sigmamuRInd}), $\Sigma_{\mu}$ in the
second and fifth vertical morphisms shown can be replaced with
$ht_{\lambda}(\Sigma_{\mu}),$ and the resulting morphisms in the
second and fifth vertical spots are isomorphisms. By the Five Lemma
the third vertical morphism is an isomorphism, i.e., $\RIndBU$
determines
$$%\label{e:nisos}
\Ext^n_{\mathbb{B}}(k_{\mathbb{B}}(\ell \lambda),k_{\mathbb{B}}(\ell
\mu)) \cong
 \Ext^n_{block(\mathbb{U})}(\RIndBU k_{\mathbb{B}}(\ell \lambda),\RIndBU k_{\mathbb{B}}(\ell
 \mu)),$$
 for each $n\geq 0$, as desired.
 This completes the proof of Step 3, and, consequently of both
 induction Theorems 1 and 2.

\smallskip

\subsection{Summary, and comparison with the approach in
\cite{ABG}}.\label{subsec: summary}

%\subsubsection{Choice of generating set}
Although a natural approach, \cite{ABG} were unable to use $S :=
\{\kB(\ell\lambda)\,|\, \lambda \in R\}$ directly as as set of
generators for the triangle category equivalence tool given by
Theorem \ref{l:ntrianequiv}. This roadblock apparently motivated
their attempt to use the set $S' := \{I_{\lambda}^{[1]}\,|\,
\lambda\in R\}$ in place of $S.$ (See \cite[Remark 4.2.7]{ABG}.)
However,  the corresponding claim in \cite[Lem. 4.3.6]{ABG} that
$S'$ (equivalently the set $\{\IndpB(\ell\lambda)\,|\,\lambda\in
R\}$) is inaccurate, since these modules do not actually lie in
$D_{triv}(\mathbb{B})$.
  %TLH: Because...?
%5/26:OLD STATEMENT, ``appears problematic; its analog in
%characteristic $p$ is certainly false, since the modules would fail
%to have finite cohomology.''
Nevertheless, in the characteristic 0 setting of \cite{ABG}, it is
true that $D_{triv}(\mathbb{B})$ is contained in the triangulated
category generated by $S',$
 so that a line of
argument establishing isomorphisms
\begin{equation}\label{e:Iisos}
\Ext^n_{\mathbb{B}}(I_{\lambda}^{[1]},I_{\mu}^{[1]})\cong
\Ext^n_{block(\mathbb{U})}(\RIndBU(I_{\lambda}^{[1]}),\RIndBU(I_{\mu}^{[1]})),
\end{equation}
as pursued in \cite[Lem. 4.3.6]{ABG} {\it would} imply the existence
of isomorphisms (\ref{e:nisos}).
%New, from "Len's 9-14-2011 corrections"
Unfortunately, there is no such inclusion of  $D_{triv}(\mathbb{B})$
in characteristic $p>0.$ In particular, the first line of the proof
of \cite{ABG}[Lem. 4.3.6], asserting that the universal enveloping
algebra $\mathcal{U}\frak{n}$ (for $\frak{n}$ a nilpotent Lie
algebra in a triangular decomposition)
%\footnote{TLH -- Need to standardize this notation above with notation we agree to use earlier, as discussed on 5/26 w/Len.}
has finite global dimension is not true for the correctly analogous
characteristic $p$ situation. It is a question of what modules are
to be pulled back under the Frobenius morphism.
%TLH -- didn't Len and I discuss this on 5/26 further?
In the characteristic $p$ situation, it is necessary to use modules
for the distribution algebra of a positive characteristic unipotent
algebraic group, not its unrestricted enveloping algebra, and the
finite global dimension property is lost. Overcoming this obstacle,
while using much of the apparatus of \cite{ABG}, is not trivial, and
our proof eventually involves parity properties for
$\mathbb{b}$-cohomology \cite[Prop. 2.3]{AJ}. See above Corollary
\ref{cor:Len2} and the proof of Lemma \ref{l:YImuExts}, which also
present our argument in the quantum case.
%\footnote{TLH -- Again, standardize notation
%-- I elected $\mathbf{b}$ here.}

%section{Truncation and the Induction Theorem} removed; ideas were shifted this past year to a new separate paper.

\newpage
\section{Appendix A}

%\subsection{Notation}

The discussion below, in the algebraic groups case, is based on
Jantzen's book (\cite{J}, pp. 258-259). We follow the notations
there. Comments on the quantum case are given in Remark
\ref{r:quantumA}, to which the reader might look ahead, now. In this
appendix and the next we will provide a proof of Lemma
\ref{wallcrossingfunctors}(ii). A closely linked goal is to
understand the adjunction map $id\circ \RIndBG\longrightarrow
T_\mu^\lambda\circ T_\lambda^\mu\circ \RIndBG$ in the spirit of the
long exact sequences on \cite[p. 259]{J}. Put $L=L(\nu_1)$ as on the
cited page. The functor $T^{\mu}_{\lambda}, T^{\lambda}_{\mu}$ are
constructed from tensor product functors $L\otimes (-) ,
~L^{*}\otimes (-)$,
 respectively, using block projections $pr_{\lambda}, pr_{\mu}$:
 $$T^{\lambda}_{\mu}=pr_{\lambda}(L^{*}\otimes -)\circ pr_{\mu},\\\
  T^{\mu}_{\lambda}=pr_{\mu}(L\otimes -)\circ pr_{\lambda}.$$
These are the definitions given in Jantzen, familiar to many
readers. All functors are
  regarded as functors from the category of rational $G$-modules to
  itself. The functor $L\otimes -:=L\otimes (-)$ is (left and right) adjoint to $L^{*}\otimes
  - := L^{*}\otimes (-)$, while $pr_{\lambda}$ and $pr_{\mu}$ are both self-adjoint (left
 and right). Consequently, $T^{\mu}_{\lambda}$ is (left and right) adjoint to
 $T^{\lambda}_{\mu}$.

\medskip
 {\Large Construction of the adjunction map}

 \smallskip

   So far, this is all standard, but we can go a
 little further.

 \medskip

 \textbf{(1)} Let $X,Y\in {G-Mod}$, the category of rational
 $G$-modules. Then the identifications $\Hom_G(pr_{\lambda}X,Y)\cong
 \Hom_G(X,pr_{\lambda}Y)$ are quite canonical: Write $X,Y$,
 respectively, as direct sums of submodules
 $$X=pr_{\lambda}X\oplus pr'_{\lambda}X,\\ Y=pr_{\lambda}Y\oplus
 pr'_{\lambda}Y,$$
  where $pr'_{\lambda}X$ has no composition
 factors in the block associated with $\lambda$ and is maximal, as
a submodule of $X$, with that property. The submodule
$pr'_{\lambda}Y$ of $Y$ is defined similarly. Obviously, any
$G$-homomorphism $pr_{\lambda}X\to Y$ has image in $pr_{\lambda}Y$.
Consequently, it identifies with a map $pr_{\lambda}X\to
pr_{\lambda}Y$. Also, any $G$-homomorphism $X\to pr_{\lambda}Y$
sends $pr'_{\lambda}$ to $0$ and sends $pr_{\lambda}X$ to
$pr_{\lambda}Y$. Thus, $$\Hom_G(pr_{\lambda}X,Y) \cong
\Hom_G(pr_{\lambda}X,pr_{\lambda}Y)\cong \Hom_G(X,pr_{\lambda}Y)$$
with each identification very obvious and canonical. We also record
$$\Hom_G(X,Y)\cong \Hom_G(pr_{\lambda}X,pr_{\lambda}Y)\oplus \Hom_G(pr'_{\lambda}X,pr'_{\lambda}Y$$

All the the observations in the above paragraph hold if $\lambda$ is
replaced by $\mu$.

\medskip

 \textbf{(2)} In particular, suppose we are given a natural
transformation $\eta=\{\eta_{X,Y}\}_{X,Y\in {G-Mod}}$ from
$\Hom_G(L\otimes Y$ to $\Hom_G(X,L^{*}\otimes Y)$. Then $\eta$ gives
maps
$$\eta_{pr_{\lambda} X,pr_{\mu} Y}: \\ \Hom_G(L\otimes pr_{\lambda}X,
pr_{\mu}Y)\longrightarrow \Hom_G( pr_{\lambda}X,L^{*}\otimes
pr_{\mu}Y).$$ This induces, using \textbf{(1)}, a natural
transformation we will call $\tilde \eta$, again defined on ${G-Mod}
\times {G-Mod}$, with $\tilde \eta_{X,Y}$ a map from
$\Hom_G(pr_{\lambda}L\otimes pr_{\lambda}X, pr_{\mu}Y)$ to
$\Hom_G(X,pr_{\lambda}(L^{*}\otimes pr_{\mu}Y))$. Explicitly,

%$$\begin{CD}
% \Ind(w\cdot\lambda\otimes \tau \otimes \vupperone ) @>Jan_Y^{w\cdot\lambda}>> \Tmulam\Ind(w\cdot\mu\otimes p\tau \otimes \vupperone\\
%@VV\cong V     @VV\cong V\\
%  \Ind(w\cdot\lambda\otimes p \tau)\otimes \vupperone @>Jan_{p\tau}^{w\cdot \lambda} \otimes \vupperone >> (\Tmulam(w\cdot\mu\otimes p \tau))\otimes \vupperone
% \end{CD}$$

$$\begin{CD}
 \Hom_G(L\otimes
 pr_{\lambda}X,pr_{\mu}Y)@>\eta_{pr_{\lambda}X,pr_{\mu}Y}>>\Hom_G(pr_{\lambda}X,L^{*}pr_{\mu}Y)\\
@VV\cong V @VV\cong V\\ \Hom_G(pr_{\mu}(L\otimes pr_\lambda
X)@>\tilde \eta_{X,Y}>>
\Hom_G(X,pr_{\lambda}(L^{*}\otimes pr_{\mu}Y))\\
@VV= V @VV= V\\
 \Hom_G(T^{\mu}_{\lambda}X,Y)@. \Hom_G(X,T^{\lambda}_{\mu}Y)
 \end{CD}$$
 with the vertical isomorphisms between the top two rows given by
 {\textbf{(1)}}. Note there is a similar diagram with $pr_{\mu}Y$
 replacing $Y$ in the bottom two rows (using $\tilde
 \eta_{X,pr_{\mu}Y}).$
\smallskip

 \textbf{(3)} Using the naturality of $\eta$, we can put another row
 and commutative diagram(s) on top of the top row above:
$$\xymatrix{
  \Hom_G(L\otimes pr_{\lambda}X,Y) \ar[r]^{~~\eta_{pr_{\lambda}X,Y}}~~ \ar@<-4pt>[d] &
  \Hom_G(pr_{\lambda}X,L^{*}\otimes Y)
  \ar@<-4pt>[d]\\
   \Hom_G(L\otimes pr_{\lambda}X,pr_{\mu}Y) \ar[u]
   \ar[r]^{~~\eta_{pr_{\lambda}X,pr_{\mu}Y}}~~
   &  ~~~~\Hom_G(pr_{\lambda}X,L^{*}\otimes pr_{\mu}Y)\ar[u]
}$$ Here the pair of vertical maps pointing upward are indexed by
the inclusion $pr_{\mu}\to Y$ and yield a commutative diagram.
Similarly the pair of downward arrows are indexed by the projection
$Y\to pr_{\mu}Y$ and give a commutative diagram. The composite of
the homomorphisms represented by the upward pointing arrows with the
homomomorphism represented by the corresponding downward pointing
arrows are identities.

\medskip

We can now prove
\begin{prop}\label{adjprop} Assume the natural transformation
$\eta=\{\eta_{X,Y}\}_{X,Y\in {G-Mod}}$ gives natural isomorphisms
$$\eta_{X,Y}:~\Hom_G(L\otimes X,Y)\longrightarrow
\Hom_G(X,L^{*}\otimes Y),$$ and let $\tilde \eta=\{\tilde
\eta_{X,Y}\}_{X,Y\in {G-Mod}}$ be the corresponding natural
transformation constructed above. Then $\tilde \eta$ gives natural
isomorphisms $$\tilde
\eta_{X,Y}:~\Hom_G(T^{\mu}_{\lambda}X,Y)\longrightarrow
\Hom_G(X,T^{\lambda}_{\mu}Y).$$ Moreover, the corresponding
adjunction transformation $\widetilde{adj}$ from the identity
functor on ${G-Mod}$ to the functor
$T^{\lambda}_{\mu}T^{\mu}_{\lambda}$ may be constructed from the
adjunction map $adj$ similarly associated with $\eta$. In fact, for
each $X\in {G-Mod}$, we have a commutative diagram
$$\xymatrix{
 X~~ \ar[d]^{=} \ar[r]^{adj_X~~} & ~~L^{*}\otimes L \otimes X
  \ar[d]\\
  X~~ \ar[r]^{\widetilde{adj}_X}  & T^{\lambda}_{\mu}T^{\mu}_{\lambda}X
}$$ where the down arrow on the right is the composite projection
$$L^{*}\otimes L\otimes X\to L^{*}\otimes L\otimes pr_{\lambda}X\to L^{*}\otimes
pr_{\mu}(L\otimes pr_{\lambda}X)\to pr_{\lambda}(L^{*}\otimes
pr_{\mu}(L\otimes pr_{\lambda}X))=
T^{\lambda}_{\mu}T^{\mu}_{\lambda}X.$$

\end{prop}

\begin{proof} We use the (noted) alternate version of the diagram in
 \textbf{(2)} in which $pr Y$ replaces $Y$, and use the diagram in
 \textbf{(3)} as given. The combination gives a commutative diagram
 $$\begin{CD}
  \Hom_G(L\otimes pr_{\lambda}X,Y)@>{~~\eta_{pr_{\lambda}X,Y}}~~ >>
  \Hom_G(pr_{\lambda}X,L^{*}\otimes Y)\\
 @VV\cong V @VV\cong V\\
 \Hom_G(L\otimes
 pr_{\lambda}X,pr_{\mu}Y)@>\eta_{pr_{\lambda}X,pr_{\mu}Y}>>\Hom_G(pr_{\lambda}X,L^{*}\otimes pr_{\mu}Y)\\
@VV\cong V @VV\cong V\\ \Hom_G(pr_{\mu}(L\otimes pr_\lambda
X),pr_{\mu}Y)@>\tilde \eta_{X,pr_{\mu}Y}>>
\Hom_G(X,pr_{\lambda}(L^{*}\otimes pr_{\mu}Y))\\
@VV= V @VV= V\\
 \Hom_G(T^{\mu}_{\lambda}X,Y)@. \Hom_G(X,T^{\lambda}_{\mu}Y)
 \end{CD}$$

  Now take $Y=L\otimes pr_{\lambda}X$. Thus $pr_{\mu}Y=T^{\mu}_{\lambda}X$, and
$\tilde \eta_{X,pr_{\mu}Y}(1_{T^{\mu}_{\lambda}}X)= \tilde
\eta_{X,T^{\mu}_{\lambda}X}(1_{T^{\mu}_{\lambda}X}) =\widetilde
{adj}_X$. Chasing the element $1_{T^{\mu}_{\lambda}X}$ up to the
second row gives an element which is the projection $Y\to pr_{\mu}Y$
in $\Hom_G(L\otimes
 pr_{\lambda}X,pr_{\mu}Y)=\Hom_B(Y,pr_{\mu}Y)$.
This element is also the image of $1_{L\otimes pr_{\lambda}X}=1_Y$
under the downward vertical map on the left. Observe
$\eta_{pr_{\lambda}X,Y}(1_Y)=adj_{pr_{\lambda}X}.$ Following the
right hand vertical maps in the case $X=pr_{\lambda}X$ gives a
commutative diagram
$$\xymatrix{
 pr_{\lambda}X~~ \ar[d]^{=} \ar[r]^{adj_{pr_{\lambda}X}~~} & ~~L^{*}\otimes L \otimes pr_{\lambda}X
  \ar[d]\\
 pr_{\lambda}X~~ \ar[r]^{\widetilde{adj}_{pr_{\lambda}X}}  & T^{\lambda}_{\mu}T^{\mu}_{\lambda}pr_{\lambda}X
}$$ with the right hand map the composite of projections
$$L^{*}\otimes L\otimes pr_{\lambda}X\to L^{*}\otimes
pr_{\mu}(L\otimes pr_{\lambda}X)\to pr_{\lambda}(L^{*}\otimes
pr_{\mu}(L\otimes pr_{\lambda}X))=
T^{\lambda}_{\mu}T^{\mu}_{\lambda}X.$$

  Now return to the case of a general$X$ and apply
  functoriality\footnote{It is a general property of adjunction maps
  $Id\to EF$ where $E$ is a right adjoint to a functor $F$, that any
  map $\phi:X\to X'$ in the underlying category gives a commutative
  diagram
  $$\xymatrix{
  X' \ar[r] &  EF(X')\\
  X  \ar[u]_{\phi} \ar[r] & EF(X) \ar[u]_{EF(\phi),}
  }$$ where both horizontal maps are adjunctions. We include a brief proof:
  $F(\phi)$ is the value at $1_{F(X)}$ of the
  evident map $\Hom(F(X),F(X))\to \Hom(F(X),F(X'))$ and also the
  value at $1_{F(X')}$ of the evident map $\Hom(F(X'),F(X'))\to
  \Hom(F(X),F(X'))$. Applying the (natural) adjointness isomorphism
  $\Hom(F(-),F(-))\cong \Hom(-,EF(-))$ to $F(\phi)$ yields a map $X\to
  EF(X')$ which, correspondingly, factors in two different ways,
  giving the desired commutative diagram.

  We remark that there is a dual commutative diagram for the ``counital adjunction" $FE\to Id$
  The formulation and
  proof may be given using dual categories and the adjunction case.}
of the adjunction maps $adj,~\widetilde{adj}$ to obtain a
commutative diagram

$$\xymatrix{
             X~~ \ar[d] \ar[r]^{adj_X}~~ & ~~L^{*}\otimes L \otimes X
  \ar[d]\\
 pr_{\lambda}X~~ \ar[d]^{=} \ar[r]^{adj_{pr_{\lambda}X}~~} & ~~L^{*}\otimes L \otimes pr_{\lambda}X
  \ar[d]\\
 pr_{\lambda}X~~ \ar[r]^{\widetilde{adj}_{pr_{\lambda}X}}  & T^{\lambda}_{\mu}T^{\mu}_{\lambda}pr_{\lambda}X
\\
  X~~ \ar[u] \ar[r]^{\widetilde{adj}_X}  & T^{\lambda}_{\mu}T^{\mu}_{\lambda}X
  \ar[u]_{=}
}$$ In this diagram, the middle rectangle is identical to the
diagram just discussed. All the unlabeled vertical maps are evident
projections. In particular, the whole commutative diagram could be
extended on the left, preserving commuttitivity, by a long downward
equality map from the upper left $X$ to the lower left $X$.  Also,
the lower right equality arrow can be reversed, still preserving
commutativity. With these changes, the perimeter rectangle becomes
the commutative diagram required in the proposition. This completes
its proof.
\end{proof}

\smallskip
We remark that the adjunction obtained from the usual natural
isomorphism
$$\Hom_G(L\otimes X,Y)\cong \Hom_G(X,L^{*}\otimes X)$$
is quite explicit: For $x\in X, ~adj_X(x)= 1_L\otimes x$,
 if $L^{*}\otimes L$ is identified with $\Hom_k(L,L)$. Even if we do not
use that identification, we can just write
$$adj_X(x)=\sum_{\epsilon \in I} \epsilon^{*}\otimes \epsilon
\otimes x,$$ where $\epsilon$ ranges over any basis $I$ of $L$, and
$\epsilon^{*}$ denotes the corresponding dual basis element. The sum
on the right is independent of the basis $I$ chosen.

\medskip
As a corollary to the proposition, we have
\begin{cor}\label{adjcor} Let $X,Y\in G-mod$, and identify
$L^{*}\otimes L\otimes (X\otimes Y)\cong (L^{*}\otimes L\otimes X)
\otimes Y$. Then $adj_{X\otimes Y}(-) = (adj_X(-))\otimes Y$. If all
weights of $Y$ lie in the root lattice, then
$\widetilde{adj}_{X\otimes Y^{[1]}}(-)=\widetilde{adj}_X(-)\otimes
Y^{[1]}$, identifying $T^{\lambda}_{\mu}T^{\mu}_{\lambda}(X\otimes
Y^{[1]})$ with $(T^{\lambda}_{\mu}T^{\mu}_{\lambda}X)\otimes
Y^{[1]}$.
\end{cor}
\begin{proof} The first equality is immediate from the formula for
$adj_X(x), x\in X$ above, applied to to $X\otimes Y$ and
$adj_{X\otimes Y}$.

We can argue with $adj$ to handle $\widetilde{adj}$: First, observe
the rearrangements
$$pr_{\lambda}(X\otimes Y^{[1]})=pr_{\lambda}X\otimes Y^{[1]}, ~~\text{and}$$
$$pr_{\mu}(L\otimes pr_{\lambda}(X\otimes Y^{[1]}))=pr_{\mu}(L\otimes pr_{\lambda}X\otimes
Y^{[1]})=pr_{\mu}(L\otimes pr_{\lambda}X)\otimes Y^{[1]}.~~Also,$$
$$pr_{\lambda}(L^{*}\otimes pr_{\mu}(L\otimes pr_{\lambda}(X\otimes Y^{[1]})))=pr_ {\lambda}(L^{*}\otimes pr_{\mu}(L\otimes
pr_{\lambda}X)))\otimes Y^{[1]}.$$ Here we have heavily used the
fact that the operator $-\otimes Y^{[1}]$ commutes with our ``block"
projections. (Recall the latter are formulated in terms of the
affine Weyl group, which contains translations by $p$-multiples of
the root lattice.) We have regarded $pr_{\lambda}X$ as a submodule
of $X$, and have taken a similar viewpoint with all the projections
in these equalities. (Similar equalities hold for complementary
projections, Thus, $pr'_{\lambda}(X\otimes
Y^{[1]})=pr'_{\lambda}X\otimes Y^{[1]},$ etc.)

Recall that we have described $\widetilde{adj}_X$ in Proposition
\ref{adjprop} as the composition of $adj_X$ followed by a sequence
of projections
$$L^{*}\otimes L\otimes X\to L^{*}\otimes L\otimes pr_{\lambda}X\to L^{*}\otimes
pr_{\mu}(L\otimes pr_{\lambda}X)\to pr_{\lambda}(L^{*}\otimes
pr_{\mu}(L\otimes pr_{\lambda}X).$$ Tensoring $adj_X(-)$ on the
right with $Y^{[1]}$ gives $\widetilde{adj}_{X\otimes Y^{[1]}}(-)$
as shown above (even with $Y^{[1]}$ any $G$-module). Next, tensor
the sequence of projections displayed above with $Y^{[1]}$,
obtaining
$$L^{*}\otimes L\otimes X \otimes Y^{[1]}\to L^{*}\otimes L\otimes pr_{\lambda}(X \otimes Y^{[1]})\to L^{*}\otimes
pr_{\mu}(L\otimes pr_{\lambda}(X\otimes Y^{[1]})$$
$$~~~~~~~~~~~~~~\to pr_{\lambda}(L^{*}\otimes pr_{\mu}(L\otimes pr_{\lambda}X))\otimes
Y^{[1]}.$$ Using the rearrangements discussed above, we get
$$L^{*}\otimes L\otimes X \otimes Y^{[1]}\to L^{*}\otimes L\otimes pr_{\lambda}X \otimes Y^{[1]}\to L^{*}\otimes
pr_{\mu}(L\otimes pr_{\lambda}X)\otimes Y^{[1]}$$
$$~~~~~~~~~~~~~~\to pr_{\lambda}(L^{*}\otimes pr_{\mu}(L\otimes pr_{\lambda}(X\otimes
Y^{[1]}))).$$ Here we have identified tensor products isomorphic
through the associative law. The lower display above is easily
recognized as the sequence in Proposition \ref{adjprop} whose
composition with $adj_{X\otimes Y^{[1]}}$ gives
$\widetilde{adj}_{X\otimes Y^{[1]}}$. Combining this with the
equality $adj_{X\otimes Y^{[1]}}(-)=adj_X(-)\otimes Y^{[1]}$ noted
above gives the identification $\widetilde{adj}_{X\otimes
Y^{[1]}}(-)=\widetilde{adj}_X(-)\otimes Y^{[1]}$, completing the
proof of the corollary.
\end{proof}
\medskip

\begin{rem}\label{r:quantumA} \textbf{The quantum case}\medskip\\
The lack of cocommutativity requires some care in treating the
quantum case, and it becomes important to distinguish right from
left. For example, consider the ``usual" natural isomorphism
$$\Hom_G(L\otimes X,Y)\cong \Hom_G(X,L^{*}\otimes X)$$
in the algebraic groups case. It may be given in more detail as a
composite \begin{equation}\label{e:homtensor} \Hom_G(L\otimes
X,Y)\cong \Hom_G(X,\Hom_k(L,Y))\cong \Hom_G(X,L^{*}\otimes Y).
\end{equation}
The right hand isomorphism depends on the isomorphism of $G$-modules
    $$\Hom_k(L,Y)\cong L^{*}\otimes Y.$$
The usual way to identify a simple tensor element $f(-)\otimes y$ on
the right ($f\in L^*, y\in Y$) with a function on the left is to let
it send $v\in L$ to $f(v)y$. Let $g\in G$, and suppose for a moment
that $G$ is not a group, but a Hopf algebra with antipode $S$, and
that $L, Y$ are left $G$-modules. Using Sweedler (implicit sum)
notation, with $g$ mapping to $g_1\otimes g_2$ under
comultiplication, the action of $g$ on the right hand element gives
$f(Sg_1(-)\otimes g_2y$, but, on the left, it gives a function
sending $v$ to $f(Sg_2(v))g_1y$. The latter is not formally the
function corresponding to the right hand element without
cocommutativity.

This can be fixed by either using the right action of the Hopf
algebra $G$ or by keeping the left action and changing the tensor
product $L^*\otimes Y$ to $Y\otimes L^*$. We prefer the latter
approach, since left actions are often implicitly used--e.g., in
\cite{J}. In keeping with the spirit of previous sections of this
paper, define an ``opposite" tensor product $\otimes^{op}$ by
               $$X\otimes^{op} Y:= Y\otimes X,$$

A similar analysis can be carried out on the left hand isomorphism
of the display (\ref{e:homtensor}). We find that the standard
correspondence gives an isomorphism of left $G$-modules
 $$\Hom_k(L\otimes^{op} X,Y)\cong \Hom_k(X,\Hom_k(L,Y)).$$
Here the left action of the Hopf algebra $G$ on the various modules
$\Hom_k(-,-)$ is given by ``conjugation." That is, if $g\in G$ and
$f$ is a linear function from one left $G$-module to another, the
action of $g$ on $f$ gives a linear function $g_1f(Sg_2(-))$. When
the antipode is surjective (as it is for all the Hopf algebras we
consider), the space of ``fixed points" of this action of $G$ (all
$f$ for which each $g\in G$ acts through the counit) results
precisely in the space of $G$-homomorphisms. (A general statement
and proof of this fact may be found in \cite[2.9]{APW}.) In
particular, we have a general version of (\ref{e:homtensor}) which
holds for any such Hopf algebra:
\begin{equation}\label{e:homtensor2} \Hom_G(L\otimes^{op}
X,Y)\cong \Hom_G(X,\Hom_k(L,Y))\cong \Hom_G(X,L^{*}\otimes^{op} Y).
\end{equation}
Finally, notice that $\otimes^{op}$ is just as associative an
operation as $\otimes$, which is strictly associative, if standard
identifications are made in iterated tensor products of $k$-spaces.

Thus, \textbf{the results and arguments of this section hold in the
quantum case.}  The reader can even read or reread the statements
and arguments in both the algebraic groups case and quantum case
simultaneously, after replacing $\otimes$ with $\otimes^{op}$, and
using the same simultaneous notations $\mathbb{U},
\mathbb{B},\mathbb{k}, \dots$, as in previous sections, in place of
$G, B, k, \ldots$.

\end{rem}

\section{Appendix B}

%\subsection{Notation}

We now return to Jantzen \cite{J}, pp. 258-259. Proposition 7.11
there implies, if $T_\lambda^\mu$ is  ``to a wall," then
$$T_\lambda^\mu\RIndBG (w\cdot \lambda)\cong \RIndBG(w\cdot\mu).$$
  Recall Jantzen denotes one dimensional weight modules by the weights alone.\\

The argument for the above isomorphism is helpful:  $T_\lambda^\mu
\RIndBG(w\cdot\lambda)=pr_\mu(L\otimes pr_{_\lambda}
\RIndBG(w\cdot\lambda)=pr_\mu(L\otimes\RIndBG(w\cdot\lambda))=pr_\mu\RIndBG
(L\otimes w\cdot\lambda).$ At this point a $B$ composition series of
$L$ is examined and it is found that there is only one composition
factor, call it $\lambda _l$ appearing with multiplicity one, such
that $pr_\mu\RIndBG(\lambda_l\otimes w\cdot \lambda)\neq 0.$ It is
determined that $\lambda_l\otimes w\cdot\lambda$ is $w\cdot \mu$,
completing the proof.

\smallskip
Next, let us come  ``out of the wall" with $T^\lambda_\mu$.  We assume $\mu$ is on a true ``wall" with stabilizer $\{1, s\}$ for a simple reflection $s$. We want to know what happens to $T_\mu^\lambda \RIndBG(w\cdot\mu).$ Again, write\\
$$
\begin{tabular}{lll}
$T^\lambda_\mu\RIndBG(w\cdot\mu)$&= & $pr_{_\lambda}(L^*\otimes pr_\mu\RIndBG(w\cdot\mu))$\\
& = & $ pr_{_\lambda}(L^*\otimes \RIndBG(w\cdot\mu)$\\
& = & $pr_{_\lambda}\RIndBG(L^*\otimes w\cdot \mu)$.
\end{tabular}$$
This time $L^*$ has two composition factors exactly, $\gamma=\nu$
and  $\gamma =\nu^\prime$, each appearing with multiplicity 1, such
that $pr_\lambda\RIndBG(\gamma\otimes w\cdot\mu)\leq0.$ The two
weights $\nu,~\nu^\prime$(not Jantzen's notation) satisfy
$\{\nu+w\cdot\mu, \nu^\prime +w\cdot \mu\}=\{w\cdot\lambda,
ws\cdot\lambda\}.$

\smallskip
Jantzen treats the case $ws\cdot\lambda < w\cdot\lambda$ (with the
roles of $\lambda, \mu$ reversed) in \cite[Prop. 7.12]{J}. For our
purposes, to be compatible with Lemma
\ref{wallcrossingfunctors}(ii), we will consider the case $ws\cdot
\lambda>w\cdot \lambda$,
which requires different arguments (in the same setting). \\

\subsec{\textbf{The main issue.~}}\label{claim} In this case there
is an exact sequence of $B$-modules
$$ 0\longrightarrow M \longrightarrow L^*\otimes sw\cdot\mu\longrightarrow M^\prime \longrightarrow 0$$
in which the weight $w\cdot \lambda$ appears in $M$ and $ws\cdot
\lambda$ appears in $M^\prime$. These appearances are each with
multiplicity 1, and no other weight $\tau$ with $pr_{_\lambda}
\RIndBG(\tau)\neq0$ appears in either $M$ or $M^\prime$. Apply
$pr_{_\lambda}\RIndBG$ to the above short exact sequence.

%\fbox{
%\begin{tcolorbox}
%{45em}
~~The result is a distinguished triangle
$$(*)~~~~~~\cdots\longrightarrow\RIndBG(w\cdot\lambda)\longrightarrow
T_\mu^\lambda\RIndBG(w\cdot\mu)\longrightarrow\RIndBG(ws\cdot\lambda)\longrightarrow\cdots$$
As previously noted, the middle term is isomorphic to
$$T^\lambda_\mu T^\mu_\lambda\RIndBG(w\cdot\lambda).$$ This leads to the question as to
whether or not the resulting map
$$\RIndBG(w\cdot\lambda)\longrightarrow T^\lambda_\mu T^{\mu}_\lambda\RInd(w\cdot \lambda)$$
is the adjunction map.    \textbf{We claim} that \underline{it is},
indeed, \underline{the adjunction map}, at least up to a nonzero
scalar multiple.
(Thus, there is a distinguished triangle (*) in which the left hand map is adjunction.)\\
%\end{tcolorbox}

\textbf{The proof of this claim will essentially occupy the rest of
this appendix!} We will use the context and notation of the
algebraic groups case, and treat the quantum case at the end in
Remark \ref{r:app2}. Part (ii) of Lemma \ref{wallcrossingfunctors}
then follows, since $\theta ^+_\alpha(\RIndBU\lambda)$, from its
mapping cone definition, fits into a triangle with the same two
objects and left hand map as (*), but replacing the object  $\RIndBU
(\lambda^{s_\alpha})$ as the third object. Standard triangulated
category axioms then give a map $\theta
^+_\alpha(\RIndBU\lambda)\longrightarrow \RIndBU
(\lambda^{s_\alpha})$, part of a commutative diagram with identify
maps on the two left hand objects in (*) and their translations
under [1]. The ``five lemma" then gives the desired isomorphism
$\theta ^+_\alpha(\RIndBU\lambda)\cong \RIndBU
(\lambda^{s_\alpha})$, completing the proof of part (ii) of Lemma
\ref{wallcrossingfunctors}. Since the a proof of part (i) has
already been given, this will complete the proof of the lemma.

triangles which is the identity on the two left hand objects.
of the algeb\\

For the moment, we prove the claim in the case where both $w\cdot
\lambda$ and $ws\cdot \lambda$ are dominant: Note that
$\RIndBG(w\cdot\lambda)\cong \IndBG(w\cdot\lambda)$ and
$\RIndBG(ws\cdot\lambda)\cong \IndBG(ws\cdot\lambda)$ by Kempf's
theorem.  Also, of course, the functor $\RIndBG$ is right adjoint to
restriction. In particular
 $$\Hom_{D^b(G)}(\RIndBG(w\cdot\lambda),\RIndBG (ws\cdot\lambda))=\Hom_B(\IndBG (w\cdot \lambda), ws\cdot\lambda)$$
 so that any map from $\RInd(w\cdot\lambda)$ to the middle term of (*) factors through the left hand map.
 However, $\Hom_B(\RIndBG(w\cdot\lambda), w\cdot\lambda)\cong k$.
 The claim follows. In our argument we have used the fat that $\RIndBG$ is right adjoint to restriction.\\

 \subsec{}\label{afterclaim} We will now try to exploit the validity of the dominant case,
 by using it to build well-behaved resolutions in the general
 $w\cdot\lambda<ws\cdot\lambda$ case, to which we now return.\\

 The $B$-modules we will use to resolve $w\cdot \lambda$ will be sums of those of the form
 $w\cdot \lambda\otimes p\tau\otimes \vupperone  $, where
 $\tau$ is in the root lattice and $\vupperone$ is a Frobenius twisted $G$-module
 (restricted to $B$) with $V$ having all weights in the root lattice.

 \begin{lem}\label{resolvetrivialmodule} The trivial module $k=k(0)$ has a positive resolution $k\stackrel{\thicksim}{\rightarrow}K^\bullet,$ where each $K^n$ is a direct sum of $B$-modules $p\tau\otimes \vupperone$, with $\tau$ and $\vupperone$ as above.  Moreover, we may assume all $\tau$ are dominant and that $q\cdot \lambda+p\tau,~ w\cdot \mu+p\tau$, are $ws\cdot\lambda+p\tau$ are also dominant.

 In fact, we can assume $\nu+p\tau$ is dominant for all $\nu$ in any fixed
  finite list of weights.
 \end{lem}

 \begin{proof}
 Notice that $k$ and all $p\tau\otimes \vupperone$ are Frobenius-twisted $B$-modules.
 Each (Frobenius-)twisted injective $B$-module hull $I_\mu^{[1]}$ for a weight
 $\mu$ in the root lattice is a direct union of modules $p(\mu+\sigma)\otimes
 \vupperone$; see({e:Itwist}). Thus, any finite dimensional $B$-module
 $N^{[1]},$ with $N$ having all weights in the root lattice, can be embedded in a direct sum of these,
 with the weights $\tau=\mu+\sigma$, as large as we like.  The cokernel of the embedding will also be a finite dimensional twisted $B$-module of the same form as $N^{[1]}$ above.  Hence the process can continue.  Starting with $k=k(0)$ in the initial role of $N^{[1]},$ we obtained the desired resolution.
 \end{proof}

 We now describe some of the main issues we face at this point.  Let $\tau$ be any weight in the root lattice such that $w\cdot\lambda+p\tau,~   w\cdot\mu+p\tau$, and $ws\cdot \lambda+p\tau$ are dominant, as well as $p\tau.$
Form the composite of the adjunction map
$$\IndBG(w\cdot\lambda+p\tau)\longrightarrow T^\lambda_\mu T^\mu_\lambda \IndBG(w\cdot \mu+p\tau)$$
and the usual isomorphism
$$\Tmulam\Tlammu\IndBG(w\cdot\lambda+p\tau)\tilde{\rightarrow}\Tmulam\IndBG(w\cdot\mu+p\tau)$$
Note $w\cdot\lambda+p\tau=w^\prime\cdot\lambda$ and $w\cdot\mu + p
\tau = w^\prime \cdot \mu$ for $w^\prime$, the composite of $w$
followed by translation by $p\tau$.  We will discuss the ``usual"
isomorphism later in some details, but it is exact nature may be
regarded as unknown at the moment, together with any details
regarding the adjunction map.  We do, however, note that the latter
map is nonzero.  The composite then gives a nonzero map
\begin{equation}\label{e:star2}
\IndBG(w\cdot\lambda+p\tau)\longrightarrow\Tlammu\IndBG(w\cdot\mu+p\tau)
\end{equation}

``Another'' map with the same domain and target objects is obtained,
as in \ref{claim},  by applying $pr_\lambda\IndBG(-)$ to the
sequence $0\longrightarrow M\longrightarrow L^*\otimes w\cdot\mu$
there, but with $w^\prime \cdot \mu = w\cdot \mu + p\tau$ playing
the role of $w\cdot \mu$. \underline{We will call the resulting map
the ``Jantzen map"} Jan$_Y^{w\cdot\lambda}$ for $Y$ the $B$-module
$p\tau = k(p\tau)$. ( We will shortly generalize this notation.) To
discuss $Jan_Y^{w\cdot\lambda}$ and (\ref{e:star2}) in a parallel
way, \underline{denote the latter as} $Adj_Y^{w\cdot\lambda}$ (for
the same $Y$). Then the discussion at the end of
 \ref{claim} gives, using $w^\prime$ in place of $w$ there:\\

\begin{prop}\label{prop:app1} The maps $Adj_Y^{w\cdot\lambda}$ and $Jan_Y^{w\cdot\lambda}$
differ by at most a nonzero scalar multiple, for $Y=p\tau,$ when
$p\tau , ~ w\cdot \lambda +p\tau, ~ w\cdot \mu + p\tau,$ and $ws\cdot p\tau$ are dominant,
and $\tau$ with the root lattice.
\end{prop}

Now let ${\mathcal Y}^{w\cdot\lambda}$ denote the full subcategory of $B$-modules $p\tau\otimes \vupperone$ with $p\tau$ as in the proposition and $\vupperone$ a finite dimensional Frobenius twisted $G$-module with tall weights of $V$ in the root lattice.  Also write $\vupperone$ for it's restriction $Y^{[1]}|_B$ depending on context.\\

\underline{We will usually abbreviate} ${\mathcal Y}:={\mathcal Y}^{w\cdot \lambda}.$ \\

 Our next goal is to extend the maps $Adj_Y^{w\cdot\lambda}, ~ Jan_Y^{w\cdot\lambda}$ to all $y \in {\mathcal Y}$ and regard then as natural transformation $Adj^{w\cdot\lambda}=\{Adj^{w\cdot\lambda}_Y\}_{Y\in {\mathcal Y}},$
$Jan^{w\cdot\lambda}=\{Jan^{w\cdot\lambda}_Y\}_{Y\in {\mathcal Y}}$
$$Adj^{w\cdot\lambda}:\IndBG(w\cdot\lambda\otimes -)\longrightarrow\Tmulam\IndBG(w\cdot\mu\otimes-) ~~~~{\rm and}$$
$$Jan^{w\cdot\lambda}:\IndBG(w\cdot\lambda\otimes -)\longrightarrow\Tmulam\IndBG(w\cdot\mu\otimes-) $$

These functors and natural transformations will then automatically extend to $add\,\mathcal{Y},$ the additive full subcategory of $B$-mod consisting of all finite direct sums of objects in ${\mathcal Y}$.  Notice that all the $K^n$ from the previous lemma belong to $add\, {\mathcal Y}$, so that the (to be demonstrated) naturality will result in two maps of complexes\\
$$\IndBG(w\cdot\lambda\otimes K^\bullet)\longrightarrow\Tmulam\IndBG(w\cdot\mu\otimes K^\bullet)$$ resulting in maps
$$\RIndBG(w\cdot\lambda)\longrightarrow \Tmulam\RIndBG(w\cdot\mu)$$
to which we want to compare.  We will return to this point after achieving the goal above.  \\

We treat first the Jantzen maps.\\

\underline{ The Jantzen maps $Jan_Y^{w\cdot \lambda}(y\in {\mathcal Y})$ and their naturality}\\

Recall the short exact sequence \\
$$0\longrightarrow M \longrightarrow L^*\otimes w\cdot \mu \longrightarrow M^\prime \longrightarrow 0$$
in \ref{claim}.  Let $Y=p\tau\otimes \vupperone\in {\mathcal Y}.$
Tensor on the right with $Y$ and apply $\RIndBG(-)$ to get a
distinguished triangle

$$\cdots\lr \RIndBG(M\otimes Y) \lr \RIndBG(L^*\otimes w\cdot\mu \otimes Y)\lr \RIndBG(M^\prime\otimes Y)\lr\cdots$$

The middle term naturally identifies with $L^*\otimes
\RIndBG(w\cdot\mu\otimes\gamma)$ through the ``generalized tensor
identity" (discussed in this paper in Remark \ref{r:Kempf}(ii)).
Note $L^*\otimes \RIndBG(w\cdot\mu\otimes\gamma) \cong L^*\otimes
\IndBG(w\cdot\mu\otimes\gamma)$, by the construction of ${\mathcal
Y}.$  As discussed earlier in this appendix , $M$ has one weight
$\nu\in W_{aff}\cdot \lambda$, namely $\nu=w\cdot\lambda$, appearing
with multiplicity 1.  Also note that $\nu+\eta$ is in the same (dot
action) affine Weyl group orbit as $\nu$ for any weight $\nu$ adn
weight $\eta$ of ${\mathcal Y}$. Consequently, $pr_\lambda
\RIndBG(M\otimes \gamma)\cong \RIndBG(w\cdot\lambda\otimes Y)\cong
\IndBG(w\cdot\lambda\otimes Y)$. A specific construction of an
isomorphism may be given from any full flag of $B$-submodules of $M$
with one dimensional sections.  If such a flag is fixed, we obtain
an isomorphism natural in $Y$ of $Y\in {\mathcal Y}$. Similar
remarks apply for $M^\prime$ and isomorphism  $pr_\lambda
\RIndBG(M^\prime\otimes Y)\cong \IndBG(ws\cdot \lambda\otimes Y)$.

 As a consequence of the discussion above, we have exact
sequences, natural in $Y\in {\mathcal Y}$
%\begin{tcolorbox}
$$ 0\lr\IndBG(w\cdot\lambda\otimes Y)\lr\Tmulam\IndBG(w\cdot\mu\otimes Y)\lr\IndBG(ws\cdot\lambda\otimes Y)\lr 0$$
~~We define the map on the left (ignoring the obvious zero map) to
be $Jan_Y^{w\cdot\lambda}.$
%\end{tcolorbox}

\smallskip
  We summarize some of its main properties (in addition to the above exact sequence).

\begin{prop}\label{prop:app2} (i) The maps $Jan_Y^{w\cdot \lambda}, ~ y\in {\mathcal Y}$, collectively define a natural transformation of functors.
$$\IndBG (w\cdot\lambda\otimes -) \lr \Tmulam\IndBG(w\cdot\mu\otimes-)$$
on the category ${\mathcal Y}$ (whose morphisms are $B$-maps).

 (ii) For any fixed $Y=p\tau \otimes \vupperone$, there is a commutative diagram with ``obvious" vertical isomorphisms, natural in $V$

 $$\begin{CD}
 \IndBG(w\cdot\lambda\otimes \tau \otimes \vupperone ) @>Jan_Y^{w\cdot\lambda}>> \Tmulam\IndBG(w\cdot\mu\otimes p\tau \otimes \vupperone\\
 @VV\cong V     @VV\cong V\\
  \IndBG(w\cdot\lambda\otimes p \tau)\otimes \vupperone @>Jan_{p\tau}^{w\cdot \lambda} \otimes {\small V^{[1]}} >> (\Tmulam(w\cdot\mu\otimes p \tau))\otimes \vupperone
  \end{CD}$$
 \end{prop}
 \begin{proof}
 Part(i) has been proved already.  For part (ii), it is enough to check the  commutativity after identifying the right hand terms with $pr_\lambda \IndBG(L^*\otimes w \cdot \mu \otimes p\tau \otimes \vupperone)$ and $pr_\lambda \IndBG(L^*\otimes w\cdot \mu \otimes p\tau)\otimes\vupperone$, respectively.  The top row in this revised diagram may be obtained by applying $pr_\lambda \RIndBG(-)$ to the inclusion $M\otimes \gamma \lr L^*\otimes w\cdot \mu \otimes \gamma$, by construction.  Similarly the bottom row ay be obtained by applying $pr_\lambda \RIndBG(-)$ to the inclusion $M\otimes p\tau \le L^*\otimes w\cdot \mu$, then tensoring on the right with $\vupperone$.
 Now, naturality of the generalized tensor identity gives commutativity of the closed rectangle in the diagram below.

 $$\begin{CD}
 \IndBG(w\cdot\lambda \otimes p\tau \otimes \vupperone) @>\cong>> pr_\lambda \RIndBG(M\otimes Y) @>>>
 pr_\lambda \RIndBG(L^*\otimes w\cdot \mu \otimes p \tau \otimes \vupperone)\\
 @. @VV\cong V @VV\cong V\\ @. pr_\lambda(\RIndBG(M\otimes p\tau)\otimes\vupperone) @>>> pr_\lambda(\RIndBG(L^*\otimes w\cdot\mu\otimes p\tau)\otimes\vupperone)\\
 @. @VV\cong V @VV\cong V\\
 \IndBG(w\cdot\lambda\otimes p \tau )\otimes\vupperone @>\cong>> pr_\lambda \RIndBG(M\otimes p \tau)\otimes \vupperone @. pr_\lambda \RIndBG(L^*\otimes w \cdot \mu \otimes p \tau)\otimes\vupperone
 \end{CD}$$

 The identity and its naturality may also be used to complete the open rectangle on the left
 to a commutative rectangle\footnote{It is is carried out by using a $B$-module flag of $M\otimes p\tau$ and applying
 naturality to the various inclusion and factor maps involved.}, using the ``obvious" tensor identity isomorphism for a vertical map.
 Finally, all the ``$\RIndBG$" symbols  in the diagram may be replaced with ``$\Ind_B^G,$" and the bottom row completed to make a lower-right commutative rectangle.\\

 The bottom row then agrees with that of the revised diagram.
 That is, its composition gives the composition of $Jan _{p\tau}^{w\cdot \lambda}\otimes
 \vupperone$ and the identification $(\Tmulam \IndBG(w\cdot\mu\otimes p\tau)\otimes
 \vupperone\cong pr_\lambda \IndBG(L^*\otimes w\cdot\mu\otimes p \tau )\otimes\vupperone.$
 We have now shown that both top and bottom rows of the now completed and commutative outer
 rectangle agree with those of the revised version of the diagram in (ii).
 The left hand columns also agree, and the ``obvious" isomorphism on the right in the outer rectangle define,
 through composition, an obvious isomorphism in the ``revised'' diagram, making the latter commutative.
 In the original diagram in ii), the composition is
$$\Tmulam\IndBG(w\cdot \mu \otimes p \tau \otimes \vupperone)\cong
\Tmulam(\IndBG(w\cdot\mu \otimes p \tau)\otimes
\vupperone)\cong(\Tmulam\IndBG(w\cdot\mu\otimes
p\tau))\otimes\vupperone$$
which may be taken as the definition of
the right hand column ``obvious'' isomorphism in the original
diagram in ii).  The latter diagram then becomes commutative, and
the proposition is proved.
 \end{proof}

 This completes our treatment of $Jan^{w\cdot \lambda}$.  Before turning to $Adj^{w\cdot \lambda},$
 we discuss some isomorphisms $\Tlammu\IndBG(w\cdot\lambda\otimes Y)\stackrel{\sim}{\longrightarrow}
 \IndBG(w\cdot\mu\otimes Y)$ and
 $\Tmulam \Tlammu \IndBG(w\cdot \lambda \otimes Y)\stackrel{\sim}{\lr} \Tmulam\IndBG(w\cdot\mu \otimes Y),~ Y\in
 {\mathcal Y}$, which  enter into the definition and discussion of $Adj^{w\cdot \lambda}$.
 \underline{We will call the first isomorphism above} $Iso_\gamma^{w\cdot \lambda}$, \underline{and the second},
 $TIso_Y^{w\cdot \lambda}(=\Tmulam\circ Iso_Y^{w\cdot \lambda}).$\\

 The isomorphism $Iso_Y^{w\cdot \lambda}$ is obtained in a similar spirit to our construction above of the
 isomorphism $pr_\lambda \RIndBG (M \otimes \gamma) \cong \IndBG ( w\cdot \lambda \otimes Y)$, except we
 apply $pr_\mu \RIndBG$ to $L \otimes\,-$. The module $L $ has only one weight, which we called $\lambda_{\ell}$
 at the beginning of this appendix (in the discussion of the ``to a wall" isomorphism), with
 the property that $\lambda_{\ell}+w\cdot \lambda$ belongs to $W_{aff}\cdot \mu$.
 The same is true if any weight of $Y$ is added to $w\cdot \lambda$.
 We have $w\cdot \mu = \lambda_\ell + w\cdot \lambda, $ and so $ w \cdot\mu \otimes Y \cong
 \lambda _\ell \otimes w \cdot \lambda \otimes Y.$ The weight $\lambda_\ell $ appears with  multiplicity one in $L$, so

 $$pr_\mu \RIndBG(L \otimes w\cdot \lambda \otimes Y) \cong \RIndBG(\lambda_\ell \otimes w \cdot \lambda \otimes Y)\cong \RIndBG(w \cdot \mu \otimes Y) \cong \IndBG (w\cdot \mu\otimes Y)$$

 The first isomorphism can be constructed by using  any $B$-flag of $L$ with $\lambda_\ell$ as a section and applying $pr_\mu \RIndBG$ to the various sub and factor modules associated to the flag terms.  If we fix the flag and procedure, the first isomorphism becomes natural in $Y\in {\mathcal Y}$.  The other isomorphisms obviously are natural in $Y$, as are the isomorphisms
 $$\Tlammu\IndBG(w\cdot \lambda \otimes Y) = pr_\mu ( L\otimes \IndBG (w\cdot \lambda \otimes Y)) \cong pr_mu(\IndBG(L\otimes w \cdot \lambda \otimes Y))$$
 $$ {\rm and} ~~ pr_mu (\IndBG(L\otimes w\cdot \lambda \otimes Y))\cong pr_\mu (\RIndBG (L \otimes w \cdot \lambda \otimes Y)).$$

 This latter isomorphism arises from the vanishing of $(R^n\IndBG)(L\otimes w \cdot \lambda \otimes Y) =0$ for $n>0$ ( a consequence of our construction of ${\mathcal Y}$ and the generalized tensor identity).\\

The composition of all these isomorphisms (in an evident order) is defined to be\\
$$Iso_Y^{w\cdot \lambda}:\Tlammu\IndBG(w\cdot \lambda \otimes Y) \stackrel{\sim}{\lr}\IndBG(w\cdot \mu \otimes Y)$$ \\

The construction shows it is natural in  $Y\in {\mathcal Y}$ as is $TIso_Y^{w\cdot \lambda}:=\Tmulam \circ Iso_Y^{w\cdot \lambda}.$ This gives part (i) of the following proposition.\\

\begin{prop}\label{prop:app3}
(i) The maps $Iso_Y^{w\cdot\lambda}$ and $TIso_Y^{w\cdot \lambda} ~
(Y\in {\mathcal Y})$ collectively define natural isomorphisms of
functors on the category $\mathcal{Y}$

$$\Tlammu\Ind(w\cdot \lambda \otimes -) \lr \IndBG ( w \cdot \mu \otimes -)~~~{\rm and }~~~ \Tmulam\Tlammu\Ind (w \cdot \lambda \otimes -) \lr \Tmulam \IndBG (w\cdot \mu  \otimes -)$$

(ii) For any fixed $Y =p\tau \otimes \vupperone \in {\mathcal Y},$
these are commutative diagrams, with ``obvious'' vertical
isomorphisms, natural in $V$.

$$\xymatrix{
\Tlammu\Ind(w\cdot \lambda \otimes p \tau \otimes \vupperone) \ar[rrr]^{Iso_Y^{w\cdot \lambda}}\ar[d]^\cong &&&\IndBG (w\cdot\mu \otimes p\tau \otimes \vupperone)\ar[d]^\cong\\
    \Tlammu\Ind(w\cdot \lambda \otimes p \tau) \otimes \vupperone \ar[rrr]^{Iso_Y^{w\cdot \lambda}\otimes
        {\tiny \vupperone}} &&&\IndBG (w\cdot\mu \otimes p\tau) \otimes \vupperone
        }$$
        \center{and}
    $$\xymatrix{
\Tmulam\Tlammu\Ind(w\cdot \lambda \otimes p \tau \otimes \vupperone) \ar[rrr]^{TIso_Y^{w\cdot \lambda}}\ar[d]^\cong &&&\Tmulam\Ind (w\cdot\mu \otimes p\tau \otimes \vupperone)\ar[d]^\cong\\
    \Tmulam\Tlammu\Ind(w\cdot \lambda \otimes p \tau) \otimes \vupperone \ar[rrr]^{TIso_Y^{w\cdot \lambda}\otimes {\tiny \vupperone}} &&&\Tmulam\Ind (w\cdot\mu \otimes p\tau) \otimes \vupperone
        }$$
        \end{prop}
 \begin{proof}

 Part (i) already has been proved. Next, note that a commutative lower diagram  in
 (ii) can be obtained by first applying $\Tmulam$ to a commutative upper diagram,
 then using the natural isomorphism
 $\Tmulam(-\otimes \vupperone)\cong\Tmulam(-\otimes \vupperone$
 on the lower row of the upper diagram.  the reader may convince him/her self
 that the entire procedure preserves the ``obvious" property of the vertical maps!\\

 Thus, it is suffice to treat the upper diagram  in (ii).
 The first thing to do here is to note the ``obvious"isomorphism
 $\Tlammu\Ind(w\cdot \lambda \otimes p\tau \otimes
 \vupperone)\cong\Tlammu(\IndBG (w\cdot \lambda \otimes p\tau)\otimes\vupperone)\cong
 \Tlammu\Ind(w\cdot\lambda\otimes p \tau)\otimes \vupperone.$
 This gives the first column in the upper diagram.  The isomorphism may be regarded as the (by now ``obvious'') process of ``pulling out'' $\vupperone$'', from inductions of tensor products, block projection or translation functors, or some combination of these operators.  The row of isomorphism above requires two steps to fully ``pullout'' $\vupperone$.  If we continue with the several steps required to define $Iso_Y^{w\cdot \lambda}$, we see at every step along the way there is an opportunity to ``pull out'' $\vupperone$.   This
 gives a series of possibly commutative diagrams, written below in top to
 bottom order.
%Though physical labels are not provided, we think of these commutative diagrams as labeled $(1), (2), \cdots, (6)$.

$$\xymatrix{
\Tlammu \IndBG _\mu(w\cdot \lambda \otimes p\tau \otimes \vupperone) \ar[r]^{\cong}\ar@{=}[d]_{{\bf(1)}\hspace{1.3in}} &  \Tlammu \IndBG(w\cdot \lambda \otimes p \tau)\otimes \vupperone \ar@{=}[d] \\
 pr_\mu (L\otimes \IndBG(w\cdot\lambda\otimes p\tau \otimes \vupperone) \ar[d]_{\bf{(2)}\hspace{1.3in}}^{\cong}\ar[r]^{\cong} & pr_\mu (L\otimes
\IndBG(w\cdot\lambda\otimes p\tau)) \otimes \vupperone \ar[d]^{\cong}\\
 pr_\mu ( \IndBG(L\otimes w\cdot\lambda\otimes p\tau \otimes
\vupperone) \ar[d]_{\bf{(3)}\hspace{1.3in}}^{\cong}\ar[r]^{\cong} &  pr_\mu ( \IndBG(L\otimes w\cdot\lambda\otimes p\tau)) \otimes \vupperone \ar[d]^{\cong}\\
 pr_\mu  (\RIndBG(L\otimes w\cdot\lambda\otimes p\tau \otimes \vupperone)) \ar[d]_{\bf{(4)}\hspace{1.3in}}^{\cong}\ar[r]^{\cong} &  pr_\mu (\RIndBG(L\otimes w\cdot\lambda\otimes p\tau)) \otimes \vupperone \ar[d]^{\cong} \\
 \RIndBG (\lambda _ \ell \otimes w\cdot \lambda \otimes p\tau \otimes \vupperone) \ar[d]_{\bf{(5)}\hspace{1.3in}}^{\cong}\ar[r]^{\cong} & \RIndBG (\lambda _ \ell \otimes w\cdot \lambda \otimes p\tau) \otimes \vupperone  \ar[d]^{\cong} \\
 \RIndBG (w\cdot\mu\otimes p\tau \otimes \vupperone) \ar[d]_{\bf{(6)}\hspace{1.3in}}^{\cong}\ar[r]^{\cong} &   \RIndBG( w\cdot \mu \otimes p\tau) \otimes \vupperone \ar[d]^{\cong} \\
 \RIndBG( w\cdot \mu \otimes p\tau \otimes \vupperone) \ar[r]^{\cong} &
 \IndBG( w\cdot \mu \otimes p\tau) \otimes \vupperone
}
$$

Diagram (1) commutes as a matter of notation, identifying the
functor $\Tlammu(-)$ with $pr_\lambda L(-)$,  when applied to the
``block"' associated to $W_{aff}\cdot \lambda$ (the top row
isomorphism has already been given in $\Tlammu$ notation.)  For
diagram (2), note that the isomorphism in its top row may formally
be applied to the same row with $pr_\mu$ removed. Next, remove
$pr_\mu$ from the bottom row of $(2)$ also.  If we can get
commutativity in the resulting rectangle

$$\xymatrix{
L\otimes \IndBG(w\cdot \lambda \otimes p\tau \otimes \vupperone) \ar[r]^\cong \ar[d]^\cong & L\otimes \IndBG(w\cdot \lambda \otimes p\tau) \otimes \vupperone \ar[d]^\cong\\
\IndBG(L\otimes w\cdot \lambda \otimes p\tau \otimes \vupperone)
\ar[r]^\cong & \IndBG(L\otimes w\cdot \lambda \otimes p\tau) \otimes
\vupperone }$$

We get it for (2) by applying $pr_\mu$ to the whole diagram, then pulling out
$\vupperone$ on the right.\\

To get commutativity of the rectangle itself note that all four of
its corners are induced modules, by the tensor identity, isomorphism
to the lower left hand corner. Using the formalism in
\cite[I.3.4]{J}, every induced module $\IndBG M$ (where $M$ here
just denotes some $B-$module) is equipped with a $B-$module map
$\epsilon_M:\IndBG M \lr M$.   If $N$ is $G-$module the tensor
identity isomorphism $\IndBG(M\otimes N)\cong (\IndBG M) \otimes N$
composed with $\epsilon_M\otimes N$ gives $\epsilon _{M\otimes N}$.
(This can be extracted from the discussion in \cite[I.3.6]{J}.) This
implies that the usual universal property of induction (see \cite[I.
Prop. 3.46]{J}) applies directly to $(\IndBG M)\otimes N)$ using
$\epsilon _M \otimes N)$ in the role of a ``counit" adjunction
(terminology of Wikipedia. Note that the target of $\epsilon _M
\otimes N \otimes M \times N$. We will just call $\epsilon _M
\otimes N$ the \underline{evaluation map} associated with $(\IndBG
M)\otimes N$ and $M\otimes N$ the associated \underline{evaluation
target}. Returning to the rectangle above, all four of its corners,
all obtained from the induced module in the lower left corner by
various applications of the tensor identity, have the same target
(up to associativity isomorphisms).  Consequently, all maps in the
rectangle may be viewed as ``induced'' from the identity map on
their (common) target. (This certainly true in the case of an
individual application of the tensor identity, from which it to
follows in the case of the tensor identity applied within a tensor
product of several factors.  All individual maps in the rectangle
arise this way, and the property of b eing ``induced'' from the
identity map on their (common) target.  (This is certainly true in
the case of an individual application of the tensor identity, from
which it follows in the case of the tensor identity applied within a
tensor product of several factors.  All individual maps in the
rectangle arise this way, and the property of being ``induced'' from
the identity map on a common target carries over to composition.)
It follows now that the rectangle above is (thoroughly) commutative, as in (2). \\

Commutativity of $(3)$ is easily seen to hold, since the
derived functor $\RIndBG$ on both sides is applied to objects
acyclic for $\IndBG$  (i.e., their `` higher derived functors
vanish''). The meaning of the vertical maps in $(4)$ was discussed in
the construction of $Iso_Y^{w\cdot \lambda}.$ The horizontal maps in
the bottom row as obtained from the generalized tensor identity. The
top row map is obtained similarly, after pulling $V^[1]$ out of the block projection. Both column constructions may be viewed, before applying
$pr_\mu$, as
arising from maps $L\lr L^\prime \longleftarrow \lambda_\ell $ where
$L^\prime $ is a quotient of $L$, tensoring with  $w\cdot \lambda
\otimes p \tau$ or $w \cdot \lambda \otimes p $ and applying
$\RIndBG(-)$ or $\RIndBG(-)\otimes \vupperone$.  Since the
generalized tensor identity may be regarded as a natural
transformation of functors. We obtain a commutative diagram rising
from maps $L\lr L^\prime \longleftarrow \lambda_\ell $ where
$L^\prime $ is a quotient of $L$, tensoring with  $w\cdot \lambda
\otimes p \tau$ or $w \cdot \lambda \otimes p $ and applying
$\RIndBG(-)$ or $\RIndBG(-)\otimes \vupperone$. Since the
generalized tensor identity may be regarded as a natural
transformation of functors.  We obtain a commutative diagram
$$
\xymatrix{
\RIndBG(L\otimes w\cdot \lambda \otimes p \tau \otimes \vupperone) \ar[d] &  \cong & \ar[d]\RIndBG(L\otimes w\cdot \lambda \otimes p\tau)\otimes \vupperone \\
\RIndBG(L^\prime\otimes w\cdot \lambda \otimes p \tau \otimes \vupperone) & \cong & \RIndBG(L^\prime \otimes w\cdot \lambda \otimes p\tau)\otimes \vupperone \\
\RIndBG(\lambda_\ell\otimes w\cdot \lambda \otimes p \tau \otimes
\vupperone) \ar[u] & \cong & \ar[u] \RIndBG(\lambda_\ell\otimes
w\cdot \lambda \otimes p\tau)\otimes \vupperone }
$$
Now apply $pr_\mu$ and pullout $\vupperone$ on the right.  All column isomorphism become the column isomorphisms in $(4),$ equating the objects in the bottom row of the latter with the same objects with $pr_\mu$ applied.  The top row of $(4)$ has been discussed and agrees with the top row of the diagram above, after the modification.  \\

The commutativity of diagram $(5)$ is easy, since $\lambda _\ell$ is
equal to $w\cdot \mu$ as a weight.  It is interesting to note that
the construction of $Iso_Y^{w\cdot \lambda}$ must fix an isomorphism
between the 1-dimensional section $\lambda _\ell$ of $L$, and the
abstract 1-dimensional weight space $w\cdot \mu$.

\textbf{Observation:} ~\underline{In this sense}
~$Iso_Y^{w\cdot\lambda}~$ \underline{can be modified by a nonzero
scalar multiplication}, and remain a version obtained by the ``same"
construction (still a natural transformation defined on the category
$\mathcal {Y}$).  ~\underline{Such a
modification carries over to} $TIso_Y^{w\cdot \lambda}$.\\
%colorbox construction had difficulties showing up in .dvi

The pull-out operation in $(6)$ is the generalized tensor identity
in both rows, except that $\RIndBG(-)$ may be identified with
$\Ind^G_B(-)$ on the bottom row.  The columns just reflect this
identification and the diagram is clearly commutative.

Note that the bottom row in $(6)$ is precisely the right hand column
in the upper diagram in $(ii)$.  The right hand column of the
iterated rectangles $(1), (2), \cdots, (6)$ is by construction,
$Iso_{p\tau}^{w\cdot \lambda} \otimes \vupperone.$ Thus, the outer
perimeter of $(1), (2), \cdots, (6)$ gives a commutative version of
the upper diagram in  $(ii),$ after turning the perimeter diagram on
its side (left hand side put on top).  This completes the proof of
the proposition.
\end{proof}

\underline{ The maps $Adj_Y^{w\cdot \lambda}(Y \in {\mathcal Y})$ and their naturality}\\

The map $Adj_Y^{w\cdot \lambda}:\IndBG(w\cdot \lambda\otimes Y) \lr
\Tmulam(w\cdot \mu \otimes Y)$ is defined as the composition of
adjunction ${adj_X}: X \lr \Tmulam\Tlammu  X,$ with $X=
\IndBG(w\cdot \lambda \otimes Y), ~ Y \in {\mathcal Y},$ and the
previously discussed isomorphism
$$TIso_Y^{w\cdot \lambda}:\Tmulam\Tlammu \IndBG (w\cdot \lambda \otimes Y)\stackrel{\sim}{\lr} \Tmulam \IndBG (w\mu \cdot \lambda \otimes Y)$$

The adjunction map $\widetilde{adj}_X, ~  X \in G-$mod, is defined
as the image of $1\in \Hom_G(\Tlammu X, \Tlammu X)$ under a natural
isomorphism  $\Hom_G(\Tlammu -, -) \stackrel{\sim}{\lr}
\Hom_G(-,\Tmulam -),$ with $X, \Tlammu X$ as the variables. (Thus
$\Hom_G(\Tlammu X, \Tlammu X) \cong \Hom_G(X, \Tmulam\Tlammu X))$.
In Appendix A, we have given a thorough discussion of $\widetilde
{adj}_X,$ constructing it from a similar adjunction map $adj_X$
associated to the adjoint functors $L\otimes -$ and $L^* \otimes -$.
We will quote from Appendix A to prove the proposition below.  \\

\begin{prop}\label{prop:app4}
(i) The maps $\widetilde{adj}_X ~ (X \in G- mod)$ collectively give the adjunction natural transformation from the identity functor to $\Tmulam\Tmulam$.\\

(ii) For any $V$ in $G-$mod with all weights in the root lattice,
there is a commutative diagram.
$$\xymatrix{
X\otimes \vupperone \ar[rr]^{\widetilde{adj}_{X \otimes {\tiny \vupperone}}~~} \ar[d]^{=} && \Tmulam\Tlammu (X\otimes \vupperone) \ar[d]^{\cong} \\
X\otimes\vupperone \ar[rr]^{\widetilde{adj}_{X \otimes
{\tiny \vupperone}}~~} && (\Tmulam\Tlammu X) \otimes \vupperone }$$

The right hand column is morphism becomes equality, if both
right hand objects are viewed as a submodules of $ L^*\otimes L
\otimes X \otimes \vupperone$.
\end{prop}
\begin{proof}
Part (i) has already been discussed.  Note the obvious fact that
adjunctions are natural transformations. (A proof is written down in
footnote 11 of this paper, noted in the proof of Proposition
\ref{adjprop}.)

Part (ii) follows from Corollary \ref{adjcor}.\\
\end{proof}

We can now give parallel properties of $Adj^{w\cdot\lambda}$, meant
especially to mirror Proposition \ref{prop:app2} for $Jan
^{w\cdot\lambda}$.
\begin{prop}\label{prop:app5}
(i) The maps $Adj_Y^{w\cdot \lambda}, Y\in \mathcal{Y}$, collectively
define a natural transformation of functors
$$\Ind_B^G(w\cdot \lambda \otimes -)\longrightarrow \Tmulam\Ind_B^G(w\cdot \mu\otimes -)$$ on the category ${\mathcal Y}$\\
(ii) For any fixed $Y=p\tau \otimes \vupperone$ in $\mathcal{Y}$, there is a commutative diagram with ``obvious'' vertical isomorphisms, natural in $V$:\\
$$\xymatrix{
\IndBG(w\cdot \lambda \otimes p\tau \otimes \vupperone )\ar[d]^\cong  \ar[rr]^{{Adj}_Y^{w\cdot\lambda}}&& \Tmulam (w\cdot \mu \otimes p\tau \otimes \vupperone) \ar[d]^\cong\\
\IndBG(w\cdot \lambda \otimes p\tau) \otimes \vupperone
\ar[rr]^{{Adj}_{p\tau}^{w\cdot\lambda }\otimes {\tiny \vupperone}}&&
\Tmulam(w\cdot \mu \otimes p\tau) \otimes \vupperone }$$
\end{prop}

\begin{proof}
Part (i) follows from the  definition $Adj_Y^{w\cdot \lambda}= TIso_Y^{w\cdot \lambda} \circ \widetilde{adj}_X$
with $X= \IndBG (w\cdot \lambda \otimes Y)$.\\

For part (ii), note that the left hand column of the lower diagram in
Proposition 3(ii) may be written as a composition
$$\Tmulam\Tlammu \IndBG(w\cdot \lambda \otimes p\tau \otimes \vupperone) \cong \Tmulam\Tlammu(\IndBG (w\cdot\lambda\otimes p \tau)\otimes \vupperone)\cong \Tmulam\Tlammu\Ind(w\cdot\lambda\otimes p \tau)\otimes \vupperone.$$
The second isomorphism is the right hand column of Proposition
\ref{prop:app4}(ii). To deal with the first isomorphism, we need the
following Lemma.

\begin{lem}\label{lem:app2} For $Y=p\tau \otimes \vupperone \in {\mathcal Y}$, there is a commutative diagram
$$\xymatrix{
\IndBG(w\cdot \lambda \otimes p\tau \otimes \vupperone )\ar[d]^\cong  \ar[rrr]^{\widetilde{adj}_{\IndBG(w\cdot\lambda \otimes p \tau \otimes {\tiny \vupperone})}~~}&&& \Tmulam\Tlammu \IndBG(w\cdot \lambda \otimes p\tau \otimes \vupperone) \ar[d]^\cong\\
\IndBG(w\cdot \lambda \otimes p\tau) \otimes \vupperone
\ar[rrr]^{\widetilde{adj}_{\IndBG(w\cdot\lambda \otimes p \tau)
\otimes {\tiny \vupperone}}~~}&&& \Tmulam\Tlammu (\IndBG(w\cdot \lambda)
\otimes p\tau \otimes \vupperone) }$$
\end{lem}

\begin{proof}
This is just naturality of $\widetilde{adj}_X$ with respect to $X\in
G-$mod, applied to the $G-$module isomorphism comprising the left
column.
\end{proof}

We mow return to the proof of Proposition \ref{prop:app5}. Put the
diagram of  Lemma \ref{lem:app2} on top of that of Proposition \ref{prop:app4}
(ii), taking $X=\Ind_B^G(w\cdot \lambda \otimes Y)$.
%(DELETE THIS PICTURE, IF IT IS NOT NEEDED)
%$$\xymatrix{
%\Ind_B^G(w\cdot\lambda\otimes Y)\otimes \vupperone \ar[rrr]^{\widetilde{adj}_{\Ind_B^G(w\cdot\lambda\otimes Y) \otimes \vupperone}~~~} \ar[d]^{=} & &&\Tmulam\Tlammu (\Ind_B^G(w\cdot\lambda\otimes Y)\otimes \vupperone) \ar[d]^{\cong} \\
%\Ind_B^G(w\cdot\lambda\otimes Y)\otimes\vupperone
%\ar[rrr]^{\widetilde{adj}_{\Ind_B^G(w\cdot\lambda\otimes Y) \otimes
%\vupperone}~~~} & &&(\Tmulam\Tlammu \Ind_B^G(w\cdot\lambda\otimes
%Y)) \otimes \vupperone }$$
In this case, the top horizontal edge of the diagram in Proposition
\ref{prop:app4}(ii) agrees with the bottom edge of the lemma, so
concatenation makes sense.  Moreover the right hand column of the
concatenated diagram agrees with the left hand column of the lower
diagram in Proposition \ref{prop:app3}(ii). (The latter column was
discussed above as composition of isomorphisms.)  This allows a
further concatenation, a giving commutative diagram

%i.e the top horizontal edge of the diagram in Proposition 4(ii) agrees with the bottom edge of the Lemma, so concatenation makes sense.  Moreover the right hand column of the concatenated diagram agrees with the left hand column of the lower diagram in Proposition 3(ii) (The latter column was discussed above as composition of isomorphisms.) This allows a further concatenation, giving a commutative diagram.

$$\xymatrix{
\IndBG(w\cdot\lambda \otimes p\tau \otimes \vupperone)\ar[d]^{\cong}\ar[r] & \Tmulam\Tlammu\Ind(w\cdot \lambda \otimes p\tau \otimes \vupperone) \ar[r] \ar[d]^{\cong} & \Tmulam\Ind(w\cdot\mu \otimes p \tau \otimes \vupperone) \ar[dd]^\cong\\
\IndBG(w\cdot\lambda \otimes p\tau )\otimes \vupperone \ar[r] \ar[d]^{=} &\Tmulam\Tlammu(\IndBG(w\cdot \lambda \otimes p\tau) \otimes \vupperone)  \ar[d]^{\cong} & \\
\IndBG(w\cdot\lambda \otimes p\tau )\otimes \vupperone \ar[r]
&\Tmulam\Tlammu\Ind(w\cdot \lambda \otimes p\tau) \otimes \vupperone
\ar[r] & \Tmulam\Ind(w\cdot\mu \otimes p \tau \otimes \vupperone). }
$$

The top row maps are $\widetilde{adj}_{\IndBG(w\cdot \lambda \otimes
p \tau \otimes \vupperone}$, on the left and $TIso_Y^{w\cdot
\lambda}, $ with $Y=p\tau \otimes \vupperone, $ on the right.  In
the same order, the bottom row maps are $\widetilde
{adj}_{\IndBG(w\cdot \lambda \otimes p \tau)}\otimes \vupperone$ and
$TIso_{p\tau}^{w\cdot \lambda} \otimes \vupperone$.  The composition
of the two top row maps is $Adj_Y^{w\cdot \lambda}$, and the
composition of the bottom row maps is $Adj_{p\tau}^{w\cdot
\lambda}\otimes \vupperone$.  The commutativity of the outer
rectangle now gives the desired commutativity of the diagram in (ii)
\end{proof}

We are just about ready for a vast improvement to Proposition
\ref{prop:app1}.
First we need an easy but key observation.\\

\begin{lem}\label{lem:app3}  For any $Y = p \tau \otimes \vupperone$ in ${\mathcal Y},$ the
left column ``obvious'' isomorphisms in the diagrams of Proposition
\ref{prop:app2}(ii) and Proposition \ref{prop:app5}(ii) are equal as
are the right column.
\end{lem}
\begin{proof}
On the left, both isomorphism s just pull out $\vupperone$ using the
tensor identity.  A similar isomorphism is used on the right (in
both cases) except it is also necessary to commute $\Tmulam(-)$ and
$(-)\otimes \vupperone.$
\end{proof}

We can now prove a main theorem.\\

\begin{thm} \label{thm:app1} There is a nonzero scalar $c \in k$ such that $Adj_Y^{w\cdot \lambda} = cJan_Y^{w\cdot \lambda}$ for all $y\in {\mathcal Y}.$
\end{thm}
\begin{proof}
Proposition \ref{prop:app1} gives a nonzero scalar that works in the especial case $Y=p\tau$.  The constant it gives is possibly dependent on $Y$ and call it $c(p\tau).$\\

Propositions \ref{prop:app2}(ii) and \ref{prop:app5}(iii), together
with  Lemma \ref{lem:app3}, show we claim, an equality
$$ adj_Y^{w\cdot \lambda}=c(p\tau)Jan_{p\tau}^{w\cdot \lambda}.$$
whenever $Y=p\tau \otimes \vupperone \in {\mathcal Y}.$ To prove
this equality, note $Adj ^{w\cdot \lambda} = c(p\tau)
Jan_{p\tau}^{w\cdot \lambda}.$ Tensor on the right with $\vupperone$
to get
$$ Adj_{p\tau}^{w\cdot \lambda} \otimes \vupperone = c(p\tau) Jan_{p\tau}^{w\cdot \lambda}\otimes \vupperone$$
Precompose each side  with the downward left column isomorphism common to the diagrams in Propositions \ref{prop:app2}(ii) and \ref{prop:app5}(ii), and postcompose with the upward right column isomorphism.  This gives the claimed equality (reading it off from the two commutative diagrams and the previous equality.)\\

It remains to prove $c(p\tau)=c(p\tau ^\prime)$ whenever $p\tau, ~ p\tau ^\prime$ are $1$-dimensional objects in ${\mathcal Y}$. Note that $p(\tau + \tau ^\prime)$ with necessarily, also belong to ${\mathcal Y}.$ We will show $c(p\tau)=c(p(\tau+\tau ^\prime)).$  This is enough, since the equality $c(\tau^ \prime)=c(p(\tau ^ \prime +\tau))$ will follow by re-choosing notations.\\

Let $V = \IndBG (\tau^\prime).$ The $p \tau \otimes \vupperone$ and $p\tau \otimes p\tau^\prime = p(\tau + \tau^\prime)$ both belong to ${Y}$.  There a $B-$ module map (``evaluation'') from $V=\IndBG(\tau^\prime)$ onto $\tau ^\prime$, and we twist it by the Frobenius to get a surjective map $\vupperone \lr p\tau^\prime.$ Tensor on the left with $p\tau$ to get a surjective map $p\tau \otimes \vupperone \lr p(\tau + \tau ^\prime).$  We will call this map $\phi$.  It is a map in the category ${\mathcal Y}$. \\

 We have commutative diagram by naturality of $Jan ^{w\cdot \lambda}$ with respect to ${\mathcal Y}$
$$\xymatrix{
\IndBG(w\cdot \lambda \otimes p(\tau+\tau^\prime)) \ar[rr] ^{Jan_{p(\tau+\tau^\prime)}^{w\cdot \lambda}} && \Tmulam\Ind(w\cdot \mu \otimes p(\tau+\tau^\prime))\\
\IndBG(w\cdot \lambda \otimes p \tau \otimes
V^{[1]})\ar[u]^{\IndBG(w\cdot\lambda {\tiny {\tiny \otimes} } \phi)}
\ar[rr]^{Jan_{p\tau\otimes {\tiny \vupperone}}^{w\cdot \lambda}} &&
\Tmulam\Ind(w\cdot \mu \otimes p \tau \otimes \vupperone)
\ar[u]_{{\small {\rm T}}_\mu^\lambda\Ind(w\cdot\mu\otimes \phi)} }$$

The left column map is nonzero, since its composition with the
evaluation map $\IndBG(w\cdot \lambda \otimes p(\tau+\tau^\prime))
\lr w\cdot \lambda \otimes p(\tau+\tau^\prime)$ is nonzero.  The top
row map is injective, an instance of the left part of the short
exact sequence displayed above Proposition \ref{prop:app2}.  Hence
the composition
$$Jan_{p(\tau+\tau^\prime)}^{w\cdot \lambda} \circ \IndBG(w\cdot \lambda \otimes \phi)$$
 is not zero.

However, there is a similar diagram, identical to the above, but with
``Jan'' replaced by ``Adj''.  We have
$$ Adj_{p(\tau+\tau^\prime)}^{w\cdot \lambda} \circ \IndBG(w\cdot \lambda \otimes \phi)= c(p(\tau+\tau^\prime)) Jan_{p(\tau+\tau^\prime)}^{w\cdot \lambda}\circ \IndBG(w\cdot \lambda \otimes \phi).$$

On the other hand, commutativity of the ``Adj'' diagram equates the
left expression with $\Tmulam \IndBG(w\cdot\mu \otimes \phi)\circ
Adj _{p\tau \otimes {\tiny \vupperone}}^{w\cdot \lambda} = \Tmulam
\IndBG(w\cdot\mu \otimes \phi)\circ c(p\tau) Jan_{p\tau \otimes
{\tiny \vupperone}}^{w\cdot \lambda}.$ Now bring out the scalar $c(\phi)$
and apply commutativity in the ``Jan'' diagram. The right expression
becomes $c(p\tau) Jan_{p(\tau+\tau^\prime)}^{w\cdot \lambda}\IndBG
(w\cdot \lambda \otimes \phi).$ We have shown
$$c(p(\tau+\tau^\prime)) Jan_{p(\tau+\tau^\prime)}^{w\cdot \lambda}\circ \IndBG(w\cdot \lambda \otimes \phi)=
c(p\tau) Jan_{p(\tau+\tau^\prime)}^{w\cdot \lambda}\circ
\IndBG(w\cdot \lambda \otimes \phi)$$
But we have shown that the map appearing to the right of both $c(p(\tau+\tau^\prime))$ and $c(p\tau)$ above is not zero.  So, the only way the equality can occur is to have $c(p(\tau+\tau^\prime)=c(p\tau).$  \\

This completes the proof of the theorem.
\end{proof}

Though not entirely necessary, it simplifies notation if we modify
$TIso_Y^{w\cdot \lambda}$ by a scalar by changing its construction,
as per the observation in the proof of Proposition \ref{prop:app3}.
We do this so that the newly constructed $TIso_Y^{w\cdot \lambda}$
is the old one
multiplied by $c$ above (uniformly in $Y\in {\mathcal Y}$).  \underline{This allows us to have actual} \underline{equalities.}\\

$Adj_Y^{w\cdot \lambda}=Jan_Y^{w\cdot \lambda}$ for all $Y\in {\mathcal Y}$.  This is also a good time to
enlarge the domain of the natural transformation $Adj^{w\cdot \lambda}$ and $Jan_Y^{w\cdot \lambda}$
from ${\mathcal Y}$
to  $add\,\mathcal{Y}$ ( which has objects  direct sums of objects of $\mathcal{Y}$.  Similar domain enlargements can be made for $Iso^{w\cdot \lambda}$ and  $TIso^{w\cdot \lambda}.$
The domain $add\,\mathcal{Y}$ \underline{ can also be extended to}
\underline{complexes of objects} from $add\,\mathcal{Y}$, such as the complex
$K^\bullet$ discussed as Lemma \ref{resolvetrivialmodule}.\\

In the next proposition, we use this formalism to illuminate the
triangle (*) above the claim in \ref{claim}, rewritten below
$$(*) ~~~\cdots \lr \RIndBG(w\cdot \lambda)\lr \Tmulam \RIndBG(w\cdot \lambda) \lr \RIndBG(ws\cdot \lambda ) \lr \cdots $$

\begin{prop}\label{prop:app6}
With a suitable choice of the complex $K^\bullet$ in Lemma \ref{resolvetrivialmodule}, there
is an exact sequence of complexes
$$0\lr\Ind(w\cdot\lambda \otimes K^\bullet) \lr \Tmulam \IndBG(w\cdot \mu \otimes K^\bullet) \lr \IndBG(ws\cdot \lambda \otimes K^\bullet) \lr 0$$
which represent (*) at the level of complexes (in the sense that its
sequence of these objects and two maps - ignoring the $0'$s -
identifies, after passing to the bounded derived category, with the
displayed portion above of (*)).

The left hand complex map is $Jan_{K^\bullet}^{w\cdot \lambda},$ the
extension of $Jan ^{w\cdot \lambda}$ to complexes of $add \, {\mathcal
Y}$ objects, in the particular case of the complex $K^\bullet$.
\end{prop}

\begin{proof}

Choose $K^\bullet$ so that each $K^n$ is a direct sum of terms
$Y=p\tau \otimes \vupperone$ in ${\mathcal Y}$ with also $\nu +
p\tau$ dominant for each weight $\nu$ of $L^* \otimes  w\cdot \mu.$
The construction of (*) is from the exact sequence at the top of \ref{claim}

$$0\lr M \lr L^*\otimes w\cdot \mu \lr M^\prime \lr 0$$
by applying $pr_\lambda \RIndBG(-).$ With our choice of $K^\bullet$, $\RIndBG(-)$ applied to each term is the same as $\IndBG(-\otimes K^\bullet).$ Also, $pr_\lambda$  applied  to $ \IndBG (\nu \otimes K^\bullet)$ for $\nu \neq w\cdot \lambda$ in $M$ or $\nu \neq ws\cdot \lambda$ in $M^\prime$, is the zero complex.\\

Thus, $pr_\lambda \IndBG(-\otimes K^\bullet),$ applied to the
displayed sequence, gives an exact sequence
$$ 0 \lr \IndBG (w\cdot \lambda \otimes K^\bullet) \lr pr_\lambda\Ind(L^* \otimes w\cdot \mu \otimes K^\bullet) \lr \IndBG(ws \cdot \lambda \otimes K^\bullet) \lr 0$$
The middle term identifies with $pr_\lambda(L^*\otimes \IndBG(w\cdot
\mu \otimes K^\bullet))$ via the tensor identity.
 Such an identification must also be made in the construction of (*), though with $\RIndBG(w\cdot \mu)$ replacing $\IndBG (w\cdot \mu \otimes K^\bullet)$.  Note also $pr_\lambda (L^*\otimes \IndBG(w\cdot \mu\otimes K^\bullet))= pr_\lambda (L^*\otimes pr_\mu\IndBG(w\cdot \mu\otimes K^\bullet))  = \Tmulam \IndBG(w\cdot \mu\otimes K^\bullet), $ and a similar equality holds for $\RIndBG(w\cdot \mu)$.\\

Following each step above gives the identifications claimed in the
proposition.  An alternate argument could be made by replacing
$K^\bullet$ in the short exact sequence displayed above with a
complex of injective $B-$modules ( a resolution of $k = k(0)$ also).
This gives a semisplit short exact sequence of complexes.  Its
three term sequence then automatically becomes a three term sequence
in a distinguished triangle, upon passing to the derived category.
In more detail, let $K^\bullet \lr I^\bullet$ be an isomorphism of
complexes, with $I^\bullet$ a $B-$module  injective resolution of
$k=k(0)$.  There are resulting commutative diagram of map of
$G-$module complexes

{\tiny
$$\xymatrix{
0 \ar[r]& p\tau\Ind(M \otimes I^\bullet) \ar[r] & pr_\lambda \IndBG(L^* \otimes w\cdot \mu \otimes I^\bullet) \ar[rr]\ar[dr] & & p\tau\Ind(M\otimes  I^\bullet)\ar[r]&0\\
& && pr_\lambda(L^* \otimes \IndBG(w\cdot \mu \otimes I^\bullet & &\\
0 \ar[r]& p\tau\Ind(M \otimes K^\bullet) \ar[r]\ar[uu] & pr_\lambda \IndBG(L^* \otimes w\cdot \mu \otimes K^\bullet) \ar[rr]\ar[dr]\ar[uu] & & p\tau\Ind(M\otimes  K^\bullet)\ar[uu]\ar[r]&0\\
& && pr_\lambda(L^* \otimes \IndBG(w\cdot \mu \otimes I^\bullet
\ar[uu] & & }$$}
The skew maps are isomorphism of complexes, and the vertical maps are all quasi-isomorphisms.  The top row is an exact sequence of complexes of injective objects, is therefore semi-split, and therefore becomes part of a distinguished triangle (ignoring the zeros and zero maps) at the derived category label ($D^+$ or $D^b$ here).  There are a few other commutative squares of quasi isomorphism need to give a complete picture of the identification claimed in the proposition, but we leave then to the reader (who should have the idea by now). On the left, for example, diagrams must be added handling the
identifications $\IndBG(w\cdot \lambda \otimes K^\bullet)\cong pr_\lambda \IndBG(M\otimes K^\bullet).$ ( This will require two rectangles, associated with the location of $w\cdot \lambda$ as a section of $M$.)\\

%The analogous and simpler identifications $\IndBG(w\cdot \lambda \otimes K^\bullet)\cong pr_\lambda \IndBG(M\otimes K^\bullet).$ ( This will require two rectangles, associated with the location of $w\cdot \lambda$ as a section of $M$.)\\

The analogous and simpler identification $\IndBG(w\cdot \lambda \otimes Y)\cong pr_\lambda \IndBG(M\otimes Y)$ is part of the Jantzen map $Jan_Y^{w\cdot \lambda},$ discussed above Proposition \ref{prop:app2}, partly using $``\RIndBG''$ notation.
 %on (on pp 5-6 of these notes MAKE proper reference).
Sticking to the ``$\IndBG$'' notation, the map $Jan_Y^{w\cdot \lambda}$ is the composite of $\IndBG (w\cdot \lambda \otimes Y)\cong pr_\lambda(L^* \otimes \IndBG (w\cdot \mu \otimes Y)$, with $pr_{\lambda}\IndBG (M\otimes Y)\to  pr_{\lambda}I\IndBG(L^*\otimes w\cdot\mu\otimes Y)\cong pr_{\lambda}(L^*\otimes Ind^G_B(w\cdot\mu\otimes Y))$, the latter equal to $pr_\lambda(L^*\otimes pr_\mu \IndBG(w\cdot \mu \otimes Y) \cong \Tmulam\Ind(w\cdot \mu \otimes Y)$.\\

Passing from ${\mathcal Y}$ to $ add \, {\mathcal Y}$ and then to
complexes of $add \, {\mathcal Y}$ objects, and, following the
pathway above, we find that $Jan_{K^\bullet}^{w\cdot \lambda}$ is
the composition of the identification $\IndBG (w\cdot \lambda
\otimes K^\bullet)\cong pr_\lambda \IndBG (M\otimes K^\bullet),$ the
bottom left map of the above diagram followed by the adjacent skew
map, and finally the identification
$$pr_\lambda(L^*\otimes \IndBG (w\cdot \mu \otimes K^\bullet)= pr_\lambda (L^* \otimes pr_\mu (\IndBG(w\cdot \mu \otimes K^\bullet) =\Tmulam \IndBG (w \cdot \mu \otimes K^\bullet).$$
This is, altogether, precisely the map
 $$\IndBG(w\cdot \lambda \otimes K^\bullet) \lr \Tmulam \IndBG(w\cdot \mu \otimes K^\bullet)$$
 which our construction, in completed form, gives for the left hand amp in the exact sequences displayed in the proposition.  So that map is $Jan_{K^\bullet}^{w\cdot \lambda}$, and our proof of the proposition is complete.
\end{proof}

\textbf{The general case of the claim of \ref{claim}}\\

Recall that we noted in \ref{claim} that the middle term of the
distinguished triangle (*) was isomorphic to $\Tmulam\Tlammu
\RIndBG(w\cdot \lambda).$ Noting further that this resulted in a map
$\RIndBG(w\cdot \lambda)\lr\Tmulam\Tlammu \RIndBG(w\cdot\mu)$, we
claimed that this map was adjunction, at least up to scalar
multiple. \\

To some extent, this requires first interpreting what isomorphism of
$\Tmulam\Tlammu \RIndBG(w\cdot \lambda)$ with the middle term
$\Tmulam \RIndBG(w\cdot \mu)$ of (*) was intended.  The top of
subsection \ref{claim} indicates that an isomorphism $\Tlammu
\RIndBG(w\cdot \lambda)\cong \RIndBG(w\cdot \mu)$ be used. We will
follow that framework. Represent $\RIndBG(w \cdot \lambda)$ by
$\IndBG(w\cdot \lambda \otimes K^\bullet)$ and $\RIndBG(w \cdot
\mu)$ by $\IndBG (w \cdot \mu \otimes K^\bullet),$ as in Proposition
\ref{prop:app6} and its proof.  Then an isomorphism $\Tlammu
\IndBG(w\cdot \lambda \otimes K^\bullet)\cong \IndBG (w \cdot \mu
\otimes K^\bullet)$ is s given by $Iso_{K^\bullet}^{w\cdot
\lambda}.$  Composing with $\Tmulam$ on both sides gives the
isomorphism of complexes

 $$TIso_{K^\bullet}^{w\cdot
\lambda}:\Tmulam\Tlammu\Ind(w\cdot\lambda \otimes
K^\bullet)\cong\Tmulam\Ind(w\cdot \mu \otimes K^\bullet).$$

Note that the two sides represent $\Tmulam\Tlammu \RIndBG(w\cdot \lambda)$ and
$\Tmulam \RIndBG(w\cdot \mu)$, respectively.\\

 Finally, we need a way
to represent the adjunction map from $\RIndBG(w\cdot \lambda),$
represented as $\IndBG(w\cdot \lambda \otimes K^\bullet),$ to
$\Tmulam\Tlammu \RIndBG(w\cdot \lambda)$, which we have represented
as $\Tmulam\Tlammu \IndBG(w\cdot \lambda \otimes K^\bullet).$ For this, we
use the map of complexes

$$ \widetilde{adj}_{\IndBG(w\cdot \lambda \otimes K^\bullet)}:\IndBG(w\cdot \lambda \otimes K^\bullet)
\lr \Tmulam\Tlammu\Ind(w\cdot \lambda \otimes K^\bullet).$$ This map
just applies adjunction to the $G-$modules in each degree of the
complex $\IndBG(w\cdot \lambda \otimes K^\bullet).$  This results in a map
of complexes, by naturality of adjunction.  In particular, passing
to the derived category level ($D^+$ or $D^b$), the map of complexes
$\tilde{adj}_{\IndBG(w\cdot \lambda \otimes K^\bullet}$ induces
``adjunction" as a map

$$\RIndBG(w\cdot \lambda) \lr \Tmulam\Tlammu
\RIndBG(w\cdot \lambda)$$. \\

The result below is a corollary to
Theorem \ref{thm:app1} and Proposition \ref{prop:app6}.

\begin{cor}  With the isomorphism $\Tmulam \RIndBG(w\cdot \mu) \xrightarrow{\sim} \Tmulam\Tlammu \RIndBG(w\cdot \lambda)$ taken as inverse to the isomorphism induced by $TIso_{K^\bullet}^{w\cdot \lambda}$ discussed above,
\underline{the claim of subsection \ref{claim} is correct.}  More
precisely, we have a commutative diagram

$$\xymatrix{
\RIndBG(w\cdot \lambda)\ar[r]\ar[d]^{=} & \Tmulam\Tlammu \RIndBG(w\cdot \lambda)\ar[d]^{\cong}\\
\RIndBG(w\cdot \lambda) \ar[r]& \Tmulam \RIndBG(w\cdot \mu) }$$ with
the top row adjunction and the bottom row from $(*)$. The right
column isomorphism is as above.
\end{cor}
\begin{proof}
We have $Adj_Y^{w\cdot \lambda} = TIso_Y^{w\cdot \lambda}\circ
\widetilde{adj}_{\IndBG(w\cdot \lambda \otimes Y)}$ for all $Y \in
{\mathcal Y},$ by definition.  Passing to $adj \,{\mathcal Y}$ and
complexes such as $K^\bullet,$ we have the similar identity

$$Adj^{w\cdot \lambda}_{K^{\bullet}}=TIso^{w \cdot \lambda}_{K^{\bullet}}\circ
\widetilde{adj}_{\Ind(w.\cdot \lambda \otimes K^\bullet)}$$
By Theorem \ref{thm:app1} we also have
$$Adj_{K^\bullet}^{w\cdot \lambda} = c Jan_{K^\bullet}^{w\cdot \lambda}$$
where $c$ is  a non zero scalar.   If we adjust $TIso ^{w\cdot
\lambda}$ as per the (bold-faced),  observation in the proof of Proposition \ref{prop:app3},
 we may assume $c=1$.  Assume that this adjustment is in force. Then we have a
commutative diagram at complexes
$$\xymatrix{
\IndBG(w\cdot\lambda \otimes K^\bullet) \ar[rrr]^{\widetilde{adj}_{\IndBG(w\cdot \lambda \otimes K^\bullet)}~~~~}\ar[d]^= &&& \Tmulam\Tlammu\Ind(w\cdot \lambda \otimes K^\bullet) \ar[d]^{TIso_{K^\bullet}^{w\cdot \lambda}}\\
\IndBG(w\cdot \lambda \otimes K^\bullet) \ar[rrr]^{Jan^{w\cdot \lambda}_
{K^\bullet}}&& &\Tmulam\Ind(w\cdot \mu \otimes K^\bullet) }$$

By Proposition \ref{prop:app3} the bottom arrow represents the map $\RIndBG(w\cdot
\lambda)\lr\Tmulam \RIndBG(w\cdot \mu)$ in (*).  We discussed, above
the statement of the corollary, the fact that the top row becomes
the adjunction map $\RIndBG(w\cdot \lambda)\lr \Tmulam\Tlammu
\RIndBG(w\cdot \lambda)$ upon passing to the derived category.  The
right column map is, as discussed in the statement of the corollary the
right column derived derived category, isomorphism in the corollary's
diagram. Altogether, the commutativity of the diagram of complexes
above gives the commutativity of the diagram in the corollary.  This
completes its proof.
\end{proof}
\begin{rem} Without the adjustment observed in the proof of Proposition \ref{prop:app3},
we only get commutativity up to a scalar, as allowed in the claim.
\end{rem}

\begin{rem}\label{r:app2} {\bf The quantum case}
The same changes of $\otimes$ to $\otimes^{op}$ observed in Appendix
A need to be made in this appendix, in the quantum case. In addition
it is necessary to replace the references to \cite{J} in the proof
of Proposition \ref{prop:app3} with references to Remarks
\ref{r:Kempf}(d),(e). Remark \ref{r:Kempf}(f) helps explain the
differences in the formalism of these remarks (which also could be
used in the algebraic groups case) with that of \cite{J}. Recall
also that Remark \ref{r:Kempf}(d) provides both right and left
generalized tensor identities in the quantum case, heavily used in
the arguments above (for example, in the proofs of Propositions
\ref{prop:app2} and \ref{prop:app3}.  With these changes and
observations, all of the proofs and results in this appendix carry
over to the quantum case.\\

  In particular, the claim of subsection
\ref{claim} holds in both the algebraic groups and quantum cases. As
argued below the claim, this completes the proof of Lemma
\ref{wallcrossingfunctors}.
\end{rem}

\section
%This is the copy of the appendix on Lemma 9.10.5 written by Len and typed up by Julie
{\bf Appendix C} The purpose of this appendix is to supplement, and,
indeed, to ``fix," the statement and proof of \cite[Lem.
9.10.5]{ABG}, as a service to the reader. This is all in
characteristic $0$, and not part of the induction theorem (except in
the way of application), but it is important to \cite{ABG} as a
whole and to the discussion in \cite[ftn.13]{PS2} concerning
Koszulity in the quantum case. The proof given in \cite{ABG} of the
lemma, corrected for misprints and issues with the induction theorem
proof, still seemed inaccurate to us, but we found it could be fixed
using an algebraic result from \cite{PS2}. The latter result is
nontrivial, but relatively elementary, not using the Lusztig quantum
conjecture. This seems desirable, so that \cite{ABG} could have the
latter conjecture, in the $\ell>h$ case, as a corollary.

Our notation in this appendix largely follows \cite{ABG}, with two
major changes: The formula for the ``dot" notation $\bullet$ is
replaced by that for the standard  ``dot" action $\cdot$ in \cite{J}
and subsection \ref{subsubsec:dot} above. Thus, the new formula
reads, for $w$ in the Weyl group or affine Weyl group, and $\lambda
\in {\mathbb{X}},$
$$w\bullet \lambda = w(\lambda + \rho) - \rho.$$
Also, we will use Borel subalgebras ${\sf B}$ whose associated roots
are negative, rather than positive. With these two changes,
\cite[Lem. 3.5.1]{ABG}, which we will use below, is correct as
stated. (It actually was not, before, even for $w=1$.) The statement
of the quantum induction theorem \cite[Thm. 3.5.5]{ABG}, which we
will also use, is unchanged. Finally, the change from positive to
negative Borels on the quantum side is deliberately not repeated on
the Langlands dual side, when choosing Borel objects there
(associated to Grassmanian varieties).

 At the point the result \cite[Lem.
9.10.5]{ABG} in question is introduced in \cite{ABG} the authors
have established an equivalence of derived categories
\cite[(9.10.1)]{ABG}
$$
\gamma: D^b\,{\sf block(U)} \rarrow D^b \ \perv\mbox{(Gr)},
$$
and it is desired to show the functor $\gamma$ induces an
equivalence from {\sf block(U)} to $\perv$(Gr). To this end,
categories $D^b_{\leq \lambda}$ {\sf block(U)} and $D^p_{\leq
\lambda} \perv$ are introduced ``for each $\lambda \in
\mathbb{Y}^{++}$." This appears to be a misprint, repeated several
times on \cite[p.~668]{ABG}, and the definition of $D^b_{\leq
\lambda}$ {\sf block(U)} is incorrect with any choice of $\lambda$.
%(It would lead to the latter category containing infinitely many
%irreducible modules of {\sf block(U)}.)
Instead, these categories
should be introduced for each $\lambda \in \mathbb{Y}$, with $\mu
\leq \lambda$ interpreted to mean $\tilde{\mu} \uparrow
\tilde{\lambda}$, where $\tilde{\mu} \in W {\scriptstyle \, \bullet
\,} \ell \mu$ is, in our notation here, the (unique) dominant
weight, and $\tilde{\lambda}$ is defined similarly.

The order $\uparrow$ above is that discussed in
\cite[II,6.1-6.11]{J}; note that $p$ there is allowed to be any
positive integer. The order $\uparrow$ should replace the order
$\preceq$ in \cite[(3.4.5)]{ABG}. The followiing equivalence (in the
case $\ell \geq h$)) follows from the more general theorem
\cite[Thm. 9.6]{PS2}, which also has a formulation for $\ell < h$
$$
y {\scriptstyle \, \bullet \,} 0 \uparrow w {\scriptstyle \, \bullet
\,} 0 \mbox{ iff } y \leq' w,
$$
whenever $y {\scriptstyle \, \bullet \,} 0$, $w {\scriptstyle \,
\bullet \,} 0$ are dominant and $y,w \in W_{\rm aff}$. The order
$\leq '$ is the Bruhat--Chevalley order with respect to the dominant
standard chamber fundamental reflections. This equivalence seems
essential to correct the lemma.

We shall use $\leq$ for the Bruhat--Chevalley order with respect to
the antidominant standard chamber. Thus, for $y, w \in W_{\rm aff}$,
$y \leq w$ iff $w_0 y w_0 \leq' w_0ww_0$, with $w_0$ the long word
in $W$. When $y,w$ are in $W$, $y \leq w$ means the same as $y \leq'
w$. When $\nu, \mu \in \mathbb{Y}$, $\nu y \leq \mu w$ iff $(-\nu)y
\leq ' (-\mu)w$. (The classical root system has an automorphism $y
\mapsto -w_0(y)$, preserving positive roots.) Notice there is an
implicit change from $\leq '$ to $\leq$ in the proof of
\cite[Cor.~8.3.2]{ABG}. (This occurs in the assertion that
``$\lambda w^{-1}$ is minimal in the right coset $\lambda W \leq
W_{\rm aff}^n$." The hypothesis of Cor.~8.3.2(ii) gives minimality
of $w\lambda$ with respect to $\leq '$.  Passing to inverses gives
minimality of $(-\lambda)w^{-1}$ with respect to $\leq '$. Now it is
necessary, it seems, to use $\leq$ to get minimality of $\lambda
w^{-1}$.)

The definition of $D^b_{\leq \lambda}\perv$ is correct as given in
\cite[p.~668]{ABG} provided it is allowed that $\lambda \in
\mathbb{Y}$. Similarly, $\lambda$ should be taken in $\mathbb{Y}$ in
the statement of the lemma, which we provide below, with this
change. Note the direction of $\Upsilon$ is reverse to the
equivalence in part (i).

\begin{lem} (\cite[9.10.5]{ABG}) For any $\lambda \in \mathbb{Y}$, we have
\begin{itemize}
\item[(i)] the functor $\Upsilon$ induces an equivalence
$$
D^b_{\leq \lambda}\perv \stackrel{\sim}{\longrightarrow} D^b_{\leq
\lambda} \ {\sf block(U)}.
$$
Moreover,
\item[(ii)]
the induced functor
$$
D^b_{\leq \lambda}\perv/D^b_{<\lambda}\perv \longrightarrow
D^b_{\leq \lambda}\ {\sf block(U)}/D^b_{<\lambda}\ {\sf block(U)}
$$
sends the class of $IC_\lambda$ to the class of $L_\lambda$.
\end{itemize}
\end{lem}

\begin{proof}
We follow \cite{ABG}, taking into account the changes above, and
also the misprints noted in \cite[p.~675]{ABG}. There are also some
inaccuracies in \cite[Cor.~8.2.4, Cor.~8.3.2]{ABG} which we address
as they arise.

We know for any $\lambda \in \mathbb{Y}$, the functor $\Upsilon$
sends, by construction, the object $R\mbox{ Ind}^{\sf U}_{\sf
B}(\ell\lambda)$ to $\overline{\mathcal{W}}_\lambda$. Fix $\lambda
\in \mathbb{Y}$, and let $w \in W$ be the element with $w\lambda
{\scriptstyle \, \bullet \,} 0 = w {\scriptstyle \, \bullet \,} \ell
\lambda$ dominant. Then, by \cite[Lem. 3.5.1]{ABG}--see our Remark
\ref{r:Kempf}(a) and subsections \ref{s:induction}, \ref{s:derived}
for additional details--we have that $R^{\ell(w)}\mbox{Ind}^{\sf
U}_{\sf B}(\ell\lambda)$ has $L_\lambda$ as a composition factor
with multiplicity one, and all other composition factors $L_\mu$ of
$R^{\ell(w)}\mbox{Ind}^{\sf U}_{\sf B}(\ell\lambda)$,
 or any composition factor $L_\mu$ of $R^j\mbox{Ind}^{\sf U}_{\sf B}(\ell\lambda)$ with $j \neq \ell(w)$, satisfy $y\mu {\scriptstyle \, \bullet \,} 0 \uparrow w \lambda {\scriptstyle \, \bullet \,} 0$, with $y \in W_{\rm aff}$, $y\mu {\scriptstyle \, \bullet \,} 0$ dominant, and $y\mu {\scriptstyle \, \bullet \,} 0 \neq w \lambda {\scriptstyle \, \bullet \,} 0$. For any such $\mu$, we have $y \mu <' w\lambda$, as
noted above, and $\mu y^{-1} < \lambda w^{-1}$. As observed in the
proof of \cite[Cor.~8.3.2]{ABG}, this implies supp
$\overline{\mathcal{W}}_\mu = \overline{\rm Gr}_\mu \subseteq
\overline{\rm Gr}_\lambda = \mbox{supp }
\overline{\mathcal{W}}_\lambda$, and the inclusion is proper. Thus,
$\Upsilon$ takes $D^b_{\leq \lambda}\ {\sf block(U)}$ into
$D^b_{\leq \lambda}\perv$. By induction (on, say, the height of the
dominant weight $w {\scriptstyle \, \bullet \,} \ell \lambda$), we
may assume $\Upsilon$ induces an equivalence of triangulated
categories, when $\leq \lambda$ is replaced by $<\lambda$. (Here
$\mu < \lambda$ is taken to mean $y \mu{\scriptstyle \, \bullet \,}
0 \uparrow w \lambda {\scriptstyle \, \bullet \,} 0$ as above, with
$y \in W$ and $y \mu {\scriptstyle \, \bullet \,} 0 \uparrow w
\lambda {\scriptstyle \, \bullet \,} 0$, $y \mu {\scriptstyle \,
\bullet \,} 0 \neq w \lambda {\scriptstyle \, \bullet \,} 0$.
Equivalently, $\overline{\rm Gr}_\mu$ is properly contained in
$\overline{\rm Gr}_\lambda$, as we have seen.) By \cite[$\S$9.1,
p.~655]{ABG} $\perv$(Gr) is generated by simple objects IC$_\nu$,
$\nu \in \mathbb{Y}$, each with support contained in $\overline{\rm
Gr}_\nu$.

We take this opportunity to mention there are errors of sign in
\cite[Cor.~8.2.4, Cor.~8.3.2]{ABG}, where
$\mathbb{C}_{yw}[-\mbox{dim }$ $\mathcal{B}_{yw}]$ should be
replaced by $\mathbb{C}_y[\mbox{dim }$ $\mathcal{B}_{yw}]$ and
$\mathbb{C}_\lambda[-\mbox{dim Gr}_\lambda - \ell(w)]$ should be
replaced by  $\mathbb{C}_\lambda [\mbox{dim
Gr}-\ell(w)]$.\footnote{This latter change needs to be made both in
the statement and proof of \cite[Cor.~8.3.2]{ABG}. The error is in
ignoring the shift in degree that can accompany direct images with
proper maps.}

With these changes, the conclusion of \cite[Cor.~8.3.2(ii)]{ABG}
shows $\overline{\mathcal{W}}_\lambda$ and $IC_\lambda[-\ell(w)]$
have the same restriction to Gr$_\lambda$ (from $\overline{\rm
Gr}_\lambda$). This shows, together with the fact that $\Upsilon$
induces an equivalence $D^b_{< \lambda}\ {\sf block(U)} \rarrow
D^b_{<\lambda}\perv$, which we obtained above by induction, that the
strict image under $\Upsilon$ of $D^b_{\leq \lambda}\ {\sf
block(U)}$ is $D^b_{\leq \lambda}\perv$. Since we already know
$\Upsilon$ provides an equivalence $D^b\ {\sf block(U)} \rarrow
D^b\perv$, it follows now that it induces one between $D^b_{\leq
\lambda}\ {\sf block(U)}$ and $D^b_{\leq \lambda}\perv$. This proves
(i).

Moreover, we also get (ii), since we have shown that
$\Upsilon(R\mbox{ Ind}^{\sf U}_{\sf B}(\ell \lambda)) =
\overline{\mathcal{W}}_\lambda$ is $IC_\lambda[-\ell(w)]$ in the
quotient category $D^b_{\leq \lambda}\perv/D^b_{<\lambda}$ $\perv$.
Our remarks on the composition factors of the cohomology groups of
the preimage $R\mbox{ Ind}^{\sf U}_{\sf B}(\ell \lambda)$ of
$\overline{\mathcal{W}}_\lambda$ show that the image in the quotient
category $D^b_{\leq \lambda}\ {\sf block(U)}/D^b_{<\lambda}$  ${\sf
block(U)}$ is $L_\lambda[-\ell(w)]$. This completes the proof of
(ii) and the lemma.\end{proof}

%\newpage
\section{Acknowledgements} The authors would like to remember Julie Riddleburger, now deceased, for her humor and patience, and for her help in preparing the manuscript.
The first and second author would like to thank  Brian Parshall,
Leonard Scott,  and the University of Virginia for their gracious
support in so many ways over the years; in particular, this project
was begun while the first author was on sabbatical and the second
author was a postdoc at the University of Virginia. Part of the
works were completed at Kalamazoo, thanks to Terrell and WMU for
such wonderful arrangements.

\end{document}